\documentclass[10pt]{amsart}
  \usepackage{amsmath,amssymb,amsthm,amsfonts,graphicx}
  \graphicspath{{../../drawings/}}
  \usepackage[dvipsnames]{xcolor}

  \definecolor{limegreen}{rgb}{0.196,0.804,0.196}
  \definecolor{darkgreen}{rgb}{0.0,0.5,0.0}
  \definecolor{darkbluegreen}{rgb}{0,0.3,0.6}
  \definecolor{badgerred}{rgb}{0.715,0.004,0.004}
  \usepackage[font=small,labelfont={bf,sf}, textfont={sf},margin=0em]{caption}
  \usepackage{url,hyperref}
  \hypersetup{colorlinks,
  citecolor=badgerred,
  filecolor=black,
  linkcolor=blue,
  urlcolor=darkgreen,
  bookmarksopen=false,
  pdftex}

  \newcommand{\cE}{\mathcal{E}}

  \newcommand{\bv}{{\bar v}}
  
  \newcommand{\bW}{{\bar W}}

  \newcommand{\cH}{\mathcal{H}}
  \newcommand{\cD}{\mathcal{D}}
  \newcommand{\cB}{\mathcal{B}}
  \newcommand{\cC}{\mathcal{C}}
  
  \newcommand{\cL}{\mathcal{L}}
  \newcommand{\cM}{\mathcal{M}}
  
  \newcommand{\cP}{\mathcal{P}}
  \newcommand{\cO}{\mathcal{O}}
  
  \newcommand{\cS}{\mathcal{S}}

  \newcommand{\h}{\mathcal{H}}
  \newcommand{\hilb}{\mathfrak{H}}
  \newcommand{\hv}{\mathfrak{D}}
  
  \newcommand{\pr}{\mathcal{P}}      
  \newcommand{\cyl}{\mathcal{C}}     
  \newcommand{\collar}{\mathcal{K}}  
  \newcommand{\tip}{\mathcal{T}}     
  \newcommand{\pd}{\partial}

  \newcommand{\so}{{\mathfrak{so}}} 
  
  \newcommand{\R}{{\mathbb R}}
  \newcommand{\supp}{\mathop{\mathrm {supp}}}
  
  \newtheorem{theorem}{Theorem}[section]
  \newtheorem{proposition}[theorem]{Proposition}
  \newtheorem{lemma}[theorem]{Lemma}
  \newtheorem{definition}[theorem]{Definition}
  \newtheorem{prop}[theorem]{Proposition}
  \newtheorem{corollary}[theorem]{Corollary}
  
  \theoremstyle{remark}
  \newtheorem{rem}{Remark}[section]
  \newtheorem{remark}[theorem]{Remark}
  
  \newtheorem{claim}[theorem]{Claim}

  \newtheorem{step}{Step}

  \numberwithin{equation}{section}
  \numberwithin{theorem}{section}

\begin{document}
\title[Uniqueness of Ancient Ovals in MCF]
{Uniqueness of two-convex closed ancient solutions to the mean curvature flow 
} 
\author[Angenent]{Sigurd Angenent}
\address{Department of Mathematics, University of Wisconsin -- Madison}
\author[Daskalopoulos]{Panagiota Daskalopoulos}
\address{Department of Mathematics, Columbia University, New York}
\author[Sesum]{Natasa Sesum}
\address{Department of Mathematics, Rutgers University, New Jersey}

\thanks{
P.~Daskalopoulos thanks the NSF for support in DMS-1266172.
N.~Sesum thanks the NSF for support in DMS-1056387.
}

\date{\today}

\begin{abstract}
  In this paper we consider closed noncollapsed ancient solutions to the mean curvature flow ($n \ge 2$) which are uniformly two-convex.
  We prove such an ancient solution is up to translations and scaling the unique rotationally symmetric closed ancient noncollapsed solution constructed in \cite{Wh} and \cite{HH}.
\end{abstract}
\maketitle

\tableofcontents

\section{Introduction} 
In this paper we consider closed noncollapsed ancient solutions $F(\cdot,t): M^n \to \mathbb{R}^{n+1}$ to the mean curvature flow ($n \ge 2$) 
\begin{equation}
  \label{eq-mcf}
  \frac{\partial}{\partial t} F = -H\, \nu
\end{equation}
for $t\in (-\infty,0)$, where $H$ is the mean curvature of $M_t := F(M^n,t)$ and $\nu$ is the outward unit normal vector.
We know by Huisken's result \cite{Hu} that the surfaces $M_t$ will contract to a point in finite time.

The main focus of the paper is the classification of two-convex {\em closed ancient solutions}  to mean curvature flow, i.e.~solutions that are defined for $t\in (-\infty,T)$, for some $T < +\infty$. 
Ancient solutions play an important role in understanding the singularity formation in geometric flows, as such solutions  are usually obtained after performing a  blow up near points where the curvature is very large.
In fact, 
Perelman's famous work on the Ricci flow \cite{P} shows that  the high curvature regions are modeled on ancient solutions  which have nonnegative curvature and are $\kappa$-noncollapsed.
Similar results for mean curvature flow were obtained in \cite{HK}, \cite{Wh0}, \cite{Wh} assuming mean convexity and embeddedness. 

Daskalopoulos, Hamilton and Sesum previously established the complete classification of  ancient compact convex solutions  to the curve 
shortening flow in \cite{DHS1},  and   ancient compact solutions  the  Ricci flow on $S^2$  in \cite{DHS2}.   
The higher dimensional cases have remained open for  both  the mean curvature flow and  the Ricci flow. 

\smallskip 
In an important work by Xu-Jia Wang \cite{W} the author introduced the following notion of non-collapsed solutions to the MCF which is the analogue to the $\kappa$-non-collapsing condition for the Ricci flow discussed above.
In the same work Xu-Jia Wang provided a number of results regarding the asymptotic behavior of ancient solutions, as $t \to -\infty$, and he also constructed new examples of ancient MCF solutions.

\begin{definition}
  Let $K^{n+1} \subset \mathbb{R}^{n+1}$ be a smooth domain whose boundary is a mean convex hypersurface $M^n$. 
  We say that $M^n$ is  $\alpha$-noncollapsed if for every $p\in M^n$ there are balls $B_1$ and $B_2$ of radius at least $\frac{\alpha}{H(p)}$ such that  $\bar B_1 \subset K^{n+1}$ and $\bar B_2 \subset \mathbb{R}^{n+1}\setminus  Int (K^{n+1})$, and such that $B_1$ and $B_2$ are tangent to $M^n$ at the point $p$, from the interior and exterior of $K^{n+1}$, respectively (in the limiting case $H(p) \equiv 0$, this means that $K^{n+1}$ is a halfspace). 
  A smooth mean curvature flow $\{M_t\}$ is $\alpha$-noncollapsed if $M_t$ is $\alpha$-noncollapsed for every $t$.
\end{definition}

In \cite{And} Andrews showed that the $\alpha$-noncollapsedness property is preserved along mean curvature flow, namely, if the initial hypersurface is $\alpha$-noncollapsed at time $t = t_0$, then evolving hypersurfaces $M_t$ are $\alpha$-noncollapsed for all later times for which the solution exists. 
Haslhofer and Kleiner \cite{HK} showed that every closed, ancient, and $\alpha$-noncollapsed solution is necessarily convex. 

\smallskip
In recent  breakthrough  works,  Brendle and Choi  \cite{BC,BC2} gave   the  complete classification of noncompact ancient solutions to the mean curvature flow that are both strictly convex and uniformly two-convex. 
More precisely, they show that any noncompact and complete ancient solution to mean curvature flow \eqref{eq-mcf} that is strictly convex, uniformly two-convex, and noncollapsed is the Bowl soliton, up to scaling and ambient isometries. 
Recall that the Bowl soliton is the unique rotationally-symmetric, strictly convex solution to mean curvature flow that translates with unit speed.
It has the approximate shape of a paraboloid and its mean curvature is largest at the tip. 
The uniqueness of the Bowl soliton among convex and uniformly two-convex translating solitons has been proved  by Haslhofer in \cite{Ha}.

While the $\alpha$-noncollapsedness property for mean curvature flow is preserved forward in time, it is not necessarily preserved going back in time.
Indeed, Xu-Jia Wang (\cite{W}) exhibited examples of ancient compact convex mean curvature flow solutions $\{M_t \,\,\,|\,\,\, t<0\}$, that is not uniformly $\alpha$-noncollapsed for any $\alpha > 0$.
Such solutions lie in slab regions.
The methods in \cite{W} rely on the level set flow.
Recently, Bourni, Langford and Tinaglia \cite{BLT} provided a detailed construction of the Xu-Jia Wang solutions by different methods, showing also that the solution they construct is unique within the class of rotationally symmetric mean curvature flows that lie in a slab of a fixed width.
In the present paper we will not consider these ancient collapsed solutions and focus on the classification of ancient closed noncollapsed mean curvature flows.

Ancient self-similar solutions to MCF are of the form $M_t = \sqrt{T-t}\,\bar M$ for some fixed surface $\bar M$ and some ``blow-up time'' $T$.
We rewrite a general ancient solution $\{M_t : t<T\}$ as
\begin{equation}
  \label{eq-parabolic-blow-up}
  M_t = \sqrt{T-t} \, \bar M_\tau, \qquad \tau:={-\log (T-t)}.
\end{equation}

\smallskip
Haslhofer and Kleiner \cite{HK} proved that every closed ancient noncollapsed mean curvature flow with strictly positive mean curvature sweeps out the whole space. 
By Xu-Jia~Wang's result \cite{W}, it follows that in this case the backward limit as $\tau \to -\infty$ of the   type-I rescaling  $\bar M_\tau$ of the original solution $M_t$, 
defined by \eqref{eq-parabolic-blow-up},   is either a sphere or a generalized cylinder $\R^k\times S^{n-k}$ of  radius $\sqrt{2(n-k)}$. 
In \cite{ADS} we showed that if the backward limit is a  sphere then the ancient solution $\{M_t\}$ has to be a family of shrinking spheres itself.

\begin{definition}
  \label{def-oval}
  We say an ancient mean curvature flow $\{M_t : -\infty < t < T\}$ is an \emph{Ancient Oval} if it is compact, smooth, noncollapsed, and not self-similar.
\end{definition}

\begin{definition}
  We say that an ancient solution $\{M_t : -\infty < t < T\}$ is   \emph{uniformly 2-convex} if there exists a uniform constant $\beta > 0$ 
  so that 
  \begin{equation}
    \label{eqn-2convex}
    \lambda_1 + \lambda_2 \ge \beta H, \qquad \mbox{ for all} \,\, t \le t_0.
  \end{equation}
\end{definition} 

Throughout the paper we will be using the following observation: if an Ancient Oval $M_t$ is uniformly 2-convex, then by results in \cite{W}, the backward limit of its type-I parabolic blow-up must be a shrinking round cylinder $\mathbb{R}\times S^{n-1}$, with radius $\sqrt{2(n-1)}$.

Based on formal matched asymptotics, Angenent \cite{A} conjectured the existence of an Ancient Oval, that is, of an ancient solution that for $t\to 0$ collapses to a round point, but for $t\to -\infty$ becomes more and more oval in the sense that it looks like a round cylinder $\mathbb{R}\times S^{n-1}$ in the middle region, and like a rotationally symmetric translating soliton (the Bowl soliton) near the tips. 
A variant of this conjecture was proved already by White in \cite{Wh}. 
By considering convex regions of increasing eccentricity and using a limiting argument, he proved the existence of ancient flows of compact, convex sets that are not self-similar.  
Haslhofer and Hershkovits \cite{HH} carried out White's construction in more detail, including, in particular, the study of the geometry at the tips.
As a result they gave a rigorous and simple proof for the existence of an Ancient Oval.

Our main result in this paper is as follows.

\begin{theorem}
  \label{thm-main-main}
  Let $\{ M_t, \, -\infty < t < T \} $ be a uniformly 2-convex Ancient Oval.
  Then it is unique and hence must be the solution constructed by White in \cite{Wh} and later by Haslhofer and Hershkovits in \cite{HH}, up to ambient isometries, scaling and translations in time.
\end{theorem}

The proof of this theorem will follow from the results stated below. 

\begin{theorem}
  \label{thm-rot-symm}
  If $\{M_t : -\infty < t <0\}$  is an Ancient Oval which is uniformly 2-convex, then it is rotationally symmetric.
\end{theorem}

Our proof of Theorem \ref{thm-rot-symm} closely follows the arguments 
by Brendle and Choi in \cite{BC, BC2} on the uniqueness of strictly convex, noncompact, uniformly 2-convex, and noncollapsed ancient
mean curvature flow. 
It was shown in \cite{BC} that such solutions are rotationally symmetric.
Then, by analyzing the rotationally symmetric solutions, Brendle and Choi showed that such solutions agree with the Bowl soliton.

Given Theorem~\ref{thm-rot-symm}, we may assume in our proof of Theorem~\ref{thm-main-main} that any Ancient Oval $M_t$ is rotationally symmetric. 
After applying a suitable Euclidean motion we may assume that its {\em axis of symmetry is the $x_1$-axis}.
Then, $M_t$ can be represented as
\begin{equation}
  \label{eq-O1xOn-symmetry}
  M_t = \bigl\{ (x, x') \in \R\times\R^{n} : -d_1(t)<x<d_2(t), \|x'\|=U(x, t)\bigr\}
\end{equation}
for some function $\|x'\|=U(x,t)$, and from now on we will set $x:=x_1$ and $x'=(x_2, \cdots, x_{n+1})$. 
We call the points $(-d_1(t), 0)$ and $(d_2(t),0)$ \emph{the tips} of the surface.  
The function $U(x, t)$, which we call the \emph{profile} of the hypersurface $M_t$, is only defined for $x\in[-d_1(t), d_2(t)]$.
Any surface $M_t$ defined by \eqref{eq-O1xOn-symmetry} is automatically invariant under $O(n)$ acting on $\R\times\R^n$.  
Convexity of the surface $M_t$ is equivalent to  concavity of the profile $U$, i.e.~$M_t$ is convex if and only if $U_{xx}\leq0$.

A family of surfaces $M_t$ defined by $\|x'\|=U(x, t)$ evolves by mean curvature flow if and only if the profile $U(x,t)$ satisfies
\begin{equation}
  \label{eq-u-original}
  \frac{\pd U}{\pd t} = \frac{U_{xx}}{1+U_x^2} - \frac{n-1}{U}.
\end{equation}
If $M_t$ satisfies MCF, then its parabolic rescaling $\bar{M}_\tau$ defined by \eqref{eq-parabolic-blow-up} evolves by the \emph{rescaled MCF}
\[
\nu\cdot\frac{\partial \bar F}{\partial \tau} = H + \tfrac12 \bar F\cdot\nu,
\]
where $\bar F(x, \tau) = e^{\tau/2}F(x, T-e^{-\tau})$ is the parametrization of $\bar M_\tau$, and $\nu = \nu(x, t)$ is the corresponding unit normal.
Also,
\[
\bar{M}_{\tau}
= \{(y,y')\in \mathbb{R}\times \mathbb{R}^n
\mid
-\bar{d}_1(\tau) \le y \le \bar{d}_2(\tau), \,\,\, \|y'\| = u(y,\tau)\}
\]
for a profile function $u$, which is related to $U$ by
\[
U(x,t) = \sqrt{T-t}\, u(y, \tau), \qquad y=\frac x{ \sqrt{T-t}}, \quad \tau=-\log (T-t).
\]
The points $(-\bar{d}_1(\tau),0)$ and $(\bar{d}_2(\tau),0)$ are referred to as the tips of rescaled surface $\bar{M}_{\tau}$.
Equation \eqref{eq-u-original} for $U(x,t)$ is equivalent to the following equation for $u(y,\tau)$ 
\begin{equation}
  \label{eq-u}
  \frac{\pd u}{\pd \tau} = \frac{u_{yy}}{1+u_y^2} - \frac y2 \, u_y - \frac{n-1}{u}+ \frac u2.
\end{equation}

\smallskip

It follows from  the discussion above, that our most general result \ref{thm-main-main} reduces to the following  classification under the presence  of rotational symmetry. 
\begin{theorem}
  \label{thm-main}
  Let $(M_1)_t$  and $(M_2)_t$, $-\infty < t < T$  be two  $O(n)$-invariant   Ancient Ovals   with the same  axis  of symmetry (which is assumed  to be the $x_1$-axis)    whose profile functions $U_1(x,t)$ and $U_2(x,t))$ satisfy equation \eqref{eq-u-original}
  and rescaled profile functions    $u_1(y,\tau)$ and $u_2(y,\tau)$ satisfy equation \eqref{eq-u}.
  Then,  they are the same up to  translations along  the axis of symmetry (translations in $x$),  translations in time and parabolic rescaling.  
\end{theorem}

Since the asymptotics result from \cite{ADS} will play a significant role in this work, we state it below for the reader's convenience. 

\begin{theorem}[Angenent, Daskalopoulos, Sesum in \cite{ADS}]
  \label{thm-old}
  Let $\{M_t\}$ be any $O(1)\times O(n)$ invariant Ancient Oval (see Definition \ref{def-oval}) .
  Then the solution $u(y,\tau)$ to \eqref{eq-u}, defined on $\mathbb{R}\times \mathbb{R}$, has the following asymptotic expansions:
  \begin{enumerate}
    \item[(i)] For every $M > 0$,
    \[
    u(y,\tau) = \sqrt{2(n-1)} \Bigl(1 - \frac{y^2 - 2}{4|\tau|}\Bigr) + o(|\tau|^{-1}), \qquad |y| \le M
    \]
    as $\tau \to -\infty$.
    \item[(ii)]  Define $z := {y}/{\sqrt{|\tau|}}$ and $\bar{u}(z,\tau) := u(z\sqrt{|\tau|}, \tau)$.
    Then, 
    $$\lim_{\tau \to -\infty} u(z,\tau) = \sqrt{(n-1)\, (2 - z^2)}$$
    uniformly on compact subsets in $|z| \leq \sqrt{2}$.
    
    \item[(iii)] Denote by $p_t$ the tip of $M_t \subset \mathbb{R}^{n+1}$, and define for any $t_*<0$ the  rescaled flow at the tip  
    \[
    \tilde{M}_{t_*}(t) = \lambda(t_*) \bigl\{M_{t_* + t \lambda(t_*)^{-2}} - p _{t_*}\bigr\}
    \]
    where
    \[
    \lambda(t) := H(p_{t}, t) = H_{\max}(t) = \sqrt{ \tfrac12 |t|\log|t| } \bigl(1+o(1)\bigr)
    \]
    Then, as $t_*\to-\infty$, the family of mean curvature flows $\tilde M_{t_*}(\cdot)$ converges to the unique unit speed Bowl soliton, i.e.~the unique convex rotationally symmetric translating soliton with velocity one.
  \end{enumerate}
\end{theorem}\noindent

Before we conclude our introduction we give a short description of our proof for Theorem \ref{thm-main}.
A more detailed outline of this proof is given in Section \ref{sec-regions}. 

\smallskip
\noindent{\em Discussion on the proof of Theorem \ref{thm-main}.}   The proof of Theorem \ref{thm-main} makes extensive use of  our previous work \cite{ADS} where the detailed asymptotic behavior of
Ancient  Ovals, as $\tau:= - \log |t|  \to -\infty$,  was given under the assumption of $O(1)\times O(n)$ symmetry (see Theorem  ~\ref{thm-old} below). 
Note that our symmetry result, Theorem ~\ref{thm-rot-symm}, which will be shown in Section \ref{sec-rot-symm},   only shows the $O(n)$-symmetry of solutions and not the $O(1)\times O(n)$-symmetry assumed in Theorem \ref{thm-old}. 
However, as we will demonstrate in the Appendix of this work (see Theorem \ref{thm-O1}),   the estimates in  Theorem  ~\ref{thm-old} 
simply extend to the $O(n)$-symmetric case.
Since the proof of Theorem \ref{thm-main} is quite involved, in Section \ref{sec-regions} we will give an outline of
the different steps of our proof.
The main idea is simple: given $U_1(x,t)$ and $U_2(x,t)$ any two solutions of \eqref{eq-u-original},
we will find parameters $\alpha, \beta, \gamma$, corresponding to translations along the x-axis, translations in time $t$ and parabolic rescaling respectively,  such that 
$U_1(x,t) \equiv  U^{\alpha\beta\gamma}(x,t)$, where  $U^{\alpha\beta\gamma}$ denotes the image of $U(x,t)$ under these transformations (see \eqref{eq-Ualphabeta}). 
To achieve this uniqueness, we will consider  the corresponding  rescaled profiles $u_1(y,\tau), u^{\alpha\beta\gamma}_2(y,\tau)$ and show that $w:= u_1(y,\tau)- u^{\alpha\beta\gamma}_2(y,\tau) \equiv 0$.
It  will mainly follow from analyzing the equation for $w$ in the {\em cylindrical region}
(the region $\{ (y, \tau): \,\, u_1(y,\tau) \geq \theta/2 >0 \}$, for some $\theta >0$ and small).
We  restrict $w$ to the cylindrical region by  introducing an  appropriate cut off function $\varphi_{\cC}$ and setting $w_{\cC}:= \varphi_{\cC}\, w$.
The difference  $w_{\cC}$ in this region satisfies the equation 
\begin{equation}\label{eq-wC-pde}
  \partial_t (w_{\cC} ) = \cL w_{\cC} + \mathcal{E} [w, \varphi_C].
\end{equation}
for a nonlinear error term $\mathcal{E} [w, \varphi_C]$.
The operator $ \cL := \partial_y^2 - \frac y2 \partial_y + 1$ is simply the linearized operator 
for equation \eqref{eq-u} on the cylinder which we see in the middle, i.e. constant the $ \sqrt{2(n-1)}$.
This operator is well studied and it is known  to have two unstable modes (corresponding to
two positive eigenvalues) and one neutral mode (corresponding to the zero eigenvalue).
The uniqueness at the end follows by a coercive estimate on \eqref{eq-wC-pde} 
with the right norm (we call it $ \| \cdot \|_{2,\infty}$),  which roughly implies that if $w \not\equiv 0$, then 
\begin{equation}\label{eq-wC-estimate}
  \| w_{\cC} \|_{2,\infty} \leq C \, \| \mathcal{E} [w, \varphi_{\cC}] \|_{2,\infty}  < \frac 12 \, \| w_{\cC} \|_{2,\infty}
\end{equation} 
thus leading to a contradiction.
It is apparent that to obtain such a coercive estimate one needs  to adjust the parameters $\alpha, \beta, \gamma$ in
such a way   that the projections $\pr_+ w (\tau)$ and $\pr_0 w(\tau)$ onto the positive and zero eigenspaces of $\cL$ are all {\em simultaneously zero } at some time 
$\tau_0 \ll -1$.
The main challenge in showing \eqref{eq-wC-estimate} comes from the  \emph{error terms}  which are introduced by the cut-off function $\varphi_{\cC}$
and supported at the {\em transition region}
between the {\em cylindrical } and {\em tip} regions (the latter is defined to be the region $\{ (y, \tau): \,\, u_1(y,\tau) \geq 2\theta \}$). 
To estimate these errors one needs to consider our equation in the tip region and show a suitable coercive estimate there which allows us to {\em bound back $w$ in the tip region back in terms of $w_{\cC}$}.
To achieve this, one heavily uses the \emph{a priori} estimates and Theorem \ref{thm-old} from \cite{ADS}.
We also need to introduce an appropriate weighted norm in the tip region which lets us show the Poincar\'e type estimate we need to proceed.
Unfortunately, numerous technical difficulties arise from various facts including the non-compactness of the limit as $\tau \to -\infty$ and the fact that $u_y \to \pm \infty$ at the tips.

\medskip

In previous classifications of ancient solutions to mean curvature flow and Ricci flow, \cite{DHS1}, \cite{DHS2}, \cite{BC, BC2}, an essential role in the proofs was played by the fact that all such solutions were given in closed form or they were solitons. 
One of the significance of our techniques in our current work is that they overcome such a requirement and potentially can be used in many other parabolic equations and particularly in other geometric flows.
To our knowledge, our work and the recent work by Bourni, Langford and Tinaglia \cite{BLT} are the first classification results of geometric ancient solutions where the solutions are not given in closed form and they are not solitons.
Let us also point out that our current techniques are reminiscent of the significant work by Merle and Zaag in \cite{MZ} which has provided an inspiration for us. 

\medskip

{\bf Acknowledgements:}  The authors are indebted to S.
Brendle  for many useful discussions regarding the rotational symmetry of ancient solutions.

\section{Rotational symmetry}\label{sec-rot-symm}

The main goal in this section is to prove Theorem \ref{thm-rot-symm}.
Our proof of Theorem \ref{thm-rot-symm} follows closely the arguments of
the recent  work by Brendle and Choi \cite{BC, BC2} on the uniqueness of strictly convex, uniformly 2-convex, noncompact and noncollapsed ancient solutions 
of mean curvature flow in $\R^{n+1}$.
It was shown in \cite{BC} that such solutions are rotationally symmetric.
Then by analyzing 
the rotationally symmetric solutions, Brendle and Choi showed that such solutions agree with the Bowl soliton. 
For the reader's convenience we state their result next.  

\begin{theorem}[Brendle and Choi \cite{BC}]
  \label{thm-BC}
  Let $\{M_t : t \in (-\infty,0)\}$ be a noncompact ancient mean curvature flow in $\R^{n+1}$ which is strictly convex, noncollapsed, and uniformly 2-convex.
  Then $M_t$ agrees with the Bowl soliton, up to scaling and ambient isometries.
\end{theorem}

In the proof of Theorem \ref{thm-rot-symm} we will use both the key results 
that led to the proof of the main theorem in \cite{BC} (see Propositions \ref{prop-neck} and \ref{prop-cap} below), and the uniqueness result as stated in Theorem \ref{thm-BC}. 
\smallskip  

Before we proceed with the proof of Theorem~\ref{thm-rot-symm}, let us recall some standard notation.  
Our solution $M_t$ is embedded in $\R^{n+1}$,  for all $t \in (-\infty,T)$ and in the mean curvature flow, time scales like distance squared.  
We denote by $\cP(\bar x,\bar t,r)$ the \emph{parabolic cylinder} centered at  $(\bar x,\bar t) \in \R^{n+1}\times \R$ of radius $r > 0$, namely the set 
\[
\cP(\bar x,\bar t,r) := \cB(\bar x,r) \times [\bar t-r^2,\bar t]
\]
where $\cB(x,r):= \{ x \in \R^{n+1} \mid |x-\bar x| \leq r \} $ denotes the \emph{closed} Euclidean ball of radius $r$ in $\R^{n+1}$. 

Also, following the notation in \cite{HS} and \cite{BC}, we denote by $ \hat \cP (\bar x,\bar t,r)$ the {\em rescaled by the mean curvature}
parabolic cylinder centered at $(\bar x,\bar t) \in \mathbb{R}^{n+1}\times \R$ of radius $r > 0$,  namely the set
\[
\hat {\cP} (\bar x,\bar t,r) 
:= \cP(\bar x,\bar t, \hat \rho (\bar x, \bar t)\, r),
\qquad 
\hat \rho(\bar x, \bar t) :=\frac{n}{H(\bar x,\bar t)}.
\]

Note that in \cite[\S7]{HS} Huisken and Sinestrari consider parabolic cylinders with respect to the intrinsic metric $g(t)$ on the solution $M_t$, which in our case is equivalent to the extrinsic metric on space-time that we are considering here.  

We recall Brendle and Choi's \cite{BC} definition of  a mean curvature flow being $\epsilon$-symmetric, in terms of the normal components of rotation vector fields.  
In what follows we identify $\so(n)$ with the subalgebra of $\so(n+1)$ consisting of skew symmetric matrices of the form
\[
J = \begin{bmatrix}
  0 & 0 \\ 0 & J'
\end{bmatrix},\quad
\text{ with }\quad
J'\in\so(n).
\]
Thus $\so(n)$ acts on the second factor in the splitting $\R^{n+1}=\R\times\R^n$.
Any $J\in \so(n+1)$ generates a vector field on $\R^{n+1}$ by $\vec v(x) =  Jx$.
If $\varPhi(x) = Sx+p$ is a Euclidean motion, with $p\in\R^{n+1}$ and $S\in O(n+1)$, then the pushforward of the vector field $\vec v(x) = Jx$ under $\varPhi$ is given by 
\[
\varPhi_*\vec v(x) = d\varPhi_x\cdot \vec v(\varPhi^{-1}x) = SJS^{-1}(x-p).
\]
Any vector field of this form is a \emph{rotation vector field.}
\begin{definition}
  A collection of vector fields $\mathcal{K} := \{K_{\alpha}\mid 1 \le \alpha \le \frac{1}{2}n(n-1)\}$ on $\mathbb{R}^{n+1}$ is a \emph{normalized set of rotation vector fields} if there exist an orthonormal basis $\{J_{\alpha}\mid 1 \le \alpha \le \frac{1}{2}n(n-1)\}$ of $\so(n) \subset \so(n+1)$, a matrix $S \in O(n+1)$, and a point $q\in \mathbb{R}^{n+1}$ such that
  \[
  K_{\alpha}(x) = S J_{\alpha} S^{-1}(x-q).
  \]
\end{definition}

\begin{definition}
  \label{def-normalized}
  Let $M_t$ be a solution of mean curvature flow.
  We say that a point $(\bar{x},\bar{t})$ is $\epsilon$-symmetric if there exists a normalized set of rotation vector fields $\mathcal{K}^{(\bar{x},\bar{t})} = \{K_{\alpha}^{(\bar{x},\bar{t})} \mid 1 \le \alpha \le \frac{1}{2}n(n-1)\}$ such that $\max_{\alpha}|\langle K_{\alpha}, \nu\rangle| H \le \epsilon$ in the parabolic neighborhood $\bar{\cP}(\bar{x},\bar{t},10)$.
\end{definition}

Lemma 4.3 in \cite{BC} allows us to control how the axis of rotation of a normalized set of rotation vector fields $\mathcal{K}^{(x,t)}$ varies as we vary the point $(x,t)$. 

The proof of Theorem \ref{thm-rot-symm} relies on the following two key propositions which were both shown in  \cite{BC}. 
The first proposition is directly taken from \cite{BC} (see Theorem 4.4 in \cite{BC}).
The second proposition required some modifications of arguments in \cite{BC} and hence we present parts of its proof below (see \ref{prop-cap}). 

\begin{definition}
  A point $(x, t)$ of a mean curvature flow lies on an $(\epsilon, L)$-neck if there is a Euclidean transformation $\varPhi:\R^{n+1}\to\R^{n+1}$ , and a scale $\lambda>0$ such that
  \begin{itemize}
    \item $\varPhi$ maps $x$ to $(0, \sqrt{2(n-1)}, 0, \dots, 0)$
    \item for all $\tau\in[-L^2, 0]$ the hypersurface $\lambda^{-1} \varPhi \bigl(M_{t + \lambda^2\tau}\bigr)$ is $\epsilon$-close in $C^{20}$ to the cylinder of length $L$, of radius $\sqrt{2(n-1)(1-\tau)}$, and with the $x_1$-axis as symmetry axis.
  \end{itemize}
\end{definition}

\begin{proposition}[Neck Improvement - Theorem 4.4 in \cite{BC}]
  \label{prop-neck}
  There exists a large constant $L_0$ and a small constant $\epsilon_0$ with the following property.  
  Suppose that $M_t$ is a mean curvature flow, and suppose that $(\bar{x}, \bar{t})$ is a point in space-time with the property that every point in $\hat \cP(\bar{x},\bar{t}, L_0)$ is $\epsilon$-symmetric and lies on an $(\epsilon_0, 10)$-neck, where $\epsilon \le \epsilon_0$.
  Then $(\bar{x}, \bar{t})$ is $\frac{\epsilon}{2}$-symmetric.
\end{proposition}

\begin{proof}
  The proof is given in Theorem 4.4 in \cite{BC}. 
\end{proof}

The next result will be shown by slight modification of arguments in the proof of Theorem 5.2 in \cite{BC}.
The proof of Proposition \ref{prop-cap} below follows closely arguments in \cite{BC}.

\begin{proposition}[Cap Improvement \cite{BC}]
  \label{prop-cap}
  Let $L_0$ and $\epsilon_0$ be chosen as in the Neck Improvement Proposition \ref{prop-neck}.
  Then there exist a large constant $L_1 \ge 2L_0$ and a small constant $\epsilon_1 \le \frac{\epsilon_0}{2}$ with the following property.
  Suppose that $M_t$ is a mean curvature flow solution defined on   $\hat \cP(\bar{x},\bar{t},L_1)$.
  Moreover, we assume that 
  $\hat \cP(\bar{x},\bar{t},L_1)$ is, after scaling to make $H(\bar x, \bar t)=1$,  $\epsilon_1$-close in the $C^{20}$-norm to a piece of a Bowl soliton which includes the tip (where the tip lies well inside in the interior of that piece of a Bowl soliton, at a definite distance from the boundary of that piece of the Bowl soliton) 
  and that every point in $\hat \cP(\bar{x},\bar{t},L_1)$ is $\epsilon$-symmetric, where $\epsilon \le \epsilon_0$.
  Then $(\bar{x}, \bar{t})$ is $\frac{\epsilon}{2}$-symmetric.
\end{proposition}

\begin{proof}
  Without loss of generality assume $\bar{t} = 1$ and $H(\bar{x},-1) = 1$ and $\bar{x}\in \Omega_{-1}^1$.
  For the sake of the proof, keep in mind that the statement in the Proposition is of local nature and it only matters what is happening on a large parabolic cylinder $\hat{\cP}(\bar{x},-1,L_1)$, while the behavior of our solution outside of this neighborhood does not matter.
  
  The assumptions in the Proposition imply that if we take $\epsilon_1$ sufficiently small, using that the Hessian of the mean curvature around the maximum mean curvature point in a Bowl soliton is strictly negative definite (note that by our assumption 
  we may assume the maximum of $H(\cdot,t)$ in $B(\bar{x},L_1)\cap M_t$ is attained at a unique interior point $q_t \in B(\bar{x},L_1)\cap M_t$.
  Moreover, the Hessian of the mean curvature at $q_t$ is negative definite.
  Hence, $q_t$ varies smoothly in $t$.
  We now conclude that if $(x_0,t_0) \in \hat{\cP}(\bar{x},-1,L_1)$, then 
  \begin{equation}
    \label{eq-mon-tip}
    \frac{d}{dt}|x_0 - q_t| < 0, \qquad -1 - L_1^2 \le t \le t_0.
  \end{equation}
  The proof of \eqref{eq-mon-tip} is the same as the proof of Lemma 5.2 in \cite{BC}. 
  
  We claim that there exists a uniform constant $s_*$ with the property that every point $(x,t) \in \hat{\mathcal{P}}(\bar{x},-1,L_1)$ with $|x - q_t| \ge s_*$ lies on an $(\epsilon_0,10)$-neck and satisfies $|x - q_t| H(x,t) \ge 1000 L_0$.
  Indeed, knowing the behavior of the Bowl soliton, it is a straightforward computation to check previous claims are true on the Bowl soliton, with a constant for example, $2000 L_0$.
  By our assumption, $\hat{\mathcal{P}}(\bar{x},-1,L_1)$ is $\epsilon_1$-close to the Bowl soliton and hence the claims are true for our solution as well.
  
  If $|\bar{x} - q_{-1}| \ge s_*$, the Proposition follows immediately from Proposition \ref{prop-neck}.
  Thus, we may assume that $|\bar{x} - q_{-1}| \le s_*$.
  Then we have the following claim.
  
  \begin{claim}
    Suppose that $M_t$ is an ancient solution of mean curvature flow.
    Given any positive integer $j$, there exist a large constant $L(j)$
    and a small constant $\epsilon(j)$ with the following property: if the parabolic neighborhood $\hat{\mathcal{P}}(\bar{x},-1,L(j))$ is $\epsilon(j)$-close in the $C^{20}$-norm to a piece of the Bowl soliton which includes the tip, and every point in $\hat{\mathcal{P}}(\bar{x},-1,L(j))$ is $\epsilon$-symmetric, then every point $(x,t)\in \hat{\mathcal{P}}(\bar{x},-1,L(j))$ with $t \in [-2^{\frac{3j}{100}},-1]$ and $s_* 2^{\frac{j}{100}} \le |x - q_t| \le s_* 2^{\frac{j+1}{100}}$ is $2^{-j} \epsilon$-symmetric.
  \end{claim}
  
  \begin{proof}
    Assume $\kappa$ is the maximal curvature of the tip of the Bowl soliton in the statement of the Claim.
    Note that $\kappa$ may depend on $(\bar{x},\bar{t})$, but is independent of $j$.
    Define $L(j) \ge \frac{2^{\frac{1}{100}}}{1000}\, s_* + s_* 2^{\frac{j+2}{100}} + \kappa\, (2^{\frac{3(j+1)}{100}}+1) + s_*$, which choice will become apparent later.
    
    The proof is by induction on $j$ and is similar to the proof of Proposition 5.3 in \cite{BC}.
    For $j = 0$ we have $s_* \le |x - q_t| \le s_* 2^{\frac{1}{100}}$.
    By above discussion we have that $(x,t)$ lies on an $(\epsilon_0,10)$-neck and $|x-q_t| \ge \frac{1000 L_0}{H(x,t)}$.
    This implies $\frac{L_0}{H(x,t)} \le \frac{s^* 2^{\frac{1}{100}}}{1000}$ and hence $\hat{\cP}(x,t,L_0) \subset \hat{\cP}(\bar{x},-1,L_1)$ if we choose $L_1$ sufficiently big compared to $s_*$.
    Hence, every point in $\hat{\cP}(x,t,L_0)$ is $\epsilon$-symmetric and lies on an $(\epsilon_0,10)$-neck (where $\epsilon < \epsilon_0$).
    By Proposition \ref{prop-neck} we conclude $(x,t)$ is $\frac{\epsilon}{2}$-symmetric.
    
    Assume the claim holds for $j-1$.
    We want to show it holds for $j$ as well, that is,  if the parabolic neighborhood $\hat{\mathcal{P}}(\bar{x},-1,L(j))$ is $\epsilon(j)$-close in the $C^{20}$-norm to a piece of the Bowl soliton which includes the tip, and every point in $\hat{\mathcal{P}}(\bar{x},-1,L(j))$ is $\epsilon$-symmetric, then every point $(x,t)\in \hat{\mathcal{P}}(\bar{x},-1,L(j))$ with $t \in [-2^{\frac{3j}{100}},-1]$ and $s_* 2^{\frac{j}{100}} \le |x - q_t| \le s_* 2^{\frac{j+1}{100}}$ is $2^{-j} \epsilon$-symmetric.
    Suppose this is false.
    Then there exists $(x_0,t_0)$ so that $x_0\in M_{t_0}$ and $t_0\in [-2^{\frac{3j}{100}},-1]$ and $s_* 2^{\frac{j}{100}} \le |x_0 - q_{t_0}| \le s_* 2^{\frac{j+1}{100}}$ so that $(x_0,t_0)$ is not $2^{-j}\epsilon$-symmetric.
    By Proposition \ref{prop-neck}, there exists a point $(y,s) \in \hat{\cP}(x_0,t_0,L_0)$ such that either $(y,s)$ is not $2^{-j+1}\epsilon$-symmetric or $(y,s)$ does not lie at the center of an $(\epsilon_0,10)$-neck.
    Note that if we choose $L(j) \ge \frac{2^{\frac{1}{100}}}{1000}\, s_* + s_* 2^{\frac{j+2}{100}} + \kappa\, (2^{\frac{3(j+1)}{100}}+1) + s_*$, then  $(y,s) \in \hat{\cP}(\bar{x},-1,L(j-1))$.
    To see this we combine the inequalities $|y-x_0| \leq \frac{L_0}{H(x_0,t_0)}$,  $|\bar x - q_{-1}| \leq s_*$, 
    $|x_0 - q_{t_0}| \le s_* 2^{\frac{j+1}{100}}$ to conclude that
    \begin{equation*}
      \begin{split}
        |y-\bar x| &\leq |y-x_0| +  |x_0 - q_{t_0}| + |q_{t_0} - q_{-1}| + |q_{-1} - \bar x| \\&\leq   \frac{L_0}{H(x_0,t_0)} +  s_* 2^{\frac{j+1}{100}} +  \kappa \, |t_0+1| + s_* \\
        &\leq \frac{2^{\frac 1{100}}} {1000} \, s_* + s_* 2^{\frac{j+1}{100}} + \kappa \, (2^{\frac{3j}{100}} +1 )  + s_*\\
        &\le L(j-1).
      \end{split}
    \end{equation*}
    
    In view of the induction hypothesis, we conclude that $|x - q_t| \le 2^{\frac{j-1}{100}} s_*$.
    We can now follow the proof of Proposition 5.3 in \cite{BC} closely to obtain a contradiction.
    In that part of the proof one needs to use \eqref{eq-mon-tip}, which is proved in Lemma 5.2 in \cite{BC}.
    It is clear from the proof in \cite{BC} that if we take a bigger parabolic cylinder around $(\bar{x},-1)$ of size $L(j)$, in order to still have \eqref{eq-mon-tip} one needs to require that $\hat{\cP}(\bar{x},-1,L(j))$ is $\epsilon(j)$ close to a Bowl soliton, where $\epsilon(j)$ needs to be taken very small, depending on $L(j)$.
    This is clear from the proof of \eqref{eq-mon-tip} that can be found in \cite{BC}.
  \end{proof}
  
  In the following, $j$ will denote a large integer, which will be determined later.
  Moreover, assume that $L \ge L(j)$ and $\epsilon \le \epsilon(j)$. 
  Using the Claim, we conclude that for every point $(x,t)\in \{(x,t)\,\,|\,\, t\in [-2^{\frac{3j}{100}},-1], \,\,\,\, s_* 2^{\frac{j+1}{100}} \le |x - q_t| \le s_* 2^{\frac{j}{100}}\}$ there exists a normalized set of rotation vector fields $\mathcal{K}^{(x,t)} = \{K_{\alpha}^{(x,t)}\,\,\,|\,\,\, 1 \le \alpha \le \frac{n(n-1)}{2}\}$, such that $\max_{\alpha}|\langle K_{\alpha}^{(x,t)},\nu\rangle| H \le 2^{-j} \epsilon$ on $\bar{\cP}(x,t,10)$.
  Moreover, since $|x-q_t| \ge s_*$ implies $H(x,t) |x-q_t| \ge 1000 L_0$, we have that
  \[\max_{\alpha}|\langle K^{(x,t)}_{\alpha},\nu\rangle| \le \frac{2^{\frac{j+1}{100} - j} s_*}{1000 L_0}\, \epsilon \le C \, 2^{\frac{j}{10} - j} \, \epsilon, \qquad j \ge j_0,\]
  for a uniform constant $C$ that is independent of $j$ and $\epsilon$.
  Lemma 4.3 in \cite{BC} allows us to control how the axis of rotation of $\mathcal{K}^{(x,t)}$ varies as we vary the point $(x,t)$. 
  More precisely, as in \cite{BC}, if $(x_1,t_1)$ and $(x_2,t_2)$ are in $\{(x,t)\,\,|\,\, t\in [-2^{\frac{3j}{100}},-1], \,\,\,\, s_* 2^{\frac{j+1}{100}} \le |x - q_t| \le s_* 2^{\frac{j}{100}}\}$ and $(x_2,t_2) \in \hat{\cP}(x_1,t_1,1)$, then
  \begin{align*}
    & \inf_{w\in O\left(\frac{n(n-1)}{2}\right)}\, \sup_{B_{10H(x_2,t_2)^{-1}}(x_2)}\, \max_{\alpha}\left| K_{\alpha}^{(x_1,t_1)} - \sum_{\beta=1}^{\frac{n(n-1)}{2}} w_{\alpha\beta} K_{\beta}^{(x_2,t_2)}\right| \\
    &\le C 2^{-j} H(x_2,t_2)^{-1}.
  \end{align*}
  Hence, we can find a normalized set of rotation vector fields $\mathcal{K}^{(j)} = \{K_{\alpha}^{(j)}\,\,\,|\,\,\, 1 \le \alpha \le \frac{n(n-1)}{2}\}$ so that if $(x,t) \in \{(x,t)\,\,|\,\, t\in [-2^{\frac{3j}{100}},-1], \,\,\,\, s_* 2^{\frac{j+1}{100}} \le |x - q_t| \le s_* 2^{\frac{j}{100}}\}$, then,
  \[
  \inf_{w\in O\left(\frac{n(n-1)}{2}\right)}\, \max_{\alpha} \left|K_{\alpha}^{(j)} - \sum_{\beta=1}^{\frac{n(n-1)}{2}} w_{\alpha\beta} K_{\beta}^{(x,t)}\right| \le C 2^{-\frac j2},
  \]
  at the point $(x,t)$.
  From this we deduce that $\max_{\alpha}|\langle K_{\alpha}^{(j)},\nu\rangle| \le C 2^{-\frac j2}$, for all points $(x,t) \in \{(x,t)\,\,|\,\, t\in [-2^{\frac{3j}{100}},-1], \,\,\,\, s_* 2^{\frac{j+1}{100}} \le |x - q_t| \le s_* 2^{\frac{j}{100}}\}$.
  As in \cite{BC} we conclude that $\max_{\alpha} |\langle K_{\alpha}^{(j)},\nu\rangle| \le C 2^{-\frac j2}$ for all $(x,t) \in  \{(x,t)\,\,|\,\, t\in [-2^{\frac{3j}{100}},-1], \,\,\,\, s_* 2^{\frac{j+1}{100}} \le |x - q_t| \le s_* 2^{\frac{j}{100}}\}$.
  Finally, note that $\max_{\alpha}|\langle K_{\alpha}^{(j)},\nu\rangle| \le C 2^{\frac{j}{100}}$, whenever $x\in  \{x\,\,|\,\,  s_* 2^{\frac{j+1}{100}} \le |x - q_t| \le s_* 2^{\frac{j}{100}}\}$ and $t = -2^{\frac{3j}{100}}$.
  
  As in \cite{BC}, for each $\alpha\in \{1,\dots, \frac{n(n-1)}{2}\}$ we define a function $f_{\alpha}^{(j)} : \{(x,t)\,\,\,|\,\,\,  t\in [-2^{\frac{3j}{100}},-1], \,\,\,\,   |x - q_t| \le s_* 2^{\frac{j}{100}}\} \to \mathbb{R}$ by
  \[
  f_{\alpha}^{(j)} := e^{2^{-\frac{j}{50}} t} \frac{\langle K_{\alpha}^{(j)},\nu\rangle}{H - 2^{-\frac{j}{100}}}.
  \]
  The same computation as in \cite{BC} implies that by the maximum principle applied to the evolution of $f_{\alpha}^{(j)}$ we get
  \[
  \sup_{
  \{
  (x,t)\,\,\,|\,\,\, t\in [-2^{\frac{3j}{100}},-1], \,\,\,\, 
  s_* 2^{\frac{j+1}{100}} \le |x - q_t| \le s_* 2^{ \frac{j}{100} }
  \} } 
  |f_{\alpha}^{(j)}(x,t)| \le C\, 2^{-\frac j4}.
  \]
  Standard interior estimates for parabolic equations give estimates for the higher order derivatives of $\langle K_{\alpha}^{(j)},\nu\rangle$.
  
  Hence, if we choose $j$ sufficiently big, then the same reasoning as in \cite{BC} implies $(\bar{x},-1)$ is $\frac{\epsilon}{2}$-symmetric.
  Having chosen $j$ in this way, we finally define $L_1 = L(j)$ and $\epsilon_1 := \epsilon(j)$.
  Then $L_1$ and $\epsilon_1$ have the desired properties as stated in Proposition \ref{prop-cap}.
\end{proof}

The goal of the remaining part of this section is to show how we can employ Propositions \ref{prop-neck} and \ref{prop-cap} to prove Theorem \ref{thm-rot-symm}. 

Observe that by the crucial work of Haslhofer and Kleiner in \cite{HK} we know that a strictly convex $\alpha$-noncollapsed ancient solution to mean curvature flow sweeps out the whole space.
Hence,    the  
well known important result of X.J.
Wang in \cite{W}  shows that the rescaled flow, after a proper rotation of coordinates, converges, as time goes to 
$-\infty$, uniformly on compact sets, to a round cylinder of radius $\sqrt{2(n-1)}$.

This has as a consequence that $M_t\cap B_{8(n-1)\sqrt{|t|}}$ is a neck with radius $\sqrt{2(n-1)|t|}$.
The complement 
$M_t\backslash B_{8(n-1)\sqrt{|t|}}$ has two connected components, call them $\Omega_1^t$ and $\Omega_2^t$, both compact.
Thus,   for every $t$, the maximum of $H$ on  $\Omega_1^t$ is attained at least at one point
in  $\Omega_1^t$ and similarly for $\Omega_2^t$.

For every $t$, we define the {\it tip} points $p_{t}^1$ and $p_{t}^2$ as follows.
Let $p_t^k$, for $k=1,2$ be a point such that
\[
|\langle F,\nu\rangle (p_t^k,t)| = |F|(p_t^k,t) \quad \mbox{and} \quad |F|(p_t^k,t) = \max_{\Omega_k^t} |F|(\cdot,t).
\]
Denote by $d_k(t) := |F|(p_t^k,t)$, for $k\in \{1,2\}$. 

Throughout the rest of the section we will be using the next observation about possible limits of our solution around arbitrary sequence of points $(x_j,t_j)$
with $x_j \in M_{t_j}$, $t_j \to -\infty$ when rescaled by $H(x_j,t_j)$.

\begin{lemma}
  \label{lemma-possible-limits}
  Let $M_t$, $t\in (-\infty,0)$ be an Ancient Oval satisfying the assumptions in Theorem \ref{thm-rot-symm}.
  Fix a $k\in \{1,2\}$.
  Then for every sequence of points $x_j\in M_{t_j}$ and any sequence of times $t_j\to -\infty$, the rescaled sequence of solutions $F_j(\cdot,t) := H(x_j,t_j) (F(\cdot,t_j+t Q_j^{-2}) - x_j)$ subconverges to either a Bowl soliton or a shrinking round cylinder.
\end{lemma}

\begin{proof}
  By the global convergence theorem (Theorem 1.12) in \cite{HK} we have that after passing to a subsequence, the flow $M^j_t$ converges, as $j\to \infty$,  to an ancient solution 
  $M^{\infty}_t$, for $t \in (-\infty, 0]$, which is convex and uniformly 2-convex.
  Note that $H(0,0) = 1$ on the limiting manifold.
  By the strong maximum principle applied to $H$ we have that $H > 0$ everywhere on $M^{\infty}_t$, where $t\in (-\infty,0]$.
  If $M^{\infty}_t$ is strictly convex, by the classification result in \cite{BC} we have that it is a Bowl soliton.
  If the limit is not strictly convex, by the strong maximum principle it splits off a line and hence it is of the form $N_t^{n-1}\times \mathbb{R}$, where $N_t^{n-1}$ is an $n-1$-dimensional ancient solution.
  On the other hand the uniform 2-convexity assumption on our solution implies the inequality  $\lambda_{\min}({N^{n-1}_t}) \ge \beta H (N^{n-1}_t)$, 
  for a uniform constant $\beta > 0$.
  Thus,  Lemma 3.14 in \cite{HK} implies that the limiting flow $M_t^{\infty}$ is a family of round shrinking cylinders $S^{n-1}\times\mathbb{R}$.
\end{proof}

We will next show that points which are away from the tip points in both regions $\Omega_t^k$, $k=1,2$ are cylindrical. 

\begin{lemma}
  \label{lemma-cylindrical}
  Let $M_t$, $t \in (-\infty,0)$, be an Ancient Oval satisfying the assumptions of Theorem \ref{thm-rot-symm} and fix $k\in 1,2$. 
  Then,  for every $\eta > 0$ there exist $\bar{L}$ and $t_0$,  so that for all  $t \le t_0$ and $L \ge \bar{L}$ the following holds 
  \begin{equation}\label{eqn-cyl10}
    |x - p_t^k|   \ge \frac{L}{H(x,t)} \implies  \frac{\lambda_{min}}{H}(x,t) < \eta.
  \end{equation}
  We may chose $L$ so that \eqref{eqn-cyl10} holds for both $k=1,2$. 
\end{lemma}

\begin{proof}
  Without loss of generality we may assume that $k=1$ and we will argue by contradiction.
  If the statement is not true, 
  then there exist $L_j\to \infty$  and sequences of times $t_j\to -\infty$ and points  $x_j \in M_{t_j}$  so that
  \begin{equation}
    \label{eqn-contra}
    |x_j - p_{t_j}^1| \ge \frac{L_j}{H(x_j,t_j)} \quad \mbox{and} \quad \frac{\lambda_{\min}}{H}(x_j,t_j)  \geq \eta.
  \end{equation}
  Rescale the flow around $(x_j,t_j)$ by $Q_j := H(x_j,t_j)$ as in Lemma \ref{lemma-possible-limits} and call the rescaled manifolds $M^j_t$. 
  Then, 
  \begin{equation}
    |0 - \bar{p}_j^1|   \ge L_j \to \infty,  \qquad \mbox{as} \,\,\,\, j\to\infty
  \end{equation}
  where the origin and $\bar{p}_j^1$ correspond to $x_j$ and tip points $p_{t_j}^1$ after rescaling, resepctively.
  By Lemma \ref{lemma-possible-limits} we have that passing to a subsequence $M_t^j$ converges to either a Bowl soliton or a cylinder. 
  Since  $\frac{\lambda_{min}}{H}$ is a scaling invariant quantity, \eqref{eqn-contra} implies that on the limiting manifold we have 
  $\frac{\lambda_{min}}{H} (0,0) \geq \eta$
  which immediately excludes the cylinder.
  Thus the limiting manifold must be the Bowl soliton. 
  
  Lets look next at the tip points $p_{t_j}^2$ of our solution which lie on the other side $\Omega_{t_j}^2$ and denote by $\bar{p}_j^2$ the corresponding points
  on our rescaled solution.
  Then we must have that $ |0 - \bar{p}_j^2|   \le C_0 $ for some constant $C_0$.
  Otherwise, if  we had that $\limsup_{j \to +\infty}|0 - \bar{p}_j^2| \to +\infty$, this
  together with \eqref{eqn-contra},   the convexity of our surface, the fact that the furthest points $p_{t_j}^1$ and $p_{t_j}^2$ lie on the opposite side of a necklike piece and the splitting theorem 
  would imply that the limit of $M_t^j$ would split off a line.
  This and Lemma \ref{lemma-possible-limits} would yield the limit of $M_t^j$ would have been the cylinder which we have already ruled out. 
  In terms of our unrescaled solution $M_t$ then we conclude that 
  $ |x_j- p_{t_j}^2|   \leq \frac{C_0}{H(x_j,t_j)}$.
  Since $x_j\in \Omega_{t_j}^1$ and $p_{t_j}^2\in \Omega_{t_j}^2$, we would have that the whole neck-like region that divides the sets $\Omega_{t_j}^1$ and $\Omega_{t_j}^2$ lies at a distance less that equal to $\frac{C_0}{H(x_j,t_j)}$ from $x_j$.
  This would imply the whole neck-like region would have to lie on 
  a compact set of the Bowl soliton, implying that $\frac{\lambda_{min}}{H}  (\cdot, t_j) \geq c_0 >0$ holds for some constant $c_0$, independent of $j$.
  This is a contradiction since on the neck-like region of our solution, the scaling invariant quantity $ \lambda_{\min}  H^{-1} \to 0$ as $t_j \to -\infty$. 
  The above discussion shows that $|x - p_t^1|    \ge \frac{L}{H(x,t)}$ implies that $  \frac{\lambda_{min}}{H}(x,t) < \eta$, thus finishing the proof of the Lemma. 
  %
  Recall that $|0 - \tilde{p}_{t_j}^k| \ge L_j \to \infty$ as $j\to\infty$, for $k=1,2$.
  This together with the convexity of our solution and the fact that the furthest points $p_{t_j}^1$ and $p_{t_j}^2$ lie on the opposite side of a necklike piece of our surface imply the limit $M^{\infty}_t$ contains a line.
  Hence, by Lemma \ref{lemma-possible-limits} we conclude that $M_t^{\infty}$ is a family of round shrinking cylinders $S^{n-1}\times\mathbb{R}$.
\end{proof}

In the following lemma we show that mean curvatures of an Ancient Oval solution 
satisfying the assumptions of Theorem \ref{thm-rot-symm},   around the tip points on 
$\Omega_t^k$, for a fixed $k=1,2$,  are uniformly equivalent in a quantitative way.

\begin{lemma}
  \label{lem-unif-equiv-curv}
  Let $M_t$, $t \in (-\infty,0)$, be an Ancient Oval satisfying the assumptions of Theorem \ref{thm-rot-symm} and fix $k=1,2$. 
  For every $L > 0$ there exist uniform constants $c >0 $, $C < \infty$ and $t_0 \ll -1$ so that for all $t \le t_0$ we have
  \begin{equation}
    \label{eq-unif-equiv-curv}
    c \, H(p_t^k,t) \le H(x,t) \le C\, H(p_t^k,t) \quad  \mbox{{\em if}} \,\,  |x - p_t^k| < \frac{L}{H(x,t)}, \, x \in \Omega_t^k.
  \end{equation}
  We may chose $c, C$ so that \eqref{eq-unif-equiv-curv}  holds for both $k=1,2$. 
\end{lemma}

\begin{proof}
  Let us take without loss of generality $k = 1$.
  First lets show the estimate from below.
  Assume the statement is false.
  This implies there exist a sequence of times $t_j\to -\infty$ and a sequence of constants $C_j\to \infty$ so that 
  \begin{equation}
    \label{eq-violate}
    H(p_{t_j}^1,t_j) \ge C_j\, H(x_j,t_j) \qquad \forall \,\, j
  \end{equation}
  for some $x_j \in \Omega_{t_j}^1$ such that the $|x_j - p_{t_j}^1| < \frac{L}{H(x_j,t_j)}$.
  Rescale the flow around $(x_j,t_j)$ by $Q_j := H(x_j,t_j)$.
  By the global convergence theorem 1.12 in \cite{HK}, the sequence of rescaled flows subconverges uniformly on compact sets to an ancient noncollapsed solution.
  Points $x_j$ get translated to the origin and points $p_{t_j}^1$ get translated to points $\tilde{p}_{t_j}^1$ under rescaling.
  Since by our assumption we have
  \[|0 - \tilde{p}_{t_j}^1| = H(x_j,t_j) \, |x_j - p_{t_j}^1| < L\]
  then due to uniform convergence of the rescaled flow on bounded sets we have
  \[H_j(\tilde{p}_{t_j}^1,0) \le C, \qquad j \ge j_0\]
  for a uniform constant $C < \infty$, that depends on $L$, but is independent of $j$.
  This implies
  \[ H(p_{t_j}^1,t_j) \le C\, H(x_j,t_j), \qquad j \ge j_0\]
  which contradicts \eqref{eq-violate}. 
  To prove the upper bound in \eqref{eq-unif-equiv-curv} note that the lower bound in \eqref{eq-unif-equiv-curv} that we have just proved implies $|x - p_t^1| \le \frac{L}{H(x,t)} \le \frac{L}{c\, H(p_t^1,t)}$.
  Hence, we can switch the roles of $x$ and $p_t^1$ in the proof above.
  This ends the proof of the Lemma.
\end{proof}

\begin{remark}
  Note that we can choose uniform $c >0$ and $t_0 \ll  -1$ so that the conclusion of Lemma \ref{lem-unif-equiv-curv} holds for both $k = 1$ and $k = 2$.
\end{remark}

Let $\epsilon > 0$ be a small number.
By our assumption the flow is $\alpha$-noncollapsed and uniformly 2-convex, meaning that \eqref{eqn-2convex} holds.
By the cylindrical estimate (\cite{HK}, \cite{HS}) we can find an $\eta = \eta(\epsilon, \alpha,\beta) > 0$ so that if the flow is defined in the normalized parabolic cylinder  $\hat \cP(x,t,\eta^{-1})$ and if
\[
\frac{\lambda_1}{H}(x,t) < \eta
\]
then the flow $M_t$ is $\epsilon$-close to a shrinking round cylinder $S^{n-1}\times \mathbb{R}$ near $(x,t)$.
Being $\epsilon$-close to a shrinking round cylinder near $(x,t)$ means that after parabolic rescaling by $H(x,t)$, shifting $(x,t)$ to $(0,0)$ and a rotation, the solution becomes $\epsilon$-close in the $C^{[\frac{1}{\epsilon}]}$-norm on $\cP(0,0,1/\epsilon)$ to the standard shrinking cylinder with $H(0,0) = 1$ (see for more details \cite{HK}).

\begin{proposition}
  \label{limit-tip}
  Fix a $k \in \{1,2\}$ and let $L > 0$ be any fixed constant.
  Let $M_t$ be an Ancient Oval that satisfies the assumptions of Theorem \ref{thm-rot-symm}.
  Then for any sequence of times $t_j\to -\infty$, and any sequence of points $x_j\in \Omega_{t_j}^k$ such that $|x_j - p_{t_j}^k| \le \frac{L}{H(x_j,t_j)}$, the rescaled limit around $(x_j,t_j)$ by factors $H(x_j,t_j)$ subconverges to a Bowl soliton.
\end{proposition}

In the course of proving this Proposition we need the following observation.

\begin{lemma}
  \label{lem-simple-lambda}
  For all $t \ll 0$ each of the two components $\Omega_t^j$ of $M_t \setminus B(0, \sqrt{8(n-1)})$ contains at least one point at which $\lambda_{\min}$ is not a simple eigenvalue.
\end{lemma}
\begin{proof}
  Suppose $\lambda_{\min}$ is a simple eigenvalue at each point on $\Omega_t^1$.
  Then the corresponding eigenspace defines a one dimensional subbundle of the tangent bundle $TM_t$.
  Since $M_t$ is simply connected any one dimensional bundle over $M_t$ is trivial, and thus has a section $v:M_t\to TM_t$ with $v(p)\neq p$ for all $p$.
  Within the region $\bar B(0, \sqrt{8(n-1)})$ the hypersurfaces $M_t$ converge in $C^2$ to a cylinder with radius $\sqrt{2(n-1)}$, so within this region $\lambda_{\min}$ is a simple eigenvalue, and the eigenvector $v(p)$ will be transverse to the boundary $\partial\Omega_t^1$.
  We may assume that it points outward relative to $\Omega_t^1$.
  
  The component $\Omega_t^1$ is diffeomorphic with the unit ball $B^n\subset\R^n$, and under this diffeomorphism the vector field $v:\Omega_t^1\to T\Omega_t^1$ is mapped to nonzero vector field $\tilde v:B^n\to\R^n$, which points outward on the boundary $S^{n-1} = \partial B^n$.
  The normalized map $\hat v = \tilde v/|\tilde v| : S^{n-1}\to S^{n-1}$ is therefore homotopic to the unit normal, i.e.~the identity map  $\mathrm{id} : S^{n-1}\to S^{n-1}$.
  Its degree must then equal $+1$, which is impossible because $\hat v$ can be extended continuously to $\hat v = \tilde v/|\tilde v| : B^n\to S^{n-1}$.
\end{proof}

\begin{proof}[Proof of Proposition \ref{limit-tip}]
  Without any loss of generality take $k = 1$ and let $\tilde{L} > 0$ be an arbitrary fixed constant.
  Let $t_j\to -\infty$ be an arbitrary sequence of times and let $x_j\in \Omega_{t_j}^1$ be an arbitrary sequence of points such that $|x_j - p_{t_j}^1| \le \frac{\tilde{L}}{H(x_j,t_j)}$.
  Rescale our solution around $(x_j, t_j)$ by scaling factors $H(x_j,t_j)$.
  By Lemma \ref{lemma-possible-limits} we know that the sequence of our rescaled solutions subconverges to either a Bowl soliton or a round shrinking cylinder.
  If the limit is a Bowl soliton, we are done.
  Hence, assume the limit is a shrinking round cylinder, which is a situation we want to rule out.
  By Lemma \ref{lem-unif-equiv-curv} we have that for $j$ large enough, curvatures $H(p_{t_j}^1,t_j)$ and $H(x_j,t_j)$ are uniformly equivalent.
  This together with  $|x_j - p_{t_j}^1| \le \frac{\tilde{L}}{H(x_j,t_j)}$ implies that if we rescale our solution around points $(p_{t_j}^1,t_j)$ by factors $H(p_{t_j}^1,t_j)$, after taking a limit we also get a shrinking round cylinder. 
  
  Since the limit around $(p_{t_j}^1,t_j)$ is a round shrinking cylinder, for every $\epsilon > 0$ there exists a $j_0$ 
  so that for $j \ge j_0$ we have ${\displaystyle \frac{\lambda_{\min}(p_{t_j}^1,t_j) }{H(p_{t_j}^1,t_j)} < \epsilon}$. 
  In the following two claims,  $p_{t_j}^1 \in \Omega_{t_j}^1$ will be a sequence of the tip points as above, such that the limit of the sequence of rescaled solutions around $(p_{t_j}^1,t_j)$ by factors $H(p_{t_j}^1,t_j)$ is a shrinking round cylinder.  
  
  In the first claim we show the ratio $\frac{\lambda_{\min}}{H}$ can be made arbitrarily small not only at points $p_{t_j}^1$, but also at all the points that are at bounded distances away from them. \smallskip
  
  \noindent{\it Claim:}  
  For every $\epsilon > 0$ and every $C_0 > 0$ there exists a $j_0$ so that for $j \ge j_0$ we have 
  \begin{equation}
    \label{eqn-tj99}
    \frac{\lambda_{\min}(p, t_j)}{H(p,t_j)} < \epsilon, \qquad \mbox{whenever} \qquad |p - p_{t_j}^1| \le \frac{C_0}{H(p,t_j)} \,\,\,\, \mbox{and} \,\,\,\,  p\in \Omega_{t_j}^1.
  \end{equation}

  \begin{proof}[Proof of Claim]
    Assume the claim is not true, meaning there exist constants $\epsilon > 0$, $C_0 > 0$, a subsequence which we still denote by $t_j$ and points $p_j \in \Omega_{t_j}^1$ so that 
    \begin{equation}
      \label{eq-contra100}
      |p_j - p_{t_j}^1| \le \frac{C_0}{H(p_j,t_j)} \qquad \mbox{but} \qquad \frac{\lambda_{\min}(p_j,t_j)}{H(p_j,t_j)} \ge \epsilon.
    \end{equation}
    Consider the sequence of rescaled flows around $(p_j,t_j)$ by factors $H(p_j,t_j)$.
    Lemma \ref{lemma-possible-limits} and the second inequality in \eqref{eq-contra100} imply the above sequence subconverges to a Bowl soliton.
    On the other hand, since $|p_j - p_{t_j}^1| \le \frac{C_0}{H(p_j,t_j)}$, by Lemma \ref{lem-unif-equiv-curv}, the curvatures $H(p_j,t_j)$ and $H(p_{t_j}^1,t_j)$ are uniformly equivalent.
    This together with our assumption on $(p_{t_j}^1,t_j)$ and the first inequality in \eqref{eq-contra100}  yield the rescaled sequence around $(p_j,t_j)$ by factors $H(p_j,t_j)$ subconverges to a round shrinking cylinder at the same time and hence we get contradiction.
    This proves the Claim.
  \end{proof}
  
  Next we claim that for sufficiently big $j$, even far away from the tip points $p_{t_j}^1$ we see the cylindrical behavior.
  Assume $\bar{L}$ is big enough so that the conclusion of Lemma \ref{lemma-cylindrical} holds.
  The immediate consequence of the Lemma \ref{lemma-cylindrical}  is that for every $\epsilon > 0$ there exists a $j_0$ so that for $j \ge j_0$  we have
  \begin{equation}\label{eqn-tj100}
    \frac{\lambda_{\min}(p,t_j)}{H(p,t_j)} < \epsilon, \qquad \mbox{whenever} \qquad p \in \Omega_{t_j}^1 \,\,\,\, \mbox{and} \,\,\,\, |p - p_{t_j}^1| \ge \frac{\bar{L}}{H(p,t_j)}
  \end{equation}

  We continue proving Proposition \ref{limit-tip}. \\ 
  
  %
  Estimates \eqref{eqn-tj99} after taking $C_0 = \bar{L}$ and \eqref{eqn-tj100} yield for every $\epsilon > 0$ there exists a $j_0$ so that for $j \ge j_0$,
  \begin{equation}
    \label{eq-cond-cyl}
    \frac{\lambda_{\min}}{H}(p,t_j) < \epsilon, \qquad \mbox{on all of} \qquad  \Omega_{t_j}^1.
  \end{equation}
  By the cylindrical estimate (\cite{HS}, \cite{HK}) we have that for every $\epsilon > 0$ there 
  exists a $j_0$ so that for $j \ge j_0$.
  \begin{equation}
    \label{eq-cross-sphere}
    \frac{|\lambda_p - \lambda_q|}{H}(p,t_j) < \epsilon, \qquad \mbox{for all} \qquad n \ge p, q \ge 2,
  \end{equation}
  on $\Omega_{t_j}^1$.
  
  For small enough $\epsilon>0$ the conditions~\eqref{eq-cond-cyl} and~\eqref{eq-cross-sphere} imply that $\lambda_{\min}$ is a simple eigenvalue, hence contradicting Lemma \ref{lem-simple-lambda}.
  This finishes the proof of Proposition \ref{limit-tip}.
\end{proof}

\begin{lemma}
  \label{lemma-tip-Bowl}
  Let $M_t$, $t \in (-\infty,0)$, be an Ancient Oval satisfying the assumptions of Theorem \ref{thm-rot-symm} and fix $k=1,2$. 
  Then for every $\epsilon > 0$ there exist uniform constants $\rho_0 < \infty$ and $t_0 \ll -1$, so that for every $t \le t_0$ we have that $\hat{\cP}(p_t^k,t,\rho_0)$ is $\epsilon$-close to a piece of a Bowl soliton that includes the tip. 
\end{lemma}

\begin{proof}
  First of all observe that by Proposition \ref{limit-tip} it is easy to argue that for every $\epsilon > 0$ and any $\rho_0 < \infty$ there exists a $t_0 \ll -1$ so that for $t \le t_0$, the parabolic cylinder $\hat{\cP}(p_t^k,t,\rho_0)$ is $\epsilon$-close to a piece of a Bowl soliton.
  The point of this lemma is to show that we can find $\rho_0$ big enough, but uniform in $t \le t_0 \ll -1$ so that the piece of the Bowl soliton above includes the tip. 
  
  To prove the statement we argue by contradiction.
  Assume the statement is not true, meaning there exist an $\epsilon > 0$, a sequence $\rho_j\to \infty$  and a sequence $t_j\to -\infty$ so that $\hat{\cP}(p_{t_j}^k,t_j,\rho_j)$ is $\epsilon$-close to a piece of Bowl soliton that does not include the tip.
  Rescale the solution around $(p_{t_j}^k,t_j)$ by factors $H(p_{t_j}^k,t_j)$.
  By Proposition \ref{limit-tip} we know that the rescaled solution subconverges to a piece of a Bowl soliton.
  Hence there exists a uniform constant $C_0$ so that the origin that lies on the limiting Bowl soliton and corresponds after scaling, to the points $(p_{t_j}^k,t_j)$,  is at distance $C_0$ from the tip of the soliton (which is the point of maximum curvature).
  This implies there exist points $q_{t_j} \in \Omega_{t_j}^k$ so that $|q_{t_j} - p_{t_j}^k| \le \frac{2C_0}{H(p_{t_j}^k,t_j)}$ for $j \ge j_0$, with the property that the points $q_{t_j}$ converge to the tip of the Bowl soliton.
  Furthermore,  for sufficiently big $j \ge j_0$, parabolic cylinders $\hat{\cP}(p_{t_j}^k,t_j,3C_0)$ are $\epsilon$-close to a piece of the Bowl soliton that includes the tip.
  That contradicts our assumption that for every $j$, $\hat{\cP}(p_{t_j}^k,t_j,\rho_j)$ is $\epsilon$-close to a piece of Bowl soliton that does not include its tip.
\end{proof}

Finally we show the crucial for our purposes proposition below which says that every point on $M_t$ has a parabolic neighborhood of uniform size, around which it is either close to a Bowl soliton or to a round shrinking cylinder. 
\begin{prop} 
  \label{lem-either-or}
  Let $M_t$ be an Ancient Oval that is uniformly two convex.
  Let  $\epsilon_0, \epsilon_1$, $L_0, L_1$ be the constants from Propositions \ref{prop-neck} and \ref{prop-cap} and let $\epsilon \le \min\{\epsilon_0,\epsilon_1\}$.
  Then, there exists $t_0 \ll -1$, depending on these constants, with the following property: for every $(\bar x, \bar t)$ with $\bar x \in M_{\bar t}$ and $\bar{t} \le t_0$,  either $\hat{\mathcal{P}} (\bar x, \bar t, L_0)$ lies
  on an $(\epsilon,10)$-neck or every point in  $\hat{\mathcal{P}} (\bar x, \bar t, L_1)$ is,  after scaling by 
  $H(\bar x,\bar t)$,  $\epsilon$-close in the $C^{20}$-norm to a piece of a Bowl soliton which includes the tip. 
\end{prop} 

\begin{proof}
  Recall that as a consequence of Hamilton's Harnack estimate our ancient solution satisfies $H_t \geq 0$.
  This implies there exists a uniform constant $C_0$ so that
  \begin{equation}
    \label{eq-H}
    \max_{M_t} H(\cdot, t) \le C_0, \qquad t \le t_0.
  \end{equation}
  Let $\bar{\epsilon} \ll \min\{\epsilon_0, \epsilon_1, L_0^{-1}\}$.
  For this $\bar{\epsilon} > 0$ find a $\delta = \delta(\bar{\epsilon})$ as in Theorem 1.19 in \cite{HK} (see also \cite{HS} for the similar estimate) so that if
  \begin{equation}
    \label{eq-delta}
    \frac{\lambda_{\min}}{H}(p,t) < \delta
  \end{equation}
  and the flow is defined in $\hat \cP(p,t,\delta^{-1})$, then the solution $M_t$ is $\bar{\epsilon}$-close to a round cylinder around $(p,t)$, in the sense that a rescaled flow by $H(p,t)$ around $(p,t)$ is $\bar{\epsilon}$-close on $\cP(0,0,{\bar{\epsilon}}^{-1})$ to a round cylinder with $H(0,0) = 1$. 
  Take $\delta > 0$ as in \eqref{eq-delta}.
  For this $\delta$ choose $\bar{L}$ sufficiently big and $t_0 \ll -1$ so that Lemma \ref{lemma-cylindrical} holds (after we take $\eta$ in the Lemma to be equal to $\delta$).
  
  \smallskip 
  Let $(\bar{x},\bar{t})$ be such that $\bar{x} \in M_{\bar{t}}$ and $\bar{t} < t_0$.
  Then either $\bar{x}\in M_{\bar{t}}\cap B_{8(n-1)\sqrt{|t|}}$,  or $\bar{x} \in \Omega_{\bar{t}}^1$, or $\bar{x} \in \Omega_{\bar{t}}^2$.
  In the first case that has been already discussed above, for $-\bar{t}$ sufficiently large, we know that $M_{\bar{t}}\cap B_{16(n-1)\sqrt{|\bar{t}|}}$ is neck-like and hence there exists $t_0 \ll -1$ so that for $t \le t_0$
  \[\max_{M_{\bar{t}}\cap B_{16(n-1)\sqrt{|\bar{t}|}}} \frac{\lambda_{\min}}{H} < \delta\]
  where $\delta$ is as in \eqref{eq-delta}.
  Thus every point $\bar{x}\in M_{\bar{t}}\cap B_{8(n-1)\sqrt{|t|}}$ has the property that every point in $\hat{\cP}(\bar{x},\bar{t},L_0)$ lies at the center of an $(\epsilon,10)$-neck.
  
  We may assume from now on, with no loss of generality, that $\bar{x}\in \Omega_{\bar{t}}^1$, since the discussion for $\bar{x}\in \Omega_{\bar{t}}^2$ is equivalent. 
  We either have  $|\bar{x} -  p_{\bar{t}}^1| \ge \frac{\bar{L}}{H(\bar{x},\bar{t})}$, or  $|\bar{x} - p_{\bar{t}}^1| \le \frac{\bar{L}}{H(\bar{x},\bar{t})}$. 
  In the first case,  Lemma \ref{lemma-cylindrical} gives that   ${\displaystyle \frac{\lambda_{\min}}{H}(\bar{x},\bar{t}) < \delta}.$ 
  As discussed above,  the cylindrical estimate then implies that the rescaled flow $H(\bar{x},\bar{t}) (F_{\bar{t} + H(\bar{x},\bar{t})^{-2} t} - \bar{x})$ is $\bar{\epsilon}$-close to the round cylinder with $H(0,0) = 1$, in a parabolic cylinder  $\cP(0,0,\bar{\epsilon}^{-1})$.
  It is straightforward then to conclude that every point in the normalized cylinder $\hat{\cP}(\bar{x},\bar{t},L_0,)$ lies on an $(\epsilon,10)$-neck, where we use that $L_0 \ll {\bar{\epsilon}}^{-1}$ and $\bar{\epsilon} \ll \epsilon$.
  
  \medskip 
  
  Assume now that  $\bar x \in \Omega_{\bar t}^1$ and $|\bar{x} - p_{\bar{t}}^1| \le \frac{\bar{L}}{H(\bar{x},\bar{t})}$. 
  Combining this with Lemma \ref{lem-unif-equiv-curv} and Lemma \ref{lemma-tip-Bowl}  yield we can find a sufficiently large but uniform constant $L_1$  and constant $t_0 \ll -1$, so that  for $\bar{t} \le t_0$ we have that $\hat{\cP}(\bar{x},\bar{t},L_1)$ is $\epsilon_1$-close to a piece of a Bowl soliton that also includes its tip.
\end{proof}

We can now conclude the proof of Theorem \ref{thm-rot-symm}. 

\begin{proof}[Proof of Theorem \ref{thm-rot-symm}]
  Let $L_0, L_1, \epsilon_0, \epsilon_1$ be chosen so that Propositions \ref{prop-neck} and \ref{prop-cap} hold.
  Let $\bar{\epsilon} \ll \epsilon := \min\ (\epsilon_0, \epsilon_1) $.
  Let $t_0 \ll -1$ be as in Proposition \ref{lem-either-or} so that for every $(\bar x, \bar t)$ with $\bar x \in M_{\bar t}$ and $\bar{t} \le t_0$, either $\hat P (\bar x, \bar t, L_0)$ lies
  on an $(\bar{\epsilon},10)$-neck (and hence on an $(\epsilon_0,10)$-neck, since $\bar{\epsilon} \le \epsilon_0$), or every point in $\hat P (\bar x, \bar t, L_1)$ is, after scaling,   $\bar{\epsilon}$-close in the $C^{20}$-norm to a piece of the Bowl soliton 
  which includes the tip (and hence is also $\epsilon_1$ close, since $\epsilon_1 \le \bar{\epsilon}$).
  Note that the axis of symmetry of this Bowl soliton may depend on the point $(\bar x, \bar t)$. 
  
  Above implies that every point $(\bar{x},\bar{t})$, for $\bar{x} \in M_{\bar{t}}$ and $\bar{t}\le t_0$, lies in a parabolic neighborhood of uniform size, that is after scaling, $\bar{\epsilon}$ close to a rotationally symmetric surface (either a round cylinder or a Bowl soliton).
  Hence, it follows that if we choose $\bar{\epsilon}$ sufficiently small relative to $\epsilon$, then $(\bar{x},\bar{t})$ is $\epsilon$-symmetric (defined as in Definition \ref{def-normalized}).
  After applying Propositions \ref{prop-neck} and \ref{prop-cap} we then conclude that  $(\bar{x},\bar{t})$ is $\frac{\epsilon}{2}$-symmetric, for all $\bar{x}\in M_{\bar{t}}$ and all $\bar{t} \le T$.
  Iterative application of Propositions \ref{prop-neck} and \ref{prop-cap} yields that $(\bar{x},\bar{t})$ is $\frac{\epsilon}{2^j}$-symmetric, for all $\bar{x}\in M_{\bar{t}}$, $\bar{t} \le t_0$  and all $j\geq 1$.
  Letting $j \to +\infty$ we finally conclude that  $M_t$ is rotationally symmetric for all $t \le t_0$ which also implies that $M_t$ is rotationally symmetric for all $t  \in (-\infty,0)$.
\end{proof}

\section{Outline of the proof of Theorem \ref{thm-main}} \label{sec-regions}

Since the proof of Theorem \ref{thm-main} is quite involved,
in this preliminary section we will give an outline of the main steps in the proof of the classification result in the presence of rotational symmetry.
Our method is based on a priori estimates for various distance functions between two given ancient solutions in appropriate coordinates and measured in weighted $L^2$ norms.
We need to consider two different regions: the {\em cylindrical} region and the {\em tip} region.
Note that the tip region will be divided in two sub-regions: the {\em collar} and the {\em soliton} region.
These are pictured in Figure~\ref{fig-three-regions} below.
In what follows, we will define these regions, review the equations in each region and define appropriate weighted $L^2$ norms with respect to which we will prove coercive type estimates in the subsequent sections.
At the end of the section we will give an outline of the proof of Theorem \ref{thm-main}.

\smallskip Let $M_{1}(t), M_2(t)$  be two rotationally symmetric ancient oval solutions satisfying the assumptions of Theorem \ref{thm-main}.
Being surfaces of rotation, they are each determined by a function $U = U_i(x, t)$, ($i=1,2$), which satisfies the equation
\begin{equation}
  \label{eq-MCF}
  U_t = \frac{U_{xx}} {1+U_x^2} - \frac{n-1} {U}.
\end{equation}

In the statement of Theorem \ref{thm-main} we claim the uniqueness of  any two Ancient Ovals  up to dilations and translations.
In fact since equation \eqref{eq-MCF} is invariant under translation in time, translation in space and also under cylindrical dilations in space-time, each solution $M_i(t)$ gives rise to a three parameter family of solutions
\begin{equation}
  \label{eq-two-parameters}
  M_i^{\alpha \beta \gamma}(t) = e^{\gamma/2} \, \varPhi_{\alpha}( M_i(e^{-\gamma}(t - \beta))),
\end{equation}
where $\varPhi_{\alpha}$ is a rigid motion, that is just the translation of the hypersurface along $x$ axis by value $\alpha$.
The theorem claims the following: \emph{ given two ancient oval solutions we can find $\alpha, \beta, \gamma$ and $t_0 \in \R$ such that
\[
M_1(t) = M_2^{\alpha\beta\gamma}(t), \qquad \mbox{for} \,\, t \leq t_0.
\]}
The profile function $U_i^{\alpha\beta\gamma}$ corresponding to the modified solution $M_i^{\alpha\beta\gamma}(t)$ is given by
\begin{equation}
  \label{eq-Ualphabeta}
  U_i^{\alpha\beta\gamma} (x, t) = e^{\gamma/2} U_i\Bigl (e^{-\gamma/2}(x-\alpha), e^{-\gamma} (t - \beta) \Bigr).
\end{equation}
We rescale the solutions $M_i(t)$ by a factor $\sqrt{-t}$ and introduce a new time variable $\tau = -\log(-t)$, that is, we set
\begin{equation}
  \label{eq-type-1-blow-up}
  M_i(t) = \sqrt{-t} \, \bar M_i(\tau), \qquad \tau:= - \log (-t).
\end{equation}
These are again $O(n)$ symmetric with profile function $u$, which is related to $U$ by
\begin{equation}
  \label{eq-cv1}
  U(x,t) = \sqrt{-t}\, u(y, \tau), \qquad y=\frac x{ \sqrt{-t}}, \quad \tau=-\log (-t).
\end{equation}
If the $U_i$ satisfy the MCF equation \eqref{eq-MCF}, then the rescaled profiles $u_i$ satisfy~\eqref{eq-u}, i.e.
\[
\frac{\pd u}{\pd \tau} = \frac{u_{yy}}{1+u_y^2} - \frac y2 \, u_y - \frac{n-1}{u}+ \frac u2.
\]
Translating and dilating the original solution $M_i(t)$ to $M_i^{\alpha\beta\gamma}(t)$ has the following effect on $u_i(y,\tau)$:
\begin{equation}
  \label{eq-ualphabeta}
  u_i^{\alpha\beta\gamma}(y, \tau) = \sqrt{1+\beta e^{\tau}} u_i\Bigl( \frac{y-\alpha e^{\frac{\tau}{2}}} {\sqrt{1+\beta e^\tau} }, \tau+\gamma-\log\bigl(1+\beta e^\tau\bigr) \Bigr).
\end{equation}

To prove the uniqueness theorem we will look at the difference $U_1 - U_2^{\alpha\beta\gamma}$, or equivalently at $u_1 - u_2^{\alpha\beta\gamma}$.
The parameters $\alpha, \beta, \gamma$  will be chosen so that the projections of $u_1 - u_2^{\alpha\beta\gamma}$ onto positive eigenspace (that is spanned by two independent eigenvectors) and zero eigenspace of the linearized operator $\cL$ at the cylinder are equal to zero at time $\tau_0$, which will be chosen sufficiently close to $-\infty$. 
Correspondingly, we denote the difference $U_1 - U_2^{\alpha\beta\gamma}$ by $U_1-U_2$ and $u_1 - u_2^{\alpha\beta\gamma}$ by $u_1-u_2$.  
What we will actually observe is that the parameters $\alpha, \beta$ and $\gamma$  can be chosen to lie in a certain range, which allows our main estimates to hold without having to keep track of these parameters during the proof.
In fact, we will show in Section \ref{sec-conclusion} that for a given small $\epsilon >0$ there exists $\tau_0 \ll -1 $ sufficiently negative for which we have
\begin{equation}
  \label{eq-asymp-par1}
  \alpha \le \epsilon \frac{e^{-\tau_0/2}}{|\tau_0|}, \qquad   \beta \leq \epsilon \, \frac{e^{-\tau_0}}{|\tau_0|},  \qquad \gamma \leq \epsilon \, |\tau_0|
\end{equation}
and our estimates hold for $(u_1 - u_2^{\alpha\beta\gamma})(\cdot,\tau)$, $\tau \leq \tau_0$.
This inspires the following definition.
\begin{definition}[Admissible triple of parameters $(\alpha, \beta,\gamma)$]
  \label{def-admissible}
  We say that the triple of parameters $(\alpha, \beta,\gamma)$ is admissible with respect to time $\tau_0$ if they satisfy \eqref{eq-asymp-par1}.
\end{definition}

\begin{figure}
  \includegraphics{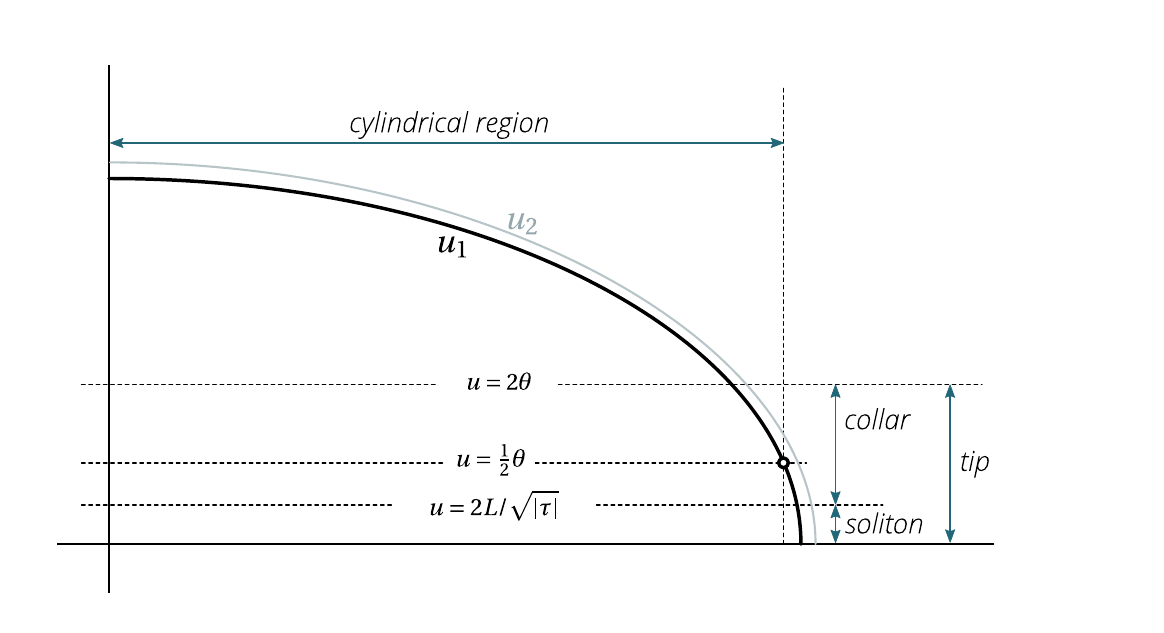}
  \caption{\textbf{The three regions. }  The \emph{cylindrical region} consists of all points with $u_1(y, \tau)\geq \frac12\theta$; the \emph{tip region} contains all points with $u_1(y, \tau)\leq 2\theta$ and is subdivided into the \emph{collar}, in which $u_1\geq 2L/\sqrt{|\tau|}$, and the \emph{soliton region}, where $u_1\leq 2L/\sqrt{|\tau|}$. }
  \label{fig-three-regions}
\end{figure}

\medskip
We will next define different regions and outline how we treat each region.

\subsection{The cylindrical region}\label{subsec-cylindrical}
For a given $\tau \leq \tau_0$ and constant $\theta$ positive and small, the cylindrical region is defined by
\[
\cyl_{\theta} = \bigl\{(y, \tau) : u_1(y,\tau) \ge \frac{\theta}{2} \bigr\}
\]
(see Figure~\ref{fig-three-regions}).
We will consider in this region a \emph{cut-off function $\varphi_\cyl(y,\tau)$} with the following properties:
\[
(i) \,\, \supp \varphi_\cyl \Subset \cyl_{\theta} \qquad (ii)\,\, 0 \leq \varphi_\cyl \leq 1 \qquad (iii) \,\, \varphi_\cyl \equiv 1 \mbox{ on } \cyl_{2\theta}.
\]
The solutions $u_i$, $i=1,2$, satisfy equation \eqref{eq-u}.
Setting
\[
w:=u_1-u_2^{\alpha\beta\gamma} \qquad \text{ and } \qquad w_\cyl:= w\, \varphi_\cyl
\]
we see that $w_\cyl$ satisfies the equation
\begin{equation}\label{eqn-w100}
  \frac{\pd}{\pd\tau} w_\cyl = \cL[w_\cyl] + \cE[w,\varphi_\cyl]
\end{equation}
where the operator $\cL$ is given by
\[
\cL = \partial_y^2 - \frac y2 \partial_y + 1
\]
and where the error term $\cE$ is described in detail in Section~\ref{sec-cylindrical}.
We will see that
\[
\cE[w,\varphi_\cyl] = \cE(w_\cC) + \bar \cE[w,\varphi_\cyl]
\]
where $\cE(w_\cC) $ is the error introduced due to the nonlinearity of our equation and is given by \eqref{eq-E10} and $\bar \cE[w,\varphi_\cyl]$ is the error introduced due to the cut off function $\varphi_\cC$ and is given by \eqref{eq-bar-E} (to simplify the notation we have set $u_2:=u_2^{\alpha\beta\gamma}$).  \smallskip

The differential operator $\cL$ is a well studied self-adjoint operator on the Hilbert space $\hilb := L^2(\R, e^{-y^2/4}dy)$ with respect to the norm and inner product
\begin{equation}\label{eqn-normp0}
  \|f\|_{\hilb}^2 = \int_\R f(y)^2 e^{-y^2/4}\, dy, \qquad \langle f, g \rangle = \int_\R f(y)g(y) e^{-y^2/4}\, dy.
\end{equation}
\smallskip

We split $\hilb$ into the unstable, neutral, and stable subspaces $\hilb_+$, $\hilb_0$, and $\hilb_-$, respectively.
The unstable subspace $\hilb_+$ is spanned by all eigenfunctions with positive eigenvalues (in this case $\hilb_+$ is spanned by a constant function equal to $\psi_0 = 1$, that corresponds to eigenvalue 1 and by a linear function $\psi_1 = y$, that corresponds to eigenvalue $\frac 12$, that is, $\hilb_+$ is two dimensional).
The neutral subspace $\hilb_0$ is the kernel of $\cL$, and is the one dimensional space spanned by $\psi_2 = y^2-2$.
The stable subspace $\hilb_-$ is spanned by all other eigenfunctions.
Let $\pr_\pm$ and $\pr_0$ be the orthogonal projections on $\hilb_\pm$ and $\hilb_0$.  \smallskip

For any function $f : \R \times (-\infty,\tau_0]$, we define the cylindrical norm
\[
\|f\|_{\hilb,\infty}(\tau) = \sup_{\sigma \le \tau} \Bigl(\int_{\sigma - 1}^{\sigma} \|f(\cdot,s)\|_{\hilb}^2\, ds\Bigr)^{\frac12}, \qquad \tau \leq \tau_0
\]
and we will often simply set
\begin{equation}
  \label{eqn-normp}
  \|f\|_{\hilb,\infty}:= \|f\|_{\hilb,\infty}(\tau_0).
\end{equation}
\medskip

In the course of proving necessary estimates in the cylindrical region we define yet another Hilbert space $\hv$ by
\[
\hv = \{f\in\hilb : f, f_y \in \hilb\},
\]
equipped with a norm
\[
\|f\|_\hv^2 = \int_{\R} \{f(y)^2 + f'(y)^2\} e^{-y^2/4}dy.
\]
We will write
\[
( f, g )_\hv = \int_\R \{f'(y)g'(y) + f(y)g(y)\} e^{-y^2/4} dy,
\]
for the inner product in $ \hv $.

Since we have a dense inclusion $ \hv \subset \hilb $ we also get a dense inclusion $ \hilb \subset \hv^* $ where every $ f\in\hilb $ is interpreted as a functional on $ \hv $ via
\[
g\in \hv \mapsto \langle f, g \rangle.
\]
Because of this we will also denote the duality between $ \hv $ and $ \hv^* $ by
\[
(f, g) \in \hv\times \hv^* \mapsto \langle f, g \rangle .
\]
Similarly as above define the cylindrical norm
\[
\|f\|_{\hv,\infty}(\tau) = \sup_{\sigma\le\tau}\Bigl(\int_{\sigma-1}^{\sigma} \|f(\cdot,s)\|_{\hv}^2\, ds\Bigr)^{\frac12},
\]
and analogously we define the cylindrical norm $\|f\|_{\hv^*,\infty}(\tau)$ and set for simplicity $\| f\|_{\hv^*,\infty}:=\|f\|_{\hv^*,\infty}(\tau_0)$.

\smallskip In Section \ref{sec-cylindrical} we will show a coercive estimate for $w_\cyl$ in terms of the error $E[w, \varphi_\cyl]$.
However, as expected, this can only be achieved by removing the projection $\pr_0 w_\cyl$ onto the kernel of $\cL$, generated by $\psi_2$.
More precisely, setting
\[
\hat w_\cyl := \pr_+ w_\cyl + \pr_- w_\cyl = w_\cyl - \pr_0 w_\cyl
\]
we will prove that for any $\epsilon >0$ there exist $\theta > 0$ and $\tau_0 \ll 0$ such that the following bound holds
\begin{equation}\label{eqn-cylindrical100}
  \| \hat w_\cyl \|_{\hv,\infty} \leq C \, \|E[w,\varphi_\cyl]\|_{\hv^*,\infty}
\end{equation}
provided that $\pr_+ w_\cyl(\tau_0) =0$.
In fact, we will show in Proposition \ref{lem-rescaling-components-zero} that the parameters $\alpha, \beta$ and $\gamma$ can be adjusted so that for $w^{\alpha\beta\gamma}:= u_1 - u_2^{\alpha\beta\gamma}$, we have
\begin{equation}\label{eqn-projections}
  \pr_+ w_\cyl(\tau_0) =0 \qquad \mbox{and} \qquad \pr_0 w_\cyl(\tau_0) =0.
\end{equation}
Thus \eqref{eqn-cylindrical100} will hold for such a choice of $\alpha, \beta, \gamma$ and $\tau_0 \ll 0$.
The condition $\pr_0 w_\cyl(\tau_0) =0$ is essential and will be used in Section \ref{sec-conclusion} to give us that $w^{\alpha\beta\gamma} \equiv 0$.
In addition, we will show in Proposition \ref{lem-rescaling-components-zero} that $\alpha$, $\beta$ and $\gamma$ can be chosen to be admissible according to our Definition \ref{def-admissible}.

The norm of the error term $\|E[w,\varphi_\cyl]\|_{\hv^*,\infty}$ on the right hand side of \eqref{eqn-cylindrical100} will be estimated in Section \ref{sec-cylindrical}, Lemmas \ref{lem-error1-est} and \ref{lem-error-bar}.
We will show that given $\epsilon >0$ small, there exists a  $\tau_0 \ll -1$ such that
\begin{equation}\label{eqn-error-111}
  \|E[w,\varphi_\cyl]\|_{\hv^*,\infty} \leq \epsilon \, \big ( \| w_\cC \|_{\hv,\infty} + \| w\, \chi_{D_\theta}\|_{\hv, \infty} \big ).
\end{equation}
where $D_\theta:= \{ (y,\tau): \,\, \theta/2 \leq u_1(y,\tau) \leq \theta \}$ denotes the support of the derivative of $\varphi_{\cC}$.
Combining \eqref{eqn-cylindrical100} and \eqref{eqn-error-111} yields the bound
\begin{equation}\label{eqn-cylindrical1}
  \| \hat w_\cyl \|_{\hv,\infty} \le \epsilon \, (\|w_\cyl\|_{\hv,\infty} + \| w\, \chi_{D_\theta} \|_{\hilb,\infty}),
\end{equation}
holding for all $\epsilon >0$ and $\tau_0 := \tau_0(\epsilon) \ll -1$.

\smallskip To close the argument we need to estimate $\| w \, \chi_{D_\theta}\|_{\hilb,\infty}$ in terms of $\|w_\cyl \|_{\hv,\infty}$.
This will be done by considering the tip region and establishing an appropriate \emph{a priori} bound for the difference of our two solutions there.

\subsection{The tip region}\label{subsec-tip}
The tip region is defined by
\[
\tip_{\theta} = \{(u, \tau):\,\, u_1 \le 2\theta, \, \tau \leq \tau_0\}
\]
(see Figure~\ref{fig-three-regions}).
In the tip region we switch the variables $y$ and $u$ in our two solutions, with $u$ becoming now an independent variable.
Hence, our solutions $u_1(y,\tau) $ and $u_2^{\alpha\beta\gamma}(y,\tau)$ become $Y_1(u,\tau)$ and $Y_2^{\alpha\beta\gamma}(u,\tau)$.
In this region we consider a {\em cut-off function $\varphi_T(u)$} with the following properties:
\begin{equation}\label{eqn-cutofftip}
  (i) \,\, \supp \varphi_T \Subset \tip_{\theta} \qquad (ii) \,\, 0 \leq \varphi_T \leq 1 \qquad (iii) \,\, \varphi_T \equiv 1, \,\, \mbox{on} \,\, \tip_{\theta/2}.
\end{equation}

Both functions $Y_1(u, \tau), Y_2^{\alpha\beta\gamma}$ satisfy the equation
\begin{equation}
  \label{eqn-Y}
  Y_\tau =\frac{Y_{uu}}{1+Y_u^2} + \frac{n-1} {u} Y_u + \frac{1} {2}\bigl(Y-uY_u \bigr ).
\end{equation}
It follows from \eqref{eqn-Y} that the difference $W:= Y_1 - Y_2^{\alpha\beta\gamma}$ satisfies
\begin{equation}
  \label{eqn-WW}
  W_\tau = \frac{W_{uu}} {1+ Y_{1u}^2} +\Bigl(\frac{n-1} {u} - \frac{u} {2} + D \Bigr) \, W_u  + \frac{1} {2} W,
\end{equation}
where
\[
D := -\frac{(Y_2^{\alpha\beta\gamma})_{uu}\, (Y_{1u} + Y_{2u}^{\alpha\beta\gamma})}{(1 + (Y_{1u})^2\, (1 + (Y_{2u}^{\alpha\beta\gamma})^2)}.
\]

Our next goal is to define an appropriate weighted $L^2$ norm
\[
\| W(\tau) \| := \Bigl ( \int_0^\theta W^2(u,\tau) \, e^{\mu(u,\tau)} \, du \Bigr )^{1/2}
\]
in the tip region $\tip_\theta$, by defining the weight $\mu(u,\tau)$.
To this end we need to further distinguish between two regions in $\tip_{\theta}$: for $L >0$ sufficiently large to be determined in Section \ref{sec-tip}, we define the {\it collar} region to be the set
\[
\collar_{\theta,L} := \Bigl\{y \mid \frac{L}{\sqrt{|\tau|}} \le u_1 (y,\tau) \le 2\theta\Bigr\}
\]
and the {\it soliton} region to be the set
\[
\mathcal{S}_L := \Bigl\{y \mid 0 \le u_1(y,\tau) \le \frac{L}{\sqrt{|\tau|}}\Bigr\}
\]
(see Figure~\ref{fig-three-regions}).
It will turn out later that one can regard the term $D$ in \eqref{eqn-WW} as an error term in $\collar_{\theta,L}$ (since in this region $D$ can be made arbitrarily small for $\tau_0 \ll -1$ and in addition by choosing $\theta, L$ appropriately).
On the contrary, $D$ is not necessarily small in the entire {\it soliton} region $\mathcal{S}_L$ and hence its approximation needs to be taken as a part of the linear operator.

\smallskip

The soliton region is the set where our asymptotic result in Theorem \ref{thm-old} implies that the solutions $Y_1$ and $Y_2$ are very close to the Bowl soliton (after re-scaling).
The collar is the transition region between the cylindrical region and the soliton region.
Having this in mind we define the weight $\mu(u,\tau)$ on the collar region $\collar_{\theta,L}$ to be
\[
\mu(u,\tau) = - \frac{1}{4} Y_1^2(u,\tau), \qquad u \in\collar_{\theta,L}
\]
which is in correspondence (after our coordinate switch) with the Gaussian weight $e^{-y^2/4}$ which we use in the cylindrical region.

In the soliton region we will define our weight $\mu(u,\tau)$ using the Bowl soliton.
In fact, we center the solution $Y_1$ at the tip and zoom in to a length scale ${1}/{\sqrt{|\tau|}}$ by setting
\begin{equation}
  \label{eqn-Y-expansion}
  Y_1(u,\tau) = Y_1(0, \tau) + \frac{1}{\sqrt{|\tau|}}\,Z_1 \Bigl(\rho, \tau\Bigr), \qquad \rho:= u\, \sqrt{|\tau|}.
\end{equation}

By Corollary \ref{cor-old}, $Z_1(\rho,\tau)$ converges, as $\tau\to -\infty$, to the unique rotationally symmetric, translating Bowl solution $Z_0(\rho)$ with speed $\sqrt2/2$, which satisfies
\begin{equation}
  \frac{Z_{0\rho\rho}} {1+Z_{0\rho}^2} + \frac{n-1} {\rho} Z_{0\rho} + \frac12 \sqrt{2} = 0,\qquad Z_0 (0) = Z_0 '(0) = 0.
  \label{eq-Z-bar}
\end{equation}

Since $n>1$ these equations determine $Z_0$ uniquely.
For large and small $\rho$ the function $Z_0(\rho)$ satisfies
\begin{equation}
  \label{eq-Z0-asymptotics}
  Z_0(\rho) =
  \begin{cases}
    -\sqrt{2}\rho^2/4(n-1) + \cO(\log \rho) & \rho\to\infty \\[2pt]
    -\sqrt{2}\rho^2/4n +\cO(\rho^4) & \rho\to0.
  \end{cases}
\end{equation}
These expansions may be differentiated.

In order to motivate the choice for the weight in the soliton region, we formally approximate $Y_{1, u}$ in equation \eqref{eqn-WW} by ${Z_0}_\rho$ using \eqref{eqn-Y-expansion} and the convergence to the soliton.
Using also the change of variables
\[
\bW(\rho, \tau) = \frac 1{\sqrt{|\tau|}} \, W(u, \tau), \qquad \rho:= u\, \sqrt{|\tau|},
\]
it follows from \eqref{eqn-WW} that $\bW$
\[
\Bigl( \frac{\bW_\rho} {1+Z_{0\rho}^2} \Bigr)_\rho + \Bigl(\frac{n-1} {\rho} \Bigr)\bW_\rho = E[\bW],
\]
where $E[\bW]$ is the error term.

This prompts us to introduce the linear differential operator
\[
\cM = \frac{d} {d\rho} \Bigl ( \frac{1} {1+Z_{0\rho}^2} \frac{d} {d\rho} \Bigr ) + \frac{n-1} {\rho}\frac{d} {d\rho},
\]
which we can write as
\[
\cM[\phi] = e^{-m} \frac{d} {d\rho}\Bigl\{ \frac{e^{m}} {1+Z_{0\rho}^2} \, \frac{d\phi} {d\rho}\Bigr\},
\]
where $m:\R_+\to\R$ is the function given by
\[
\begin{aligned}
  m(\rho) &= (n-1)\log \rho + \int_0^\rho \frac{n-1}{s} Z_{0}'(s)^2\, ds \\
  &= (n-1)\log \rho - \frac12\sqrt2 Z_0(\rho) -\frac12 \log \bigl(1+Z_0'(\rho)^2).
\end{aligned}
\]
The operator $\cM$ is symmetric and negative definite in the Hilbert space $ {\mathcal H} =L^2\Bigl(\R_+, e^{m(\rho)}\,d\rho\Bigr) $ in which the norm is given by
\[
\|f\|_{\cH}^2 = \int_0^\infty f(\rho)^2 \, e^{m(\rho)}\,d\rho.
\]

\smallskip We will use the variable $\rho$ in the proof of the Poincar\'e inequality in Proposition \ref{prop-Poincare}, while in our main estimate~\eqref{eqn-tip2} in the tip region we will bound $W(u,\tau)$ in an appropriate weighted norm, using the $u$ variable.
We would like to define our weight $\mu(u,\tau)$ in the soliton region $\mathcal{S}_L $ to be equal more or less to $m(\rho)$.
In order to make $\mu(u,\tau)$ to be a $C^1$ function in the whole tip region $\tip_{\theta}$ we will modify $m(\rho)$ in $\mathcal{S}_L$ by adding a linear correction term, setting
\[
\mu(u,\tau) = m(\rho) + a(L,\tau ) \rho + b(L,\tau), \qquad u \in \mathcal{S}_L, \,\, \rho = u\, \sqrt{|\tau|}.
\]

It follows from our discussion above we have the following definition for the weight $\mu(u,\tau)$ in the tip region $\tip_{\theta}$:
\begin{equation}
  \label{eq-weight}
  \mu(u,\tau) :=
  \begin{cases}
    -\frac14 Y_1^2(u,\tau), & \quad u \in\collar_{\theta,L} \\
    m(\rho) + a(L,\tau)\, \rho + b(L,\tau), & \quad u \in \mathcal{S}_L.
  \end{cases}
\end{equation}
The requirement that $\mu(u, \tau)$ be $C^1$ at $u=L/\sqrt{|\tau|}$ dictates the following choice of $a$ and $b$:
\begin{align}\label{eqn-aL}
  a(L,\tau)&:= - m'(L) - \frac{1}{2\sqrt{|\tau|}} \, Y_1 Y_{1u}\Bigr|_{u=L/\sqrt{|\tau|}}
  \\
  \label{eqn-bL}
  b(L,\tau) &:= -\frac{Y_1^2(L,\tau)}{4} - m(L) + L\, m'(L) + - \frac{L}{2\sqrt{|\tau|}} \, Y_1 Y_{1u}\Bigr|_{u=L/\sqrt{|\tau|}}.
\end{align}

For a function $W : [0, 2\theta]\to\R$ and any $\tau\leq \tau_0$, we define the weighted $L^2$ norm with respect to the weight $e^{\mu(\cdot,\tau)} \, du$ by
\[
\|W\|_\tau^2 := \int_0^{2\theta} W^2(u, \tau) \, e^{\mu(u,\tau)}\, du, \qquad \tau \leq \tau_0
\]
For a function $W:[0,2\theta]\times(-\infty, \tau_0] \to \R$ we define \emph{the cylindrical norms}
\[
\|W\|_{2,\infty,\tau} = \sup_{\tau' \le \tau}|\tau'|^{-1/4}\, \Bigl(\int_{\tau'-1}^{\tau'} \|W\|_s^2\, ds\Bigr)^{\frac12}
\]
for any $\tau\leq \tau_0$.
We include the weight in time $|\tau|^{-1/4}$ to make the norms equivalent in the transition region, between the cylindrical and the tip region, as will become apparent in Lemma \ref{prop-norm-equiv}.
We will also abbreviate
\[
\|W\|_{2,\infty}:= \|W\|_{2,\infty,\tau_0}.
\]

\smallskip

For a cutoff function $\varphi_T$ as in \eqref{eqn-cutofftip}, we set $W_T(u,\tau):= W(u,\tau) \, \varphi_T$.
We will see in Section \ref{sec-tip} that the following bound holds in the tip region
\begin{equation}
  \label{eqn-tip2}
  \| W_T \|_{\hilb,\infty} \leq \frac{C}{\sqrt{|\tau_0|}} \, \| W \chi \|_{2,\infty}
\end{equation}
where $\chi$ is the cut off function that is supported in the overlap between cylindrical and tip regions, for $\chi = 1$ on an open neighborhood of the support of $\partial_u\varphi_T$.

\subsection{The conclusion}
\label{subsec-conclusion}
The statement of Theorem \ref{thm-main} is equivalent to showing there exist parameters $\alpha, \beta$ and $\gamma$ so that $u_1(y,\tau) = u_2^{\alpha\beta\gamma}(y,\tau)$, where $u_2^{\alpha\beta\gamma}(y,\tau)$ is defined by \eqref{eq-ualphabeta} and both functions, $u_1(y,\tau)$ and $u_2^{\alpha\beta\gamma}(y,\tau)$, satisfy equation \eqref{eq-u}.
We set $w:= u_1 - u_2^{\alpha\beta\gamma}$, where $(\alpha,\beta,\gamma)$ is an admissible triple of parameters with respect to $\tau_0$, such that \eqref{eqn-projections} holds for a $\tau_0 \ll -1$.
Now for this $\tau_0$, the main estimates in each of the regions, namely \eqref{eqn-cylindrical1} and \eqref{eqn-tip2} hold for $w$.
Next, we want to combine \eqref{eqn-cylindrical1} and \eqref{eqn-tip2}.
To this end we need to show that the norms of the difference of our two solutions, with respect to the weights defined in the cylindrical and the tip regions, are equivalent in the intersection between the cylindrical and the tip regions, the so called {\em transition region}.
More precisely, we will show in Section \ref{sec-conclusion} that for every $\theta > 0$ small, there exist $\tau_0 \ll 0$ and uniform constants $c(\theta), C(\theta) > 0$, so that for $\tau\le \tau_0$, we have
\begin{equation}\label{eqn-normequiv1}
  c(\theta ) \, \| W \chi_{_{[\theta, 2\theta]}} \|_{\hilb,\infty} \leq \| w \, \chi_{_{D_{2\theta}}} \|_{\hilb,\infty} \leq C(\theta) \, \| W \chi_{_{[\theta, 2\theta]}} \|_{\hilb,\infty}
\end{equation}
where $\cD_{2\theta} := \{y\,\,\, |\,\,\, \theta \le u_1(y,\tau) \le 2\theta\}$ and $\chi_{[\theta,2\theta]}$ is the characteristic function of the interval $[\theta,2\theta]$.

Combining \eqref{eqn-normequiv1} with \eqref{eqn-cylindrical1} and \eqref{eqn-tip2} finally shows that in the norm $\| w_\cyl \|_{\hv,\infty}$ what actually dominates is $\| \pr_0 w_\cyl \|_{\hv,\infty}$.
We will use this fact in Section \ref{sec-conclusion} to conclude that $w(y,\tau):= w^{\alpha\beta\gamma}(y,\tau) \equiv 0$ for our choice of parameters $\alpha, \beta$ and $\gamma$.
To do so we will look at the projection $a(\tau):= \pr_0 w_\cyl$ and define its norm
\[
\|a\|_{\hilb,\infty}(\tau)
= \sup_{\sigma \le \tau} \Bigl(\int_{\sigma - 1}^{\sigma} \|a(s)\|^2\, ds\Bigr)^{\frac12},
\qquad
\tau \leq \tau_0
\]
with $\|a\|_{\hilb,\infty}:=\|a\|_{\hilb,\infty}(\tau_0)$.

By projecting equation \eqref{eqn-w100} onto the zero eigenspace spanned by $\psi_2$ and estimating error terms by $\| a\|_{\hilb,\infty}$ itself, we will conclude in Section \ref{sec-conclusion} that $a(\tau)$ satisfies a certain differential inequality which combined with our assumption that $a(\tau_0)=0$ (that follows from the choice of parameters $\alpha, \beta$, $\gamma$  so that \eqref{eqn-projections} hold) will yield that $a(\tau)=0$ for all $\tau \leq \tau_0$.
On the other hand, since $\| a\|_{\hilb,\infty} $ dominates the $\| w_\cyl \|_{\hilb,\infty}$, this will imply that $w_\cyl \equiv 0$, thus yielding $w \equiv 0$, as stated in Theorem \ref{thm-main}.

\begin{remark}
  \label{rem-symmetry}
  Note that our evolving hypersurface has  $O(n)$ symmetry and can be represented as in \eqref{eq-O1xOn-symmetry}.
  Due to asymptotics proved in Theorem \ref{thm-old}, when considering the tip region, it is enough to consider our solutions and prove the estimates only around $y = \bar{d}_1(\tau)$, where after switching the variables as in \eqref{eqn-Y-expansion}, we have $\rho \ge 0$.
  There we have $Z(\rho,\tau) \le 0$ and $Z_{\rho} \le 0$.
  We also have $Z_{\rho\rho} \le 0$, due to our convexity assumption.
  The estimates around $y = -\bar{d}_2(\tau)$ are similar.
\end{remark}

\section{A priori estimates}
Let $u(y,\tau)$ be an ancient oval solution of \eqref{eq-u} which satisfies the asymptotics in Theorem \ref{thm-old}.
In this section we will prove some further \emph{a priori} estimates on $u(y,\tau)$ which hold for $\tau \ll -1$.
These estimates will be used in the subsequent sections.
Throughout this section we will use the notation introduced in the previous section and in particular the definition of $Y(u,\tau)$ as the inverse function of $u(y,\tau)$ in the tip region and $Z(\rho,\tau)$ given by \eqref{eqn-Y-expansion}.

Before we start discussing \emph{a priori} estimates for our solution $u(y,\tau)$, we recall a corollary of Theorem \ref{thm-old} that will be used throughout the paper, especially in dealing with the tip region.

\begin{corollary}[Corollary of Theorem \ref{thm-old}]
  \label{cor-old}
  Let $M_t$ be any ancient oval satisfying the assumptions of Theorem \ref{thm-old}.
  Consider the tip region of our solution as in part (iii) of Theorem \ref{thm-old} and switch the coordinates around the tip region as in formula \eqref{eqn-Y-expansion}.
  Then, $Z(\rho,\tau)$ converges, as $\tau\to -\infty$, uniformly smoothly to the unique rotationally symmetric translating Bowl solution $Z_0(\rho)$ with speed $\sqrt2/2$.
\end{corollary}
\begin{proof}
  According to the asymptotic description of the tip-region from \cite{ADS} (see part (iii) of Theorem~\ref{thm-old}) the family of hypersurfaces that we get by translating the tip of $M_t$ to the origin and then rescaling so that the maximal mean curvature becomes equal to one, converges to the translating Bowl soliton with velocity equal to one.
  
  In defining $Z(\rho, \tau)$ by
  \[
  Y(u,\tau) = Y(0,\tau) + \frac{1}{\sqrt{|\tau|}}\, Z(\rho,\tau)
  \]
  we have in fact translated the tip to the origin, and rescaled the surface $M_t$, first by a factor $1/\sqrt{|t|} = e^{\tau/2}$ (the cylindrical rescaling \eqref{eq-type-1-blow-up} which leads to $u(y,\tau)$ or equivalently $Y(u,\tau)$,  and then by the factor $\sqrt{|\tau|}$ from \eqref{eqn-Y-expansion}.
  These two rescalings together shrink $M_t$ by a factor $\sqrt{|t| / \log|t|}$.
  Since by Theorem \ref{thm-old} the maximal mean curvature at the tip satisfies 
  \[
  H_{\max}(t) =  (1+o(1)) \sqrt{\frac {\log|t|}{2 |t|}}
  \]
  the hypersurface of rotation given by $z=Z(\rho, \tau)$ has maximal mean curvature 
  $ H_{\max} (t) \cdot \sqrt{{|t|}/{ \log |t|} } =  \sqrt 2/2 + o(1).$  It therefore converges to the unique rotationally symmetric, translating Bowl solution $Z_0(\rho)$ with speed $\sqrt2/2$, which satisfies equation \eqref{eq-Z-bar}.
\end{proof}

Next we prove a Proposition that will play an important role in obtaining the coercive type estimate \eqref{eqn-tip2} in the tip region.

\begin{prop}\label{pro-Sigurd} Let $u$ be an ancient oval solution of  \eqref{eq-u} which satisfies the assumptions and conclusion of
  Theorem \ref{thm-old}.
  Then, there exists $\tau_0 \ll -1$ for which we have $(u^2)_{yy}(y,\tau) \le 0$, for all $\tau \leq \tau_0$.
\end{prop}

The proof of this Proposition will combine a contradiction argument based on scaling and the following maximum principle lemma.
\begin{lemma}
  \label{lemma-help}
  Under the assumptions of Proposition \ref{pro-Sigurd}, there exists time $\tau_0 \ll -1$ such that
  \[ \max_{\bar M_\tau} (u^2)_{yy}(\cdot, \tau) >0 \quad \mbox{and} \quad \tau \leq \tau_0 \implies \frac {d}{d\tau} \max (u^2)_{yy} (\cdot, \tau) \leq 0.
  \]
\end{lemma}
\begin{proof}
  For the proof of this lemma it is more convenient to work in the original scaling $(x,t,U(x,t))$ (see equation \eqref{eq-u-original}) and is related to $(y,\tau,u(y,\tau))$ via the change of variables \eqref{eq-cv1}.
  Set $Q(x,t):=U^2(x,t)$.
  The inequality we want to show is scaling invariant,
  namely   $(U^2)_{xx}(x,t) = (u^2)_{yy}(y,\tau)$.
  Hence, it is sufficient to show that there exists $t_0 \ll -1$ such that
  \[
  \max_{M_t} Q_{xx}(\cdot, t) >0 \quad \mbox{and} \quad t \leq t_0 \implies \frac {d}{d t} \max_{M_t} Q_{xx} (\cdot, t) \leq 0.
  \]
  To this end, we will apply the maximum principle to the evolution of $Q_{xx}$.
  Since $U$ satisfies \eqref{eq-u-original}, a simple calculation shows that
  \[
  Q_t = \frac{ 4QQ_{xx} - 2Q_x^2 }{ 4Q + Q_x^2 } - 2(n-1).
  \]
  Differentiate this equation with respect to $x$ to get
  \begin{equation}
    \label{eq-qxpde}
    \begin{split}
      Q_{xt} &= \frac{4QQ_{xxx}}{4Q+Q_x^2}
      - (4QQ_{xx} - 2Q_x^2) \frac{4Q_x+2Q_xQ_{xx}}{(4Q+Q_x^2)^2} \\
      &= \frac{4QQ_{xxx}}{4Q+Q_x^2} -(4QQ_{xx}-2Q_x^2)(Q_{xx}+2) \frac{2Q_x}{(4Q+Q_x^2)^2}.
    \end{split}
  \end{equation}
  We differentiate again, but this time we only consider points where $Q_{xx}$ is either maximal or minimal, so that $Q_{xxx}=0$.
  Note that
  \begin{equation}
    (4QQ_{xx}-2Q_x^2)_x = 4QQ_{xxx} = 0 \qquad \text{and}\qquad (Q_{xx}+2)_x = Q_{xxx} = 0,
  \end{equation}
  at those points.
  Also,
  \begin{equation*}
    \begin{split}
      \Bigl(\frac{2Q_x}{(4Q+Q_x^2)^2}\Bigr)_x &=
      \frac{2Q_{xx}(4Q+Q_x^2) - 2(4Q_x+ 2 Q_xQ_{xx})(2Q_x)}{(4Q+Q_x^2)^3} \\
      &= 2\; \frac{(4Q-3Q_x^2) Q_{xx} - 8 Q_x^2}{(4Q+Q_x^2)^3} \\
      &= 2\frac{4Q-3Q_x^2}{(4Q+Q_x^2)^3} \Bigl(Q_{xx} - \frac{8Q_x^2}{4Q-3Q_x^2}\Bigr).
    \end{split}
  \end{equation*}
  Using these facts we now differentiate \eqref{eq-qxpde}.
  This leads us to
  \begin{equation*}
    \begin{split}
      Q_{xxt} &- \frac{4QQ_{xxxx}}{4Q+Q_x^2}\\
      &= - (Q_{xx}+2)(4QQ_{xx}-2Q_x^2) \cdot 2\frac{4Q-3Q_x^2}{(4Q+Q_x^2)^3} \Bigl(Q_{xx} - \frac{8Q_x^2}{4Q-3Q_x^2}\Bigr),
    \end{split}
  \end{equation*}
  holding at the maximal or minimal points of $Q_{xx}$.
  Recall that since $Q = U^2$, we have $Q_x^2 = 4U^2U_x^2$.
  Thus the previous equation becomes
  \begin{equation}
    \begin{split}
      \bigl(Q_{xx}\bigr)_t &- \frac{\bigl(Q_{xx}\bigr)_{xx}}{1+U_x^2}\\
      &= - \frac{2}{4Q} (Q_{xx}+2)(Q_{xx}-2U_x^2) \Bigl(Q_{xx} - \frac{8U_x^2}{1-3U_x^2}\Bigr)\frac{1-3U_x^2}{(1+U_x^2)^3}.
      \label{eq-qxxpde}
    \end{split}
  \end{equation}
  
  We will now use \eqref{eq-qxxpde} to conclude that at a maximum point of $Q_{xx}$, such that $Q_{xx} > 0$, we have
  \begin{equation}\label{eqn-Qxx}
    \bigl( Q_{xx}\bigr )_t - \frac{\bigl(Q_{xx}\bigr)_{xx}}{1+U_x^2} \leq 0.
  \end{equation}
  
  Since the equation becomes singular at the tip of the surface, we will
  first show that very near the tip we have $Q_{xx} <0$.
  After going to the $y$ variable and setting $q(y,\tau):=u^2(y,\tau)$,  we have $Q_{xx}=q_{yy}$, where after switching coordinates
  \begin{equation}\label{eqn-qyy10}
    q_{yy} = 2\, (u u_{yy} + u_y^2) = 2\Bigl( -u\,\frac{Y_{uu}}{Y_u^3} + \frac{1}{Y_u^2}\Bigr) = \frac{2}{Z_{\rho}^3}\, (Z_{\rho} - \rho Z_{\rho\rho}).
  \end{equation}
  Since by Corollary \ref{cor-old} we have that $Z(\rho,\tau)$ converges
  uniformly smoothly, as $\tau\to -\infty$, on the set $\rho \leq 1$, to the translating soliton $Z_0(\rho)$, it will be sufficient to show that ${\displaystyle \frac{2}{Z_{\rho}^3}\, ( {Z_0}_{\rho} - \rho {Z_0}_{\rho\rho}) <0}$ near $\rho=0$.
  Since $Z_0$ is a smooth function, this can be easily seen using the Taylor expansion of $Z_0$ near the origin.
  Let $Z_0(\rho) = a \, \rho^2 + b\, \rho^2 + o(\rho)$, as $\rho \to 0$.
  A direct calculation using \eqref{eq-Z-bar} shows that ${\displaystyle a=-\frac 1{2\sqrt{2}n}}$ and ${\displaystyle b=-\frac{\sqrt{2}}{16 n^3 (2+n)}}$, implying that
  \begin{equation}\label{eqn-Qxx10}
    \frac{2}{Z_{\rho}^3}\, ( {Z_0}_{\rho} - \rho {Z_0}_{\rho\rho}) = \frac 1{(2a\rho)^3} \, \frac{\sqrt{2} \, \rho^3}{ 2 n^3 (2+n)} + o(1) = - \frac{2}{2+n} + o(1)
  \end{equation}
  as $\rho \to 0$.
  We conclude that for $\tau \leq \tau_0 \ll -1$ and $\rho$ sufficiently close to zero we have
  \begin{equation}
    \label{eq-Qxx-neg}
    Q_{xx} =q_{yy} \leq - \frac{1}{2+n} < 0.
  \end{equation}
  
  We will now show that at a maximum point where $Q_{xx} > 0$, \eqref{eqn-Qxx} holds.
  By \eqref{eq-Qxx-neg} we know this point cannot be at the tip, and hence all derivatives are well defined at the maximum point of $Q_{xx}$.
  At such a point $Q_{xx}+2 >0$.
  We also have $Q_{xx}= 2UU_{xx}+2U_x^2$, so convexity of the surface implies $ Q_{xx}-2U_x^2 = 2UU_{xx} <0 $ on the entire solution.
  Thus, it is sufficient to show that when $Q_{xx} >0$,
  \begin{equation}\label{eqn-Qxx2}
    \Bigl ( Q_{xx}- \frac{8U_x^2}{1-3U_x^2} \Bigr)\frac{1-3U_x^2}{(1+U_x^2)^3} \leq 0
  \end{equation}
  holds.
  To this end, we will look at the two different cases, when $3U_x^2 <1$ or $3U_x^2 \geq 1$.
  When $3U_x^2 <1$, we also have ${\displaystyle \frac{ 8U_x^2 }{ 1-3U_x^2 } > 8 U_x^2 \geq 2 U_x^2}$, hence ${\displaystyle Q_{xx} - \frac{ 8U_x^2 }{ 1-3U_x^2 } < 0}$ implying that \eqref{eqn-Qxx2} holds.
  In the region where $3U_x^2 \geq 1$, we have ${\displaystyle Q_{xx} - \frac{ 8U_x^2 }{ 1-3U_x^2 } \geq 0}$, thus \eqref{eqn-Qxx2} holds as well.
  We conclude from both cases that at a maximum point where $Q_{xx} > 0$, \eqref{eqn-Qxx} holds.
  
\end{proof}

Let $Z_0$ be the translating Bowl soliton which satisfies \eqref{eq-Z-bar} and the asymptotics \eqref{eq-Z0-asymptotics}.
Recall that we have $Z_0(0) = (Z_0)_\rho(0) =0$, and the sign conventions $(Z_0)_\rho(\rho) < 0$ and $(Z_0)_{\rho\rho} (\rho) < 0$, for $\rho >0$ (see Remark \ref{rem-symmetry}), which also imply that $Z_0(\rho) < 0$ for $\rho >0$.
By Corollary \ref{cor-old} we have $\lim_{\tau \to -\infty} Z(\rho,\tau) =Z_0(\rho) $, smoothly on compact sets in $\rho$.
Thus \eqref{eqn-qyy10} implies that ${\displaystyle q_{yy} \sim \frac{2}{(Z_0)_{\rho}^3}\, ((Z_0)_{\rho} - \rho (Z_0)_{\rho\rho})}$ for $\tau \leq \tau_0 \ll -1$.
In the proof of the previous lemma we have shown that this quantity is negative near the origin $\rho=0$.
We will next show that it remains negative for all $\rho >0$.

\begin{lemma}\label{lemma-help2} On the translating Bowl soliton
  $Z_0(\rho)$ which satisfies equation \eqref{eq-Z-bar} we have
  \[
  \frac{2}{(Z_0)_{\rho}^3}\, \big ((Z_0)_{\rho} - \rho (Z_0)_{\rho\rho} \big )< 0
  \]
  for any $\rho \geq 0$.
\end{lemma}
\begin{proof}
  The proof simply follows from the maximum principle in a similar manner as the proof of Lemma \ref{lemma-help}.
  To use the calculations from before we need to flip the coordinates.
  Setting $x=Z_0(\rho)$, after we flip coordinates we have $\rho=U_0(x)$ for some function $U_0>0$.
  Since we have assumed above that $Z_0 \leq 0$, we also have that $x \leq 0$.
  Setting $Q:=U_0^2$ we find that ${\displaystyle Q_{xx} = \frac{2}{(Z_0)_{\rho}^3}\, \big ((Z_0)_{\rho} - \rho (Z_0)_{\rho\rho} \big )}$, hence it is sufficient to show that $Q_{xx} <0$ for $x < 0$.
  
  A direct calculation shows that $U_0$ satisfies the equation
  \[
  \frac{(U_0)_{xx}} {1+(U_0)_x^2} - \frac{n-1} {U} = \frac{\sqrt 2}{2} \, (U_0)_x.
  \]
  Note that in addition to $U >0$ for $x<0$, we have $(U_0)_x = 1/(Z_0)_\rho <0$ and $(U_0)_{xx} = - (Z_0)_{\rho\rho}/(Z_0)_\rho^3 <0$.
  Also since $(U_0)_x \to -\infty$ as $x \to 0$ the function $U_0$ fails to be a $C^1$ function near $x=0$.
  However this is not a problem since we have shown in the proof of the previous Lemma that \eqref{eqn-Qxx10} holds, implying that $Q_{xx} <0$ for $|x| \leq \eta$, if $\eta$ chosen sufficiently small.
  In addition a direct calculation where we use that $Z_0(\rho)$ satisfies the asymptotics
  \[
  Z_0(\rho) = -\frac{\rho^2}{2\sqrt{2}(n-1)} + \log\rho + o(\log\rho), \qquad \mbox{as}\,\,\,\rho\to \infty,
  \]
  as shown in by Proposition 2.1 in \cite{AV}, leads to
  \[
  Q_{xx} = \frac{2}{(Z_0)_{\rho}^3}\, \big ((Z_0)_{\rho} - \rho (Z_0)_{\rho\rho} < 0,
  \]
  for $\rho$ sufficiently large which is equivalent to $|x| \geq \ell $ with $\ell $ sufficiently large.
  
  We will now use the maximum principle to conclude that $Q_{xx} <0$ for $\eta < |x| < \ell$.
  Similarly to the computation in the proof of the previous lemma, after setting $Q:=U_0^2$, we find that
  \[
  \frac{ 4QQ_{xx} - 2Q_x^2 }{ 4Q + Q_x^2 } - 2(n-1) = \frac{\sqrt 2}{2} \, Q_x.
  \]
  After we differentiate twice in $x$,  following the same calculations as in the proof of Lemma \ref{lemma-help},  we find that $Q_{xx}$ satisfies the equation
  \[
  \begin{split}
    \frac{\sqrt 2}2 \, & \Bigl(Q_{xx}\Bigr)_x - \frac{\bigl(Q_{xx}\bigr)_{xx}}{1+(U_0)_x^2}\\
    &= - \frac{2}{4Q} (Q_{xx}+2)(Q_{xx}-2(U_0)_x^2) \Bigl(Q_{xx} - \frac{8 (U_0) _x^2}{1-3(U_0)_x^2}\Bigr)\frac{1-3(U_0)_x^2}{(1+(U_0)_x^2)^3}.
  \end{split}
  \]
  Assume that $Q_{xx}$ assumes a {\em positive} maximum at some point $x_0 \in [-\ell, -\eta]$.
  Arguing exactly as in Lemma \ref{lemma-help} we conclude that at a maximum point of $Q_{xx}$ where $Q_{xx} >0$, we have
  \[
  - \frac{2}{4Q} (Q_{xx}+2)(Q_{xx}-2U_x^2) \Bigl(Q_{xx} - \frac{8U_x^2}{1-3U_x^2}\Bigr)\frac{1-3U_x^2}{(1+U_x^2)^3} \leq 0.
  \]
  On the other hand at this point we also have that $Q_{xxx} =0$ and $Q_{xxxx} \leq 0$.
  If $Q_{xxxx}(x_0) <0$ at the maximum point $x_0$ we have reached a contradiction.
  If $Q_{xxxx}(x_0)=0$, then by replacing $Q_{xx}$ by $Q_{xx} - \epsilon (x-x_0)^2$ where $\epsilon = \epsilon (\eta, \ell) >0$ and sufficiently small, then $Q_{xx} - \epsilon (x-x_0)^2$ also attains its maximum at point $x_0$, where now
  \[
  \frac{\sqrt 2}2 \, \Bigl(Q_{xx} - \epsilon (x-x_0)^2 \Bigr)_x - \frac{\bigl(Q_{xx} - \epsilon (x-x_0)^2 \bigr)_{xx}}{1+(U_0)_x^2} >0
  \]
  leading again to contradiction.
  Hence, $Q_{xx}$ cannot achieve a positive maximum on $[-\ell,-\eta]$ finishing the proof of our lemma.
\end{proof}

For the purpose of the next lemma we consider $U(x,t)$ to be a solution to the unrescaled mean curvature flow equation \eqref{eq-u-original}.

\begin{lemma}\label{lemma-help3}
  If the hypersurface $M_{t_0}$ defined by $U(\cdot,t_0)$ encloses the interval $(x_0-2\ell, x_0+2\ell) \times \{0\}$, and if this interval is sufficiently long in the sense that
  \[
  \ell \geq \sqrt{2n+1}\,\, U(x_0, t_0),
  \]
  then
  \[
  U(x, t) + \ell| \, U_{x}(x, t)| \leq 8 \sqrt{2n+1} \, U(x_0, t_0),
  \]
  for all $x\in [x_0-\ell, x_0+\ell]$ and $t\in (t_0- U(x_0, t_0)^2, t_0]$.
\end{lemma}

\begin{proof}
  After translation in space and time, and after cylindrical rescaling of the solution $M_t$ we may assume that $U(x_0,t_0)=1$, $x_0=t_0=0$.
  The assumption on $\ell$ then simply reduces to $\ell\geq \sqrt{2n+1}$.
  
  Since the hypersurfaces $M_t$ are convex they expand in backward time under MCF.
  Thus, if $M_0$ encloses the line segment $[-2\ell, +2\ell] \times \{0\}$ (i.e.~if $U(x, 0)$ is defined for $|x|\leq 2\ell$), then so does $M_t$ for all $t<0$.
  
  For now we ignore the fact that $M_t$ evolves by MCF, and merely consider the consequences of convexity for the hypersurface at some time $t\in(-1,0]$.
  
  If $x\mapsto U(x, t)$ is a non negative concave function that is defined for $|x|\leq 2\ell$, then we have
  \begin{equation}
    \label{eq-U-convex-estimate}
    \tfrac12 U(0, t) \leq U(x,t)\leq 2 U(0, t)\,\, \text{ for } \,\, |x|\leq \ell \,\, \text{ and } \, -1<t\leq 0
  \end{equation}
  (see Figure~\ref{fig-cylindrical}).
  The concavity of $x\mapsto U(x, t)$ also implies that at any $x\in\R$ for which $U (\cdot, t)$ is defined on the whole interval $(x-\ell, x+\ell)$, one has the derivative estimate
  \[
  |U_x(x, t)| \leq \frac{U(x,t)} {\ell}
  \]
  (see again Figure~\ref{fig-cylindrical}).
  Combined with \eqref{eq-U-convex-estimate} this leads us to 
  \begin{equation}
    \label{eq-U-convex-derivative-estimate}
    |U_x(x, t)| \leq 2\, \frac{U(0,t)}{\ell} \,\, \text{ for } \ |x|\leq \ell \ \text{ and } \, -1<t\leq 0.
  \end{equation}
  
  We now recall that $M_t$ is a solution to MCF.
  Since $U(0,0)=1$, the hypersurface $M_0$ intersects the closed ball $\bar B_1(0,0)$.
  By the maximum principle for MCF, the hypersurface $M_t$ must intersect $B_{\sqrt{1-2nt}}(0,0)$ for all $t\in(-1, 0]$.
  It follows that for each $t\in(-1, 0]$ there is an $x'$, with $|x'|\leq \sqrt{1+2n}$, for which $U(x', t) \leq \sqrt{1+2n}$.
  If we assume that $\ell \geq \sqrt{1+2n}$, then \eqref{eq-U-convex-estimate} implies that $U(0, t) \leq 2 U(x', t) \leq 2\sqrt{1+2n}$.
  Applying this estimate to \eqref{eq-U-convex-estimate} and \eqref{eq-U-convex-derivative-estimate} we find
  \[
  |U(x,t)|\leq 4\sqrt{2n+1}, \qquad \ell |U_x(x, t)| \leq 4\sqrt{2n+1},
  \]
  when $|x|\leq \ell$ and $-1<t\leq 0$.
  
  \begin{figure}[t]\centering
    \includegraphics[width=\textwidth]{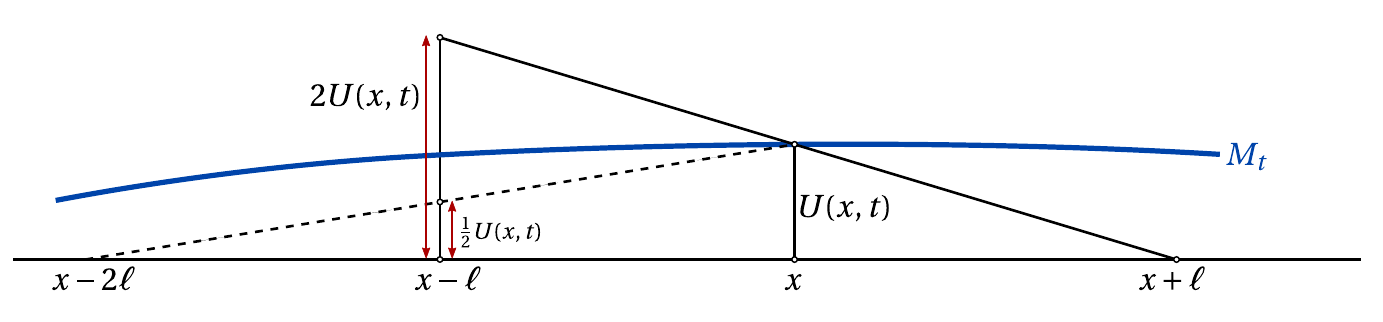}
    
    \includegraphics[width=\textwidth]{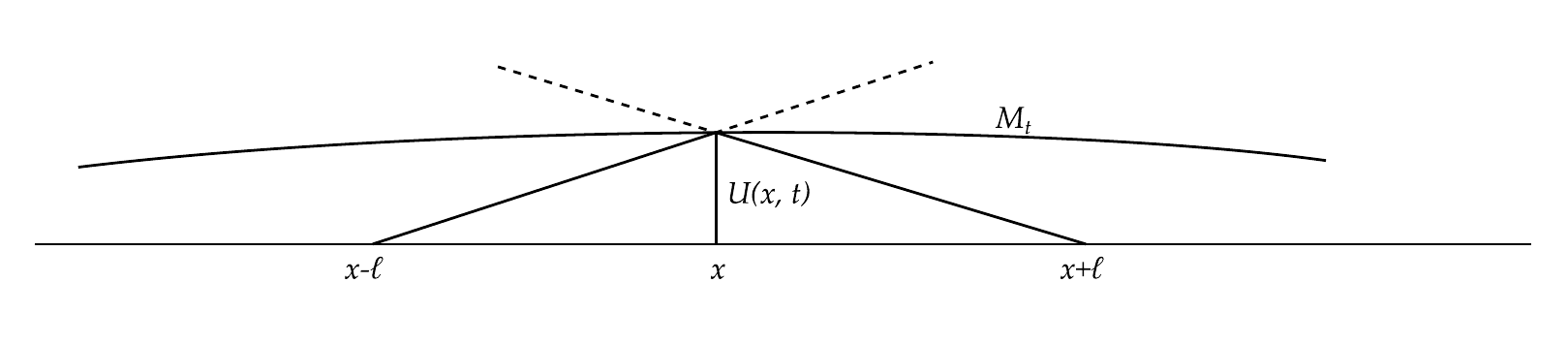}
    \caption{\textbf{Top: } If $M_t$ is a rotationally symmetric convex hypersurface that encloses the line segment $\{(x', 0) \mid x-2\ell \leq x'\leq x+2\ell\}$, then one has $\frac12 U(x,t) \leq U(x', t) \leq 2U(x, t)$ whenever $x-\ell\leq x'\leq x+\ell$.  \textbf{Bottom: } If $M_t$ encloses the line segment $\{(x', 0) \mid x-\ell\leq x'\leq x+\ell\}$, then at $x$ one has $|U_x(x,t)| \leq U(x, t)/\ell$.}
    \label{fig-cylindrical}
  \end{figure}
  
\end{proof}

We will now proceed to the proof of Proposition \ref{pro-Sigurd}.
\begin{proof}[Proof of Proposition \ref{pro-Sigurd}]
  We will argue by contradiction.
  Assuming that our claim doesn't hold, we can find a decreasing sequence $\tau_j \to -\infty$ and points $(y_j,\tau_j)$ such that $q_{yy} (y_j,\tau_j) = \max_{\bar M_{\tau_j }} q_{yy} (\cdot,\tau_j) > 0.$ By the symmetry of our surface, we may assume without loss of generality that $y_j >0$.
  It follows from Lemma \ref{lemma-help} that the sequence $\{ q_{yy} (y_j,\tau_j) \}$ is non-increasing, implying that
  \begin{equation}\label{eq-qyy}
    q_{yy} (y_j,\tau_j) := \max_{\bar M_{\tau_j }} q_{yy} (\cdot,\tau_j) \geq c >0, \qquad \forall i.
  \end{equation}
  
  We need the following two simple claims.
  
  \begin{claim}\label{claim-uy}
    Set $ \delta_j:= |u_y(y_j,\tau_j)|$, where $(y_j,\tau_j)$ as in \eqref{eq-qyy}.
    Then, we have
    \[
    \lim_{\tau_j \to -\infty} \delta_j = 0.
    \]
  \end{claim}
  
  \begin{proof}[Proof of Claim \ref{claim-uy}] Indeed, if our claim doesn't hold this means that there exists a subsequence, which may be assumed without loss of generality to be the sequence $\tau_j$ itself, for which $ |u_y(y_j,\tau_j)| \geq \delta >0$.
    However, after flipping the coordinates and using the change of variables
    \[
    Y(u,\tau) = Y(0, \tau) + \frac{1}{\sqrt{|\tau|}}\,Z \Bigl(\rho, \tau\Bigr), \qquad \rho = \sqrt{|\tau|}\, u
    \]
    we find that for $u_j=u(y_j, \tau_j), \rho_j = \sqrt{|\tau_j|}\, u_j$ we have
    \[
    |u_y(y_j,\tau_j)| = \frac 1{|Y_u(u_j,\tau_j)|} = \frac 1{|Z_\rho(\rho_j,\tau_j)|} \geq \delta \implies |Z_\rho(\rho_j,\tau_j)| \leq \frac 1{\delta}.
    \]
    The monotonicity of $Z_\rho(\rho,\tau)$ in $\rho$ and the convergence $\lim_{\tau \to \infty} Z(\rho, \tau) = Z_0(\rho)$ smoothly on any compact set in $\rho$, imply that $\rho_j \leq \rho_{\delta}$, where $\rho_{\delta}$ is the point at which $|(Z_0)_\rho(\rho_\delta)|=2/\delta$.
    We may assume, without loss of generality that $\delta $ is small, which means that $\rho_\delta$ is large.
    The asymptotics \eqref{eq-Z0-asymptotics} for $Z_0(\rho)$ as $\rho \to \infty$, give that $|(Z_0)_\rho(\rho) | \sim {\rho}/({\sqrt{2} (n-1)})$, as $\rho \to +\infty$, implying that by choosing $\delta$ sufficiently small we have $2/ \delta = |(Z_0)_\rho(\rho_\delta)| \sim {\rho_\delta}/({\sqrt{2} (n-1)})$, or equivalently $\rho_\delta \sim 2\sqrt{2}(n-1)/\delta$.
    Since $\rho_j \leq \rho_\delta$, we conclude that the points $(\rho_j, \tau_j, Z(\rho_j,\tau_j))$, or equivalently the points $(y_j,\tau_j, u(y_j,\tau_j))$, belong to the soliton region where we know that $q_{yy} <0$ by Lemma \ref{lemma-help2}, contradicting our assumption \eqref{eq-qyy}.
  \end{proof}
  
  \begin{claim}\label{claim-uy2}
    Let $y_j^1>0$ be the point for which ${\displaystyle u(y_j^1,\tau_j) = \frac 12 u(y_j,\tau_j)}$, where $(y_j,\tau_j)$ as in \eqref{eq-qyy}.
    If $\bar \delta_j:= |u_y(y_j^1,\tau_j)|$, then
    \[
    \lim_{j \to +\infty}\bar \delta_j =0.
    \]
  \end{claim}
  
  \begin{proof}[Proof of Claim \ref{claim-uy2}]
    We will again argue by contradiction.
    If our claim doesn't hold this means that there exists a subsequence, which may be assumed without loss of generality to be the sequence $\tau_j$ itself, for which $ |u_y^1(y_j^1,\tau_j)| \geq \delta >0$.
    Then, arguing as in the previous claim implies that $(y_j^1,\tau_j,u(y_j^1,\tau_j))$ belong to the tip region which means that $u(y_j^1,\tau_j) \leq L / \sqrt{|\tau|} $, for some uniform number $L >0$.
    Since ${\displaystyle u(y_j^1,\tau_j) = \frac 12 u(y_j,\tau_j)}$, we also have $u(y_j,\tau_j) \leq 2\, L / \sqrt{|\tau|} $, implying that the points $(y_j,\tau_j,u(y_j,\tau_j))$ belong to the tip region as well and by Lemma \ref{lemma-help2} we must have $q_{yy}(y_j,\tau_j) <0$, for $\tau_j \ll -1$ contradicting our assumption \eqref{eq-qyy}.
  \end{proof}
  
  We will now conclude the proof of the proposition.
  Let $(y_j,\tau_j)$ be our maximum points for $q_{yy}$ as in \eqref{eq-qyy} and let $y_j^1$ be the points for which ${\displaystyle u(y_j^1,\tau_j) = \frac 12 u(y_j,\tau_j)}$, as in the previous claim.
  Since $u_y <0$, we must have $ y_j^1 > y_j$ and by the concavity of $u$ we obtain
  \begin{equation}\label{eqn-est100}
    \frac 12 u(y_j,\tau_j) = u(y_j,\tau_j) - u(y_j^1,\tau_j) \leq |u_y (y_j^1, \tau_j)|\, |y_j - y_j^1| = \bar \delta_j \, |y_j - y_j^1|.
  \end{equation}
  
  We will now use a rescaling argument to reach a contradiction.
  To this end it is more convenient to work in the original variables, rescaling the solution $U(x,t)$.
  Denote by $(x_j,t_j)$, $(x_j^1,t_j)$ the points corresponding to $(y_j,\tau_j)$, $(y_j^1,\tau_j)$, respectively.
  Setting
  \[
  U_j( \bar x, \bar t) = \frac 1{\alpha_j} U(x_j+ \alpha_j \, \bar x , t_j + \alpha_j^2 \, \bar t), \qquad \alpha_j := U(x_j,t_j)
  \]
  it follows that all $U_j$ satisfy the same equation \eqref{eq-u-original} with $U_j(0,0)=1$.
  Denote by $\bar x_j^1$ the point at which $x_j + \alpha_j \, \bar x_j^1 = x_j^1$.
  Since, by \eqref{eqn-est100} we have
  \[
  \frac 12 U(x_j,t_j) \leq \delta_j \, |x_j - x_j^1|
  \]
  in terms of the rescaled solutions we have
  \[
  \frac 12 \alpha_j \, U_j(0,0) \leq \bar \delta_j \, |x_j - (x_j + \alpha_j \, \bar x_j^1 ) | = \bar \delta_j \, \alpha_j \, | \bar x_j^1|,
  \]
  where $x_j + \alpha_j \, \bar x_j^1 = x_j^1$.
  Thus, defining the length $l_j$ so that $2 l_j:= x_j^1 >0$, it follows that
  \[
  2l_j= \bar x_j^1 \geq \frac{1}{2 \bar \delta_j}.
  \]
  
  Now consider $U_j$ on the interval $[-2 l_j, 2 l_j]=[-\bar x_j^1, \bar x_j^1]$ and apply Lemma \ref{lemma-help3}.
  Let us verify that the assumptions of the lemma hold.
  The rescaled surface, defined through the rescaled width function $U_j(\cdot,0)$ encloses the interval $(-\bar{x}_j^1, \bar{x}_j^1)$ and
  \[
  \bar{x}_j^1 = 2l_j \ge \frac{1}{2\bar{\delta}_j} \to \infty,
  \]
  as $j\to\infty$, where we have used Claim \ref{claim-uy2}.
  Moreover,
  \[
  l_j \ge \sqrt{2n+1} = \sqrt{2n+1} U_j(0,0), \quad \mbox{for}  \,\, j \gg 1. 
  \]
  We can now apply lemma \ref{lemma-help3} to conclude 
  \[
  |(U_j)_{\bar{x}}(\cdot,\bar{t})| \le \frac{8\sqrt{2n+1}}{l_j} \le C\, \bar{\delta}_j \ll -1, \quad \mbox{for}  \,\, j \gg 1,
  \]
  for all $\bar t \in [-1,0]$ and $|\bar x | \leq l_j=1/(4 \bar \delta_j)$, in particular for $ |\bar x | \leq 2$.
  Thus, on the cube $Q_2:= \{ | \bar x| \leq 2, \,\, -1 \leq t \leq 0 \}$ we have $U \geq 1/2$ (since $|(U_j)_{\bar x}(\cdot, \bar t)| \leq C\, \bar \delta_j \ll-1$).
  By standard cylindrical estimates, passing if necessary to a subsequence $j_k$, we conclude that $\lim_{j_k \to +\infty} U_{j_k} = \hat U$ on the cube $Q_1:= \{ | \bar x| \leq 1, \,\, -1/4 \leq t \leq 0 \}$, where $\hat U$ still solves equation \eqref{eq-u-original} and satisfies $\hat U(0,0)=1$ and $\hat U_{\bar x}(\bar x,0)=0$ for $\bar x \in [-1,1]$.
  This in particular implies that $\hat U_{\bar x\bar x}(0,0) =0$, thus $\hat Q_{\bar x \bar x}(0,0) := (\hat U^2)_{\bar x\bar x} (0,0)=0$.
  On the other hand, since the quantity $Q_{xx}$ is scaling invariant, we have
  \[
  \lim_{j_k \to +\infty} (Q_{j_k})_{\bar x \bar x} (0,0) = \lim_{j_k \to +\infty} Q_{xx} (x_{j_k},t_{j_k}) = \lim_{j_k \to +\infty} q_{xx} (y_{j_k},\tau_{j_k}) \geq c >0,
  \]
  where in the last inequality we used our assumption \eqref{eq-qyy}.
  This is a contradiction, finishing the proof of the Proposition.
\end{proof}

In the rotationally symmetric case that we consider here, the principal curvatures of our hypersurface are given by
\[
\lambda_1 = -\frac{ u_{yy}}{(1+u_y^2)^{3/2}} \qquad \mbox{and} \qquad \lambda_2 = \cdots = \lambda_n =\frac{1}{u \, (1 + u_y^2)^{1/2}}.
\]
In \cite{ADS} we showed that on our Ancient Ovals $M_t$ we have
\[
\lambda_1 \le \lambda_2.
\]
We also showed $\lambda_1 = \lambda_2$ at the tip of the Ancient Ovals, at which the mean curvature is maximal as well.
The quotient
\[
R:= \frac{\lambda_1}{\lambda_2} = \frac{ U\, U_{xx}}{1+U_x^2} = \frac{ u\, u_{yy}}{1+u_y^2}
\]
is a scaling invariant quantity and in some sense measures how close we are to a cylinder, in a given region and at a given scale.
It turns out that this quotient can be made arbitrarily small {\em outside} the {\em soliton } region $ S_L(\tau) := \big \{y \mid 0 \le u(y,\tau) \le \frac{L}{\sqrt{|\tau|}}\big\}$, by choosing $L \gg 1$ and $\tau \leq \tau_0 \ll -1$.
This is shown next.

\begin{proposition}
  \label{prop-ratio-small}
  For every $\eta > 0$, there exist $L \gg 1$ and $\tau_0 \ll -1$ so that
  \[
  \frac{\lambda_1}{\lambda_2}(y,\tau) < \eta, \qquad \mbox{if} \,\,\, u(y,\tau) > \frac{L}{\sqrt{|\tau|}} \,\,\, \mbox{and } \, \,\, \tau \le \tau_0.
  \]
\end{proposition}

\begin{proof} The proof is by contradiction in the spirit of Proposition \ref{pro-Sigurd} but easier.
  Assuming that our proposition doesn't hold, this means that there is an $\eta >0$ and sequences $\tau_j \to -\infty$, $L_j \to +\infty$ and points $(y_j,\tau_j)$ for which we have
  \begin{equation}
    \label{eqn-ratio}
    \frac{\lambda_1}{\lambda_2}(y_j,\tau) \geq \eta >0 \quad \mbox{and} \quad u(y_j,\tau_j) > \frac{L_j}{\sqrt{|\tau_j|}}.
  \end{equation}
  
  \begin{claim}\label{claim-uy5}
    Set $ \delta_j:= |u_y(y_j,\tau_j)|$, where $(y_j,\tau_j)$ as in \eqref{eqn-ratio}.
    Then, we have
    \[
    \lim_{\tau_j \to -\infty} \delta_j = 0.
    \]
  \end{claim}
  
  \begin{proof} [Proof of Claim \ref{claim-uy5}] This claim is shown in a very similar away as Claim \ref{claim-uy} in Proposition \ref{pro-Sigurd}.
    Arguing by contradiction, if our claim doesn't hold this means that there exists a subsequence, which may be assumed without loss of generality to be the sequence $\tau_j$ itself, for which $ |u_y(y_j,\tau_j)| \geq \delta >0$.
    Arguing exactly as in the proof of Claim \ref{claim-uy} we conclude that the points $(y_j,\tau_j)$ satisfy $\sqrt{|\tau_j|} \, u(y_j,\tau_j) \leq C/{\delta} $, for an absolute constant $C$, contradicting that $u(y_j,\tau_j) > {L_j}/{\sqrt{|\tau_j|}}$ with $L_j \to +\infty$.
  \end{proof}
  
  We will now use the same rescaling argument as in Proposition \ref{pro-Sigurd} to reach a contradiction.
  Working again in the original variables, we rescale the solution $U(x,t)$, setting
  \[
  U_j( \bar x, \bar t) = \frac 1{\alpha_j} U(x_j+ \alpha_j \, \bar x , t_j + \alpha_j^2 \, \bar t), \qquad \alpha_j := U(x_j,t_j),
  \]
  where $(x_j,t_j)$ are the points in the original variables corresponding to $(y_j,\tau_j)$.
  The same argument as before, based now on Claim \ref{claim-uy5} instead of Claim \ref{claim-uy} (note that Claim \ref{claim-uy2} still holds in our case) allows us to pass to the limit and conclude that passing to a subsequence $j_k$ we have $U_{j_k} \to \hat U$, smoothly on compact sets.
  The limit $\hat U$ still solves equation \eqref{eq-u-original} and satisfies $\hat U(0,0)=1$ and $\hat U_{\bar x}(\bar x,0)=0$ for $\bar x \in [-1,1]$.
  This in particular implies that $\hat U_{\bar x\bar x}(0,0) =0$.
  On the other hand, the ratio ${\displaystyle R:= \frac{ \lambda_1}{\lambda_2}}$ is scaling invariant, which means that
  \[
  \hat R(0,0) = \lim_{j_k \to +\infty} R_{j_k} (0,0) = \lim_{j_k \to +\infty} R (x_{j_k},t_{j_k}) = \lim_{j_k \to +\infty} R(y_{j_k},\tau_{j_k}) \geq \eta >0,
  \]
  where in the last inequality we used our assumption \eqref{eqn-ratio}.
  This is a contradiction, since at the point $(0,0)$ we also have
  \[
  \hat R (0,0) = \frac{U\, U_{\bar x \bar x} (0,0)}{1 + \hat U_{\bar x}^2(0,0)}=0,
  \]
  therefore finishing the proof of the Proposition.

\end{proof}

We will finally use the convexity estimate shown in Proposition \ref{pro-Sigurd} to show the estimates in the next two Corollaries which will play a crucial role in estimating various terms in the tip region $\mathcal{T}_\theta$, in Section \ref{sec-tip}.
The first  Corollary concerns with an estimate which holds in the collar region $\collar_{\theta,L}$, as defined in \ref{subsec-tip}.

\begin{corollary}
  \label{lemma-Sigurd}
  Let $u$ be an ancient oval solution of \eqref{eq-u} which satisfies the asymptotics in Theorem \ref{thm-old}.
  Then, for $0 < \theta \ll 1$ and $L \gg 1$ large, there exists $\epsilon(\theta,L)$ small and a $\tau_0 \ll -1$ for which we have
  \[
  \Bigl|1 + \frac{1}{2(n-1)} Y \, \frac{u}{Y_u} \Bigr | < \epsilon(\theta,L) \quad \mbox{in} \, \, \collar_{\theta,L}, \quad \mbox{for} \,\, \tau \le \tau_0.
  \]
  Moreover, for $L \gg 1$ and $\theta \ll 1$, we can choose $\epsilon := \max\{4\theta^2, c(n) L^{-1}\}$.
\end{corollary}

\begin{proof}
  Assume for the moment we have $(u^2)_{yy} \le 0$.
  We need to show that ${\displaystyle 1-\epsilon \leq - \frac 1{2(n-1)} \frac{Y\, u}{Y_u} \leq 1+ \epsilon}$ in the considered region which is equivalent to
  \begin{equation}\label{eqn-crucialY2}
    1-\epsilon \leq - \frac 1{4(n-1)} y\, (u^2)_y \leq 1+ \epsilon.
  \end{equation}
  Now since at $u=2\theta$ we have
  \[
  y \approx \sqrt{|\tau|} \, \sqrt{ 2 - \frac{u^2}{n-1}} = \sqrt{ 2|\tau|} \sqrt{1- \frac{2 \theta^2}{n-1}} \approx \sqrt{ 2|\tau|} \, \big ( 1- \frac{ \theta^2}{n-1} \big ),
  \]
  it follows that at $u=2\theta$ and for $\theta$ small, $y \geq \sqrt{ 2|\tau|} \, ( 1- 2 \theta^2 ).$ Hence, in the considered region ${L}/{\sqrt{|\tau|}} \leq u \leq 2\theta$, we have
  \begin{equation}\label{eqn-y12}
    \sqrt{2|\tau|} (1- 2 \theta^2) \leq y \leq \sqrt{2|\tau|}.
  \end{equation}
  Next, using the inequality $-(u^2)_{yy} \geq 0$ which was shown in Proposition \ref{pro-Sigurd}, we can estimate
  \[
  - (u^2)_{y}\bigr|_{u=2\theta} 
  \leq -(u^2)_{y} 
  \leq - (u^2)_{y}\bigr|_{u=L/\sqrt{|\tau|}}.
  \]
  Our intermediate region asymptotics from Theorem \ref{thm-old} imply that at $u=2\theta$,
  \[
  - (u^2)_{y} |_{u=2\theta} = 2 (n-1) \frac{y}{|\tau|} \approx \frac{2 \sqrt{2} (n-1)}{\sqrt{|\tau|}} ( 1- \frac{\theta^2}{n-1} ).
  \]
  On the other hand, the asymptotics in the tip region give us that at $u=L/\sqrt{|\tau|}$, we have
  \[
  - (u^2)_{y}|_{u=L/\sqrt{|\tau|}} = - \frac{2u}{Y_u}|_{u=L/\sqrt{|\tau|}} \approx \frac{2L}{\sqrt{|\tau|}} \frac 1{Z_\rho(L,\tau)}.
  \]
  The smooth convergence $\lim_{\tau\to -\infty} Z(\rho,\tau) = Z_0(\rho)$, together with the asymptotics \eqref{eq-Z0-asymptotics} imply that for $L \gg1$, we have
  \[
  Z_\rho(L,\tau) \geq \frac{L - c}{\sqrt{2} (n-1)},
  \]
  for a fixed constant $c=c(n)$, hence
  \[
  - (u^2)_{y}|_{u=L/\sqrt{|\tau|}} \leq \frac{2L}{\sqrt{|\tau|}} \, \frac{\sqrt{2} (n-1)}{L-c} = \frac{2 \sqrt {2} (n-1)}{\sqrt{|\tau|}} \, (1+ \epsilon),
  \]
  for $\epsilon = c/L$, for another fixed constant $c=c(n)$.
  We conclude that
  \begin{equation}\label{eqn-u12}
    \frac{2 \sqrt{2} (n-1)}{\sqrt{|\tau|}} ( 1- \frac{\theta^2}{n-1} ) \leq -(u^2)_{y} \leq \frac{2 \sqrt {2} (n-1)}{\sqrt{|\tau|}} \, (1+ \epsilon).
  \end{equation}
  Combining \eqref{eqn-y12} and \eqref{eqn-u12} yields
  \[
  (1- 2 \theta^2)\, ( 1- \frac{\theta^2}{n-1} ) \leq - \frac 1{4(n-1)} y\, (u^2)_y \leq (1+\epsilon),
  \]
  which yields \eqref{eqn-crucialY2} for $\epsilon := \max (4\theta^2, c(n)\, L^{-1})$ and $L\gg1$, $\theta \ll1$.
\end{proof}

\begin{remark}
  \label{cor-sigurd}
  It is an easy consequence of Corollary \ref{lemma-Sigurd} that for $0 < \theta \ll 1$ small and $L \gg 1$ large, there exists a $\tau_0 \ll -1$ for which we have
  \begin{equation}
    \label{eq-cor-Sigurd}
    (1 - \epsilon)\,\frac{Y}{2(n-1)} 
    < \frac{|Y_u|}{u} 
    < (1 + \epsilon)\, \frac{Y}{2(n-1)} 
    \quad \mbox{in} \,\, \collar_{\theta,L}, 
    \quad \mbox{for} \,\, \tau \le \tau_0
  \end{equation}
  with $\epsilon:= \max\{4\theta^2, c(n) L^{-1}\}$ small.
  From now on we denote by the same symbol $\epsilon =\epsilon(\theta,L)$ a constant that is small for $\theta \ll 1$ and $L \gg 1$, but may differ from line to line.
\end{remark}

We will next show that if $u_i$, $i=1,2$ are two solutions as in Theorem \ref{thm-main}, then $Y_1(u,\tau)$ and $Y_2(u,\tau)$ are comparable to each other in the whole tip region $\tip_{\theta}$.

\begin{corollary}
  Let $u_i(y,\tau)$, $i=1,2$ be two solutions as in Theorem \ref{thm-main}, and let $Y_i(u,\tau)$, $i=1,2$ be the corresponding solutions in flipped coordinates.
  Then, for every $\epsilon > 0$ there exist $0 < \theta \ll 1$ small and $\tau_0 \ll -1$ so that
  \begin{equation}
    \label{eq-comp-der}
    1 - \epsilon < \frac{Y_{1u}}{Y_{2u}} < 1 + \epsilon \quad \, \mbox{in} \,\,\, \tip_{\theta}, \quad \mbox{for} \,\, \tau \le \tau_0.
  \end{equation}
\end{corollary}

\begin{proof}
  We begin by observing that \eqref{eq-comp-der} holds in the collar region $\collar_{\theta,L}$, which is an immediate consequence of \eqref{eq-cor-Sigurd} and  ${\displaystyle 1 -\epsilon < \frac{Y_1}{Y_2} < 1 + \epsilon}$, which holds in the considered region. 
  
  Hence, we only need to show that \eqref{eq-comp-der} holds in the soliton region $S_L$.
  Let $Z_i$, $i=1,2$, be the functions defined in terms of $Y_i$ by \eqref{eqn-Y-expansion}.
  Then $\lim_{\tau\to -\infty} Z_i(\rho,\tau) = Z_0(\rho)$, uniformly smoothly on compact sets in $\rho$, for both $i\in \{1,2\}$.
  Write
  \begin{equation}
    \label{eq-Y1u-Y2u}
    \Bigl|\frac{Y_{1u}}{Y_{2u}} - 1\Bigr| = \Bigl|\frac{Z_{1\rho}}{Z_{2\rho}} - 1\Bigr| \le \frac{|Z_{1\rho} - Z_{0\rho}|}{|Z_{2\rho}|} + \frac{|Z_{2\rho} - Z_{0\rho}|}{|Z_{2\rho}|}.
  \end{equation}
  By the smoothness of $Z_i(\rho,\tau)$, $Z_0(\rho)$ around the origin we have $(Z_i)_\rho(0,\tau) = (Z_0)_\rho (0) = 0$ and hence,
  \begin{equation}
    \label{eq-diff-Z-small}
    |(Z_i(\rho,\tau) - Z_0(\rho))_{\rho}| \le \sup_{\rho\in [0,2L]} |(Z_i(\rho,\tau) - Z_0(\rho))_{\rho\rho}|\, \rho < \eta \, \rho
  \end{equation}
  if $\tau \le \tau_0 \ll -1$ is sufficiently small and close to negative infinity ($\eta > 0$ a small positive number to be chosen below).
  This also implies
  \[
  |Z_{i\rho}(\rho,\tau)| \ge |(Z_{0\rho}(\rho)| - \eta \, \rho.
  \]
  On the other hand, by the asymptotics for $Z_0(\rho)$ we have the
  \[
  \lim_{\rho\to 0} \frac{Z_{0\rho}}{\rho} = -\frac{1}{ \sqrt{2} \, n} \qquad \mbox{and} \qquad \lim_{\rho\to +\infty} \frac{Z_{0\rho}}{\rho} = -\frac{1}{\sqrt{2}\, (n-1)}.
  \]
  Let
  \[
  2a := \min\Biggl\{\frac{1}{2\sqrt{2}\, (n-1)}, \min_{\rho\in [\rho_1,\rho_2] } \, \frac{|Z_{0\rho}|}\rho\Biggr\},
  \]
  where $\rho_1 > 0$ is close to zero and $\rho_2 > 0$ is very large, so that ${\displaystyle \frac{|Z_{0\rho}|}\rho > \frac{1}{2\sqrt{2} (n-1)} }$, for $\rho \le \rho_1$ or $\rho \ge \rho_2$.
  By the definition of $a$ we have ${\displaystyle \frac{|Z_{0\rho}|}\rho \ge 2 a > 0}$ for all $\rho$.
  Choosing ${\displaystyle 0 < \eta < \min\{a, \frac{\epsilon}{2}\, a\}}$ we can make
  \[
  |Z_{i\rho}(\rho,\tau)| \ge a \, \rho, \qquad \rho\in [0,L], \quad \tau\le \tau_0 \ll -1.
  \]
  Combining this, \eqref{eq-Y1u-Y2u} and \eqref{eq-diff-Z-small} yields
  \[
  \Bigl|\frac{Z_{1\rho}}{Z_{2\rho}} - 1\Bigr| < 2\, \frac{\eta }{a } \leq \epsilon, \qquad \rho\in [0,L], \quad \tau\le \tau_0 \ll -1.
  \]
  This concludes the proof of the Corollary.
\end{proof}

\medskip
\section{The cylindrical region}
\label{sec-cylindrical}

Let $u_1(y,\tau)$ and $u_2(y,\tau)$ be the two solutions to equation \eqref{eq-u} as in the statement of Theorem \ref{thm-main} and let $u_2^{\alpha\beta\gamma}$ be defined
by \eqref{eq-ualphabeta}.
In this section we will estimate the difference $w:= u_1-u_2^{\alpha\beta\gamma}$ in the cylindrical region 
$\cyl_{\theta} = \{y\,\,\, |\,\,\, u_1(y,\tau) \ge {\theta}/{2}\, \}$, for a given number $\theta > 0$ small  and any $\tau \leq \tau_0 \ll -1$.
Recall all the definitions and notation introduced in Section \ref{subsec-cylindrical}.
Before we state and prove the main estimate in the cylindrical region we give a remark that a reader should be aware of throughout the whole section.

\begin{remark}
  \label{rem-cylindrical}
  Recall that we write simply $u_2(y,\tau)$ for $u_2^{\alpha\beta\gamma}(y,\tau)$, where
  \[
  u_2^{\alpha\beta\gamma}(y,\tau) = \sqrt{1+\beta e^{\tau}}\, u_2\Bigl(\frac{y}{\sqrt{1+\beta e^{\tau}}}, \tau + \gamma - \log(1 + \beta e^{\tau})\Bigr),
  \]
  is still a solution to \eqref{eq-u} and simply write $w(y,\tau)$ for $w^{\alpha\beta\gamma}(y,\tau) := u_1(y,\tau) - u_2^{\alpha\beta\gamma}(y,\tau)$.
  As it has been already indicated in Section \ref{subsec-conclusion}, we will choose $\alpha=\alpha(\tau_0)$, $\beta = \beta(\tau_0)$ and $\gamma = \gamma(\tau_0)$ (as it will be explained in Section \ref{sec-conclusion}) so that the projections $\pr_+ w_\cyl(\tau_0) = \pr_0 w_\cyl(\tau_0) = 0$, at a suitably chosen $\tau_0 \ll -1$.
  In Section \ref{sec-conclusion} we show the pair $(\alpha,\beta,\gamma)$ is admissible with respect to $\tau_0$, in the sense of Definition \ref{def-admissible}, if $\tau_0$ is sufficiently small.
  That will imply all our estimates that follow are independent of parameters $\alpha, \beta, \gamma$, as long as they are admissible with respect to $\tau_0$, and will hold for $u_1(y,\tau) - u_2^{\alpha\beta\gamma}(y,\tau)$, for $\tau\le \tau_0$ (as explained in section \ref{sec-regions}).
\end{remark}

Our goal in this section is to prove that the bound~\eqref{eqn-cylindrical1} holds as stated next.

\begin{prop}\label{prop-cylindrical}
  For every $\epsilon > 0$ and $\theta > 0$ small there exists a $\tau_0 \ll -1$ so that if $w(y,\tau)$ is a solution to \eqref{eqn-ww} for which $\pr_+w_\cyl(\tau_0) = 0$, then we have
  \[
  \| \hat w_\cyl \|_{\hv,\infty}
  \leq
  \epsilon\, \big(\| w_\cyl \|_{\hv,\infty} + \|w\, \chi_{D_{\theta}}\|_{\hilb,\infty}\big),
  \]
  where $D_{\theta} := \{y\,\,\,|\,\,\, \theta/2 \le u_1(y,\theta) \le \theta\}$ and $\hat{w}_C = \pr_- w_\cyl + \pr_+ w_\cyl$.
\end{prop}

\smallskip

The rest of this section will be devoted to the proof of Proposition \ref{prop-cylindrical}.
To simplify the notation for the rest of the section we will simply denote $u_2^{\alpha\beta\gamma}$ by $u_2$ and set $w:=u_1-u_2$.
The difference $w$ satisfies
\begin{equation}
  \label{eqn-ww}
  w_\tau = \frac{w_{yy}}{1+u_{1y}^2} -\frac{(u_{1y}+u_{2y})u_{2yy}}{(1+u_{1y}^2)(1+u_{2y}^2)} w_y -\frac y2 w_y + \frac 12 w +\frac{n-1}{u_1u_2} w
\end{equation}
which we can rewrite as
\begin{equation}
  w_\tau = \cL w + \cE w
  \label{eq-w-linear}
\end{equation}
in which $ \cL = \pd_y^2 - \frac y2 \pd_y + 1$ is as above, and where $ \cE $ is given by
\begin{equation}
  \label{eq-E10}
  \cE[\phi] = -\frac{u_{1y}^2}{1+u_{1y}^2} \phi_{yy} -\frac{(u_{1y}+u_{2y})u_{2yy}}{(1+u_{1y}^2)(1+u_{2y}^2)} \phi_y +\frac{2(n-1) - u_1u_2}{2u_1u_2}\phi.
\end{equation}

\subsection{The operator $ \cL $}
We recall the definition of the Hilbert spaces $\hilb$, $\hv$ and $\hv^*$ are given in Section \ref{subsec-cylindrical}.
The formal linear operator
\[
\cL = \pd_y^2 - \frac y2 \pd_y + 1 = -\pd_y^* \pd_y + 1
\]
defines a bounded operator $\cL:\hv\to\hv^*$, meaning that for any $f\in \hv$ we have that $Lf\in \hv^*$ is the functional given by
\[
\forall \phi\in \hv : \langle \cL f, \phi\rangle = \int_{\R} \bigl ( -f_y\phi_y + f\phi\bigr ) \, e^{-y^2/4}\, dy.
\]
By integrating by parts one verifies that if $f\in C^2_c$, one has
\[
\langle f, \phi \rangle = \int_\R \bigl ( f_{yy} - \frac y2 f_y + f\bigr ) \phi \,e^{-y^2/4}dy,
\]
so that the weak definition of $ \cL f $ coincides with the classical definition.

\subsection{Operator bounds and Poincar\'e type inequalities}

The following inequality was shown in Lemma 4.12 in \cite{ADS}.

\begin{lemma} \label{lem-Poincare} For any $f\in \hv$ one has
  \[
  \int_\R y^2 f(y)^2 e^{-y^2/4} dy \leq C \int_\R \big ( f(y)^2 + f_y(y)^2 \big ) \, e^{-y^2/4} dy,
  \]
  which implies that the multiplication operator $ f\mapsto yf $ is bounded from $ \hv $ to $ \hilb $, i.e.
  \[
  \|yf\|_\hilb \leq C \|f\|_\hv,
  \]
  for all $ f\in \hv $.
\end{lemma}

As a consequence we have the following two lemmas:

\begin{lemma} \label{lem-bounded-first-order} The following operators are bounded both as operators from $\hv$ to $\hilb$ and also as operators from $\hilb$ to $\hv^*$:
  \[
  f\mapsto yf,\quad f\mapsto \pd_y f, \quad f\mapsto \pd_y^* f = \bigl(-\pd_y + \frac y2\bigr)f,
  \]
  where $\pd_y^*$ is the formal adjoint of the operator $\pd_y$, it satisfies $\langle f, \pd_y^* g\rangle = \langle \pd_yf, g\rangle$ for all $f,g\in\hv$.
\end{lemma}

\begin{lemma}\label{lem-bounded-second-order}
  The following operators are bounded from $\hv$ to $\hv^*$:
  \[
  f\mapsto y^2f, \quad f\mapsto y\pd_y f, \quad f\mapsto \pd_y^2 f.
  \]
\end{lemma}

\begin{proof}[Proof of Lemmas \ref{lem-bounded-first-order} and \ref{lem-bounded-second-order}]
  By definition of the norms in $\hv$ and $\hilb$ the operator $\pd_y$ is bounded from $\hv$ to $\hilb$, and by duality its adjoint $\pd_y^* = -\pd_y + \frac y2$ is bounded from $\hilb$ to $\hv^*$.
  
  The Poincar\'e inequality from Lemma~\ref{lem-Poincare} implies directly that $f\mapsto yf$ is bounded from $\hv$ to $\h$.
  By duality the same multiplication operator is also bounded from $ \h $ to $ \hv^* $; i.e.~for every $ f\in \h $ the product $ y f $ defines a linear functional on $ \hv $ by $ \langle yf, \phi \rangle = \langle f, y\phi \rangle $ for every $ \phi\in \hv $.
  We get
  \[
  \|y\, f \|_{\hv^*} \leq C \|f\|_\hilb,
  \]
  for all $ f\in\hilb $.
  
  Composing the multiplications $y : \hv \to \h $ and $ y:\h\to\hv^* $ we see that multiplication with $ y^2 $ is bounded as operator from $ \hv $ to $ \hv^* $, i.e.~for all $ f\in\hv $ we have $ y^2f\in \hv^* $, and
  \[
  \|y^2\, f \|_{\hv^*} \leq C^2 \|f\|_\hv.
  \]
  Since $y:\hv \to \h$ and $\pd_y:\hv\to\h$ are both bounded operators, we find that $\pd_y^* = - \pd_y^*+\frac y2$ also is bounded from $\hv$ to $\hv$.
  By duality again, it follows that $\pd_y$ is bounded from $\h$ to $\hv^*$.
  This proves Lemma~\ref{lem-bounded-first-order}.
  
  Each of the operators in Lemma~\ref{lem-bounded-second-order} is the composition of two operators from Lemma~\ref{lem-bounded-first-order}, so they are also bounded.
\end{proof}

More generally, to estimate the operator norm of multiplication with some function $ m:\R\to\R $, seen as operator from $ \hv $ to $ \h $, we have
\[
\|m\,f\|_\hilb \leq \sup_{y\in\R} \frac{|m(y)|}{1+|y|} \; \|f\|_\hv.
\]

Indeed the following lemma can be easily shown.

\begin{lemma}
  Let $m:\R\to\R$ be a measurable function, consider the multiplication operator $\cM : f\mapsto mf$.
  Then, the following hold:
  
  $\cM:\h\to\h$ is bounded if $m\in L^\infty(\R)$, and $\|\cM\|_{\h\to\h} \leq \|m\|_{L^\infty}$.
  
  $\cM:\hv\to\h$ is bounded if and only if $\cM:\h\to\hv^*$ is bounded.
  Both operators are bounded if $(1+|y|)^{-1}m(y)$ is bounded, and
  \[
  \|\cM\|_{\h\to\hv^*} = \|\cM\|_{\hv\to\h} \leq C\,\mathrm{ess~sup}_{y\in\R} \frac{|m(y)|}{1+|y|}.
  \]
  Finally, $\cM$ is a bounded operator from $ \hv $ to $ \hv^* $ if $(1+|y|)^{-2}m(y)$ is bounded, and the operator norm is bounded by
  \[
  \|\cM\|_{\hv\to\hv^*} \leq \mathrm{ess~sup}_{y\in\R} \frac{|m(y)|}{(1+|y|)^2} .
  \]
\end{lemma}

\subsection{Eigenfunctions of $ \cL $}
There is a sequence of polynomials $ \psi_n(y) = y^n+\cdots$ that are eigenfunctions of the operator $ \cL $.
The $ n^{\rm th} $ eigenfunction has eigenvalue $ \lambda_n = 1 - \frac n2 $.
The first few eigenfunctions are given by
\[
\psi_0(y) = 1, \quad \psi_1(y) = y, \quad \psi_2(y) = y^2-2
\]
up to scaling.

The functions $ \{\psi_n : n\in\mathrm{N}\} $ form an orthogonal basis in all three Hilbert spaces $ \hv$, $\h$ and $\hv^*$.
The three projections $ \pr_\pm $ and $ \pr_0 $ onto the subspaces spanned by the eigenfunctions with negative/positive, or zero eigenvalues are therefore the same on each of the three Hilbert spaces.
Since $ \psi_2 $ is the eigenfunction with eigenvalue zero, they are given by
\[
\pr_+f = \sum_{j=0}^1 \frac{\langle \psi_j, f\rangle }{\langle \psi_j, \psi_j\rangle}\psi_j,\quad \pr_-f = \sum_{j=3}^\infty \frac{\langle \psi_j, f\rangle }{\langle \psi_j, \psi_j\rangle}\psi_j,\quad \pr_0f = \frac{\langle \psi_2, f\rangle }{\langle \psi_2, \psi_2\rangle}\psi_2.
\]

\subsection{Estimates for ancient solutions of the linear cylindrical equation}

In this section we will give energy type estimates for ancient solutions $ f:(-\infty, \tau_0] \to \hv $ of the linear cylindrical equation
\begin{equation}
  \label{eqn-linear1}
  \frac{df}{d\tau} - \cL f(\tau) = g(\tau).
\end{equation}

\begin{lemma}
  \label{lem-linear-cylindrical-estimates}
  Let $ f:(-\infty, \tau_0] \to \hv $ be a bounded solution of \eqref{eqn-linear1}.
  Then there is a constant $ C <\infty $ that does not depend on $ f $, such that
  \[
    \sup_{\tau\leq\tau_0}\| \hat f(\tau)\|_\hilb^2 +\frac 1C \int_{-\infty}^{\tau_0} \|\hat f(\tau)\|_\hv^2\, d\tau 
    \leq \|f_+(\tau_0)\|_\hilb^2 + C \int_{-\infty}^{\tau_0} \|\hat g(\tau)\|_{\hv^*}^2\, d\tau
  \]
  where $ f_+ = \pr_+ f $ and $ \hat f = \pr_+f + \pr_- f$.
\end{lemma}
\begin{proof}
  This is a standard cylindrical estimate applied to the infinite time domain $ (-\infty, \tau_0] $.
  Since the operator $ \cL $ commutes with the projections $ \pr_\pm$ we can split $ f(\tau) $ into its $ \pr_+ $ and $ \pr_- $ components, and estimate these separately.
  
  Applying the projection $ \pr_- $ to both sides of the equation $ f_\tau - \cL f = g $ we get
  \[
  f_-'(\tau) = \cL f_-(\tau) + g_-(\tau),
  \]
  where $g_-(\tau) = \pr_- g(\tau)$.
  This implies
  \[
  \frac 12\frac{d}{d\tau} \|f_-\|_\hilb^2 = \langle f_-, \cL f_- \rangle + \langle f_-, g_- \rangle.
  \]
  Using the eigenfunction expansion of $ f_- $ we get
  \[
  \langle f_-, \cL f_- \rangle \leq - C \|f_-\|_\hv^2.
  \]
  We also have
  \[
  \langle f_-, g_- \rangle \leq \|f_-\|_\hv \; \|g_-\|_{\hv^*} \leq \frac C2 \|f_-\|_\hv^2 + \frac {1}{2C}\|g_-\|_{\hv^*}^2.
  \]
  We therefore get
  \[
  \frac 12\frac{d}{d\tau} \|f_-\|_\hilb^2 \leq -\frac C2 \|f_-\|_\hv^2 + \frac 1 {2C} \|g_-\|_{\hv^*}^2.
  \]
  Integrating in time over the interval $ (-\infty, \tau]$ then leads to
  \[
  \frac12 \|f_-(\tau)\|_\hilb^2 + \frac C2 \int _{-\infty}^{\tau} \|f_-(\tau')\|_\hv^2 \; d\tau' \leq \frac 1{2C} \int_{-\infty}^\tau \|g_-(\tau')\|_{\hv^*}^2 \;d\tau'.
  \]
  Taking the supremum over $ \tau\leq \tau_0 $ then gives us the $ \pr_- $ component of \eqref{eq-basic-cylindrical-estimate}.
  For the other component, $ f_+(\tau) = \pr_+ f$, we have
  \[
  \langle f_+, \cL f_+ \rangle \geq C \|f_+\|_\hv^2.
  \]
  A similar calculation then leads to
  \[
  \frac 12\frac{d}{d\tau} \|f_+\|_\hilb^2 \geq \frac C2 \|f_+\|_\hv^2 - \frac 1 {2C} \|g_+\|_{\hv^*}^2.
  \]
  Integrating this over the interval $ [\tau, \tau_0] $  introduces the boundary term $ \|f_+(\tau_0)\|_\hilb^2 $, and gives us the estimate
  \[
  \frac12 \|f_+(\tau)\|_\hilb^2 + \frac C2 \int _{\tau}^{\tau_0} \|f_+(\tau')\|_\hv^2 \; d\tau' \leq \frac12 \|f_+(\tau_0)\|_\hilb^2 + \frac 1{2C} \int_{\tau}^{\tau_0} \|g_+(\tau')\|_{\hv^*}^2 \; d\tau'.
  \]
  Adding the estimates for $ \pr_+ f $ and $ \pr_-f $ yields \eqref{eq-basic-cylindrical-estimate}.
\end{proof}

\begin{lemma}
  \label{lem-linear-cylindrical-estimates-sup-L2-version}
  Let $f:(-\infty, \tau_0] \to \hv$ be a bounded solution of equation \eqref{eqn-linear1}.
  If $T>0$ is sufficiently large, then there is a constant $C_{\star}$ such that
  \begin{equation}
    \begin{split}
      \sup_{\tau\leq\tau_0}\| \hat f(\tau)\|_\hilb^2
      +\frac 1{C_{\star}} \sup_{n\geq 0} \int_{I_n} &\|\hat f(\tau)\|_\hv^2\, d\tau \\
      &\leq \|f_+(\tau_0)\|_\hilb^2 + C_\star \sup _{n\geq 0}\int_{I_n} \|\hat g(\tau)\|_{\hv^*}^2\, d\tau,
      \label{eq-basic-cylindrical-estimate}
    \end{split}
  \end{equation}
  where $I_n$ is the interval $ I_n = [\tau_0-(n+1)T, \tau_0-nT] $ and where $ f_+ = \pr_+ f $ and $ \hat f = \pr_+f + \pr_- f$.
  
\end{lemma}
\begin{proof}
  To simplify notation we assume in this proof that $\pr_0f(\tau)=0$, i.e.~that $\hat f(\tau) = f(\tau)$ for all $\tau$.
  Likewise we assume that $\hat g(\tau) = g(\tau)$ for all $\tau\leq \tau_0$.
  
  Choose a large number $T>0$ and let $\eta\in C^\infty_c(\R)$ be a smooth cut-off function with $\eta(t)=1$ for $t\in[-T,0]$, $\supp\eta \subset(-2T, +T)$.
  We may assume that
  \begin{equation}
    \label{eq-cut-off-f-derivative-bound}
    |\eta'(\tau)| \leq \frac 2 T\text{ for all }\tau\in\R.
  \end{equation}
  For any integer $n\geq 0$ we consider
  \[
  f_n(\tau) = \eta_n(\tau)f(\tau), \quad\text{where}\quad \eta_n(\tau) = \eta(\tau - \tau_0 + nT).
  \]
  \begin{figure}[b]
    \includegraphics{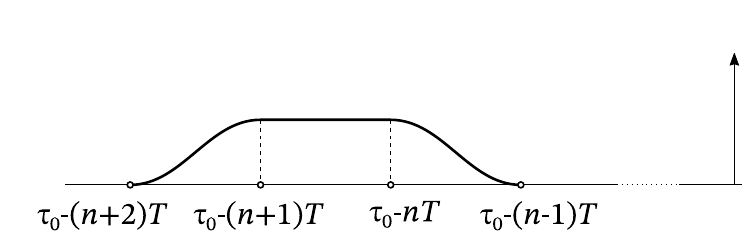}
    \caption{The cut off function $\eta_n(\tau)$, and the intervals $I_n$ and $J_n$.}
  \end{figure}
  The cut-off function $\eta_n$ satisfies $\eta_n(\tau)=1$ for $\tau\in I_n$, and $\supp\eta_n \subset J_n$, where, by definition,
  \[
  J_n = I_{n+1}\cup I_n \cup I_{n-1}.
  \]
  The function $f_n$ is a solution of
  \[
  f_n'(\tau) -\cL f_n(\tau) = \eta_n'(\tau) f(\tau) + \eta_n(\tau) g(\tau).
  \]
  If $n\geq 1$, then we can apply Lemma~\ref{lem-linear-cylindrical-estimates} to $f_n$, with $f_n(\tau_0) = 0$.
  Since $f_n$ and $f$ coincide on $I_n$, we get
  \begin{align*}
    \sup_{\tau\in I_n}\|f(\tau)\|_\hilb^2 +\frac1{C} \int_{I_n} \|f\|_{\hv}^2 d\tau &\leq \sup_{\tau\in J_n}\|f_n(\tau)\|_\hilb^2
    +\frac1{C} \int_{J_n} \|f_n\|_{\hv}^2 d\tau \\
    &\leq C \int_{J_n} \|\eta_n' f +\eta_n g\|_{\hv^*}^2 d\tau.
  \end{align*}
  Here $C$ is the constant from Lemma~\ref{lem-linear-cylindrical-estimates}.
  Using $(a+b)^2\leq 2(a^2+b^2)$ and also our bound \eqref{eq-cut-off-f-derivative-bound} for $\eta_n'(\tau)$ we get
  \[
  \sup_{\tau\in I_n}\|f(\tau)\|_\hilb^2 +\frac1{C} \int_{I_n} \|f\|_{\hv}^2 d\tau \leq C \int_{J_n} \biggl\{\frac 2{T^2}\| f \|_{\hv^*}^2 + \| g \|_{\hv^*}^2 \biggr\} d\tau.
  \]
  It follows that
  \begin{equation}
    \label{eq-supinfty-proof-1}
    \begin{split}
      \sup_{\tau\in I_n}\|f(\tau)\|_\hilb^2 + \frac1{C} \int_{I_n} &\|f\|_{\hv}^2 d\tau \\
      &\leq \frac{3C}{T^2} \sup_k \int_{I_k} \| f\|_{\hv^*}^2 d\tau + 3C\sup_k \int_{I_k} \| g \|_{\hv^*}^2 d\tau.
    \end{split}
  \end{equation}
  For $n=0$ the truncated function $f_n(\tau)$ is not defined for $\tau > \tau_0$ and we must use an estimate on $J_0 = I_1\cup I_0$.
  We apply Lemma~\ref{lem-linear-cylindrical-estimates} to the function $f_0(\tau) = \eta_0(\tau)f(\tau)$:
  \begin{equation} \label{eq-supinfty-proof-2}
    \begin{split}
      \sup_{\tau\in I_0}\|f(\tau)\|_\hilb^2 &
      +\frac1{C} \int_{I_0} \|f\|_{\hv}^2 d\tau\\
      &\leq \sup_{\tau\leq\tau_0}\|f_0(\tau)\|_\hilb^2 +\frac1{C} \int_{-\infty}^{\tau_0} \|f_0\|_{\hv}^2 d\tau
      \\
      &\leq \|f_+(\tau_0)\|_\hilb^2 + C \int_{-\infty}^{\tau_0} \|\eta_0'f+\eta_0 g\|_{\hv^*}^2 d\tau
      \\
      &\leq \|f_+(\tau_0)\|_\hilb^2 + 2C \int_{I_1} (\eta_0')^2 \|f\|_{\hv^*}^2 d\tau + 2C \int_{J_0} \|g\|_{\hv^*}^2 d\tau
      \\
      &\leq \|f_+(\tau_0)\|_\hilb^2 + \frac{2C}{T^2} \sup_k \int_{I_k} \| f\|_{\hv^*}^2 d\tau + 2C\sup_k \int_{I_k} \| g \|_{\hv^*}^2 d\tau.
    \end{split}
  \end{equation}
  Combining \eqref{eq-supinfty-proof-1} and \eqref{eq-supinfty-proof-2}, and taking the supremum over $n$, yield
  \[
  \begin{split}
    \sup_{\tau\leq \tau_0}\|f(\tau)\|_\hilb^2
    &+\frac1{C}\sup_{n}\int_{I_n} \|f\|_{\hv}^2 d\tau \\
    &\leq \|f_+(\tau_0)\|_\hilb^2 +\frac{3C}{T^2} \sup_k \int_{I_k} \| f\|_{\hv^*}^2 d\tau + 3C\sup_k \int_{I_k} \| g \|_{\hv^*}^2 d\tau.
  \end{split}
  \]
  Since $\|u\|_\hilb \leq \|u\|_\hv$ for all $u\in\hv$, it follows by duality that $\|u\|_{\hv^*} \leq \|u\|_\hilb$ for all $u\in\h$, and thus we have $\|f(\tau)\|_{\hv^*} \leq \|f(\tau)\|_{\hv}$.
  Therefore
  \[
  \begin{split}
    \sup_{\tau\leq \tau_0}\|f(\tau)\|_\hilb^2
    &+\frac1{C}\sup_{n}\int_{I_n} \|f\|_{\hv}^2 d\tau \\
    &\leq \|f_+(\tau_0)\|_\hilb^2 +\frac{3C}{T^2} \sup_k \int_{I_k} \| f\|_{\hv}^2 d\tau + 3C\sup_k \int_{I_k} \| g \|_{\hv^*}^2 d\tau.
  \end{split}
  \]
  At this point we assume that $T$ is so large that $3C/T^2 \leq 1/{2C}$, which lets us move the terms with $f$ on the right to the left hand side of the inequality:
  \[
  \sup_{\tau\leq \tau_0}\|f(\tau)\|_\hilb^2 +\frac1{2C}\sup_{n}\int_{I_n} \|f\|_{\hv}^2 d\tau \leq \|f_+(\tau_0)\|_\hilb^2 + 3C\sup_k \int_{I_k} \| g \|_{\hv^*}^2 d\tau.
  \]
\end{proof}

\subsection{$L^2$-Estimates for the error terms} 
The two solutions $u_1, u_2$ of equation \eqref{eq-u} that we are considering are only defined for \(y^2 \leq (2+o(1))|\tau|\).
This follows from the asymptotics in our previous work \cite{ADS} (see also Theorems  \ref{thm-old} and \ref{thm-O1}) where it was also shown that they satisfy the asymptotics 
\[
u(y, \tau) = \sqrt{(n-1)\, (2-z^2)} + o(1), \qquad \mbox{as}\,\, \tau\to-\infty
\]
uniformly in $z$, where  $ {\displaystyle  z = \frac{y}{\sqrt{|\tau|}}}$. 

We have seen that $w:= u_1-u_2$ satisfies \eqref{eq-w-linear} where the error term $\cE$ is given by \eqref{eq-E10}.
We will now consider this equation only in the ``cylindrical region,'' i.e.~the region where
\[
u > \frac{\theta}2  \qquad \text{ i.e. }\quad \frac{y}{\sqrt{|\tau|}} < \sqrt{2 - \frac{\theta^2}{4(n-1)}} + o(1).
\]
To concentrate on this region, we choose a cut-off function $\Phi\in C^\infty(\R)$ which decreases smoothly from $1$ to $0$ in the interior of the interval
\[
\sqrt{2 - \frac{\theta^2}{n-1}} < z < \sqrt{2 - \frac{\theta^2}{4(n-1)}}.
\]
With this cut-off function we then define
\[
\varphi_\cyl (y, \tau) = \Phi\Bigl(\frac{y}{|\tau|}\Bigr)
\]
and
\[
w_\cyl(y, \tau) = \varphi_\cyl(y, \tau) w(y, \tau).
\]
The cut-off function $\varphi_\cyl$ satisfies the bounds
\[
|(\varphi_\cyl)_y|^2 + |(\varphi_\cyl)_{yy}| + |(\varphi_\cyl)_{\tau}| \le \frac{\bar{C}(\theta)}{|\tau|}
\]
where $\bar{C}(\theta)$ is a constant that depends on $\theta$ and that may change from line to line in the text.
The localized difference function $w_\cyl$ satisfies
\begin{equation}
  w_{C,\tau} - \cL w_\cyl = \cE[w_\cyl] +\bar{\cE} [w,\varphi_\cyl]
  \label{eq-w-c-evolution}
\end{equation}
where the operator $\cE$ is again defined by \eqref{eq-E10} and where the new error term $\bar\cE$ is given by the commutator
\[
\bar\cE[w,\varphi_\collar] = \bigl[ \pd_\tau - (\cL +\cE), \varphi_\collar \bigr]w,
\]
i.e.
\begin{multline}
  \label{eq-bar-E}
  \bar{\cE}[w,\varphi_\cyl] = \\
  \Biggl \{\varphi_{\cyl,\tau} - \varphi_{\cyl, yy} + \frac{u_{1y}^2}{1+u_{1y}^2}\, \varphi_{\cyl, yy}
  + \frac{(u_{1y}+u_{2y})u_{2yy}}{(1+u_{1y}^2)(1+u_{2y}^2)} (\varphi_\cyl)_y + \frac y2  (\varphi_\cyl)_y\Biggr \}   w  \\
  + \biggl \{\frac{2u_{1y}^2}{1+u_{1y}^2} (\varphi_\cyl)_y - 2(\varphi_\cyl)_y\biggr\} w_y.
\end{multline}

Equation \eqref{eq-w-c-evolution} for $w_\cyl$ is not self contained because of the last term $\bar{\cE}[w,\varphi_\cyl]$, which involves $w$ rather than $w_\cyl$.
The extra non local term is supported in the intersection of the cylindrical and tip regions because all the terms in it involve derivatives of $\varphi_\cyl$, but not $\varphi_\cyl$ itself.

Let us abbreviate the right hand side in \eqref{eq-w-c-evolution} to
\[
g := \cE[w_\cyl] + \bar{\cE}[w,\varphi_\cyl].
\]
Apply Lemma \ref{lem-linear-cylindrical-estimates-sup-L2-version} to $w_\cyl$ solving \eqref{eq-w-c-evolution}, to conclude that there exist $\tau_0 \ll -1$ and constant $C_* > 0$, so that if the parameters $(\alpha, \beta, \gamma)$ are chosen to ensure that $\pr_+ w_\cyl(\tau_0) = 0$, then $\hat w_\cyl := \pr_+ w_\cyl + \pr_- w_\cyl$ satisfies the estimate
\begin{equation}
  \label{eq-first-step}
  \|\hat{w}_\cyl\|_{\hv,\infty} \le C_* \, \|g\|_{\hv^*,\infty}
\end{equation}
for all $\tau \le \tau_0$.

In the next two lemmas we focus on estimating $\|g\|_{\hv^*}$.

\begin{lemma}
  \label{lem-error1-est}
  For every $\epsilon > 0$ there exist a $\tau_0$ and a uniform constant $C$ so that for $\tau\le \tau_0$ we have
  \[
  \|\cE[w_\cyl]\|_{\hv^*} \le \epsilon\, \|w_\cyl\|_{\hv}.
  \]
\end{lemma}

\begin{proof}
  Recall that
  \[\cE[w_\cyl]
  = -\frac{u_{1y}^2}{1+u_{1y}^2} (w_\cyl)_{yy} -\frac{(u_{1y}+u_{2y})u_{2yy}}{(1+u_{1y}^2)(1+u_{2y}^2)} (w_\cyl)_y +\frac{2(n-1) - u_1u_2}{2u_1u_2}w_\cyl.\] In \cite{ADS} we showed that for $\tau \leq \tau_0 \ll-1$
  \begin{equation}
    \label{eq-apriori-bounds}
    |(u_i)_y| + |(u_i)_{yy}| + |(u_i)_{yyy}| \le \frac{\bar{C}(\theta)}{\sqrt{|\tau|}}, \qquad \mbox{for} \,\,\, (y,\tau) \in \cyl_{\theta}
  \end{equation}
  where $u_i, i=1,2$ is any of the two considered solutions.
  The constant $\bar C(\theta)$ depends on $\theta$ and may change from line to line, but it is independent of $\tau$ as long as $\tau \leq \tau_0 \ll -1$.
  
  Using \eqref{eq-apriori-bounds} and Lemma \ref{lem-bounded-second-order} we have,
  \begin{equation}
    \label{eq-term-one}
    \Bigl\|\frac{u_{1y}^2}{1+u_{1y}^2} (w_\cyl)_{yy}\Bigr\|_{\hv^*} \le \frac{\bar{C}(\theta)}{|\tau|} \|(w_\cyl)_{yy}\|_{\hv^*} \le \frac{\bar{C}(\theta)}{|\tau|}\, \|w_\cyl\|_{\hv}.
  \end{equation}
  while by \eqref{eq-apriori-bounds} and Lemma \ref{lem-bounded-first-order} we have,
  \begin{equation}
    \label{eq-term-two}
    \Bigl\|\frac{(u_{1y} + u_{2y}) u_{2yy}}{(1+u_1y^2) (1+ u_{2y}^2)}\, (w_\cyl)_y\Bigr\|_{\hv^*} \le \frac{\bar{C}(\theta)}{|\tau|} \|(w_\cyl)_y\|_{\hv^*} \le \frac{\bar{C}(\theta)}{|\tau|}\, \|w_\cyl\|_{\hilb}.
  \end{equation}
  Also,
  \[
  \Bigl \|\frac{(2(n-1) - u_1 u_2)}{2u_1 u_2}\, w_\cyl\Bigr \|_{\hv^*} \le \Bigl \| \frac{(2(n-1) - u_1^2)}{2u_1 u_2}\, w_\cyl\Bigr \|_{\hv^*} + \Bigl \|\frac{(u_1 - u_2)}{2u_2}\, w_\cyl\Bigr\|_{\hv^*}.
  \]
  It is very similar to deal with either of the terms on the right hand side, so we explain how to deal with the first one next: Lemma \ref{lem-bounded-first-order}, the uniform boundedness of our solutions and the fact that $u_i \ge \theta/4$ in $\mathcal{C}$ for $i\in \{1,2\}$, give
  \[
  \begin{split}
    \Bigl \| \frac{(2(n-1) - u_1^2)}{2u_1 u_2}\, w_\cyl\Bigr\|_{\hv^*} &\le \frac{\bar{C}(\theta)}{\theta^2}\|(\sqrt{2(n-1)} - u_1) \, w_\cyl\|_{\hv^*} \\
    &\le \frac{\bar{C}(\theta)}{\theta^2}\Bigl \|\frac{(\sqrt{2(n-1)} - u_1)}{y+1}\, w_\cyl\Bigr \|_{ \hilb}.
  \end{split}
  \]
  Then, for any $K >0$ we have
  \[
  \begin{split}
    \Bigl\| \frac{(2(n-1) - u_1^2)}{2u_1 u_2}\, w_\cyl\Bigr\|_{\hv^*} &\le
    \frac{\bar{C}(\theta)}{\theta^2} \, \int_{0 \le y \le K} \frac{(\sqrt{2(n-1)} - u_1)^2}{(y+1)^2}\, w_\cyl^2 \, e^{-\frac{y^2}{4}}\, dy \\
    &+ \frac{\bar{C}(\theta)}{\theta^2}\, \int_{y \ge K} \frac{(\sqrt{2(n-1)} - u_1)^2}{(y+1)^2}\, w_\cyl^2 \, e^{-\frac{y^2}{4}}\, dy.
  \end{split}
  \]
  Now for any given $\epsilon >0$ we choose $K$ large so that ${\displaystyle \frac{\bar{C}(\theta)}{\theta^2 K^2} < \frac \epsilon{6} }$, and then for that chosen $K$ we choose a $\tau_0 \ll -1$ so that ${\displaystyle \frac{\bar{C}(\theta)}{\theta^2}(\sqrt{2(n-1)} - u_1) < \frac \epsilon{6}}$ for all $\tau \le \tau_0$ and $0 \le y \le K$ (note that here we use that $u_i(y,\tau)$ converges uniformly on compact sets in $y$ to $\sqrt{2(n-1)}$, as $\tau\to -\infty$).
  We conclude that for $\tau \geq \tau_0$
  \begin{equation}
    \label{eq-term-three}
    \Bigl\| \frac{(2(n-1) - u_1^2)}{2u_1 u_2}\, w_\cyl\Bigr\|_{\hv^*} \leq \frac{\epsilon}3 \, \|w_\cyl\|_{\hilb} \leq \frac{\epsilon}3 \, \|w_\cyl\|_{\hv}.
  \end{equation}
  Finally combining \eqref{eq-term-one}, \eqref{eq-term-two} and \eqref{eq-term-three} finishes the proof of Lemma.
\end{proof}

We will next estimate the error term $\bar{\cE}[w,\varphi_\cyl]$.

\begin{lemma}
  \label{lem-error-bar}
  There exists a $\tau_0 \ll -1$ and $\bar{C}(\theta)$ so that for all $\tau \le \tau_0$ we have
  \[
  \|\bar{\cE}[w,\varphi_\cyl]\|_{\hv^*} \le \frac{\bar{C}(\theta)}{\sqrt{|\tau_0|}}\, \|\chi_{D_{\theta}}\, w\|_{\hilb}
  \]
  where $\bar{\cE}[w,\varphi_\cyl]$ is defined by \eqref{eq-bar-E} and $\chi_{D_{\theta}}$ is the characteristic function of the set $D_\theta:= \{ \theta/2  < u < \theta \}$.
\end{lemma}

\begin{proof}
  Setting
  \[
  a(y,\tau) := \varphi_{\cyl,\tau} - \varphi_{\cyl, yy} + \frac{u_{1y}^2}{1+u_{1y}^2}\, \varphi_{\cyl, yy} + \frac{(u_{1y}+u_{2y})u_{2yy}}{(1+u_{1y}^2)(1+u_{2y}^2)} \varphi_{\cyl,y}
  \]
  and
  \[
  b(y,\tau) := (\varphi_\cyl)_y \qquad \mbox{and} \qquad d(y,\tau) := \frac{2u_{1y}^2}{1+u_{1y}^2} (\varphi_\cyl)_y - 2(\varphi_\cyl)_y
  \]
  we may write
  \begin{equation}
    \label{eq-rewriting-error}
    \bar{\cE}[w,\varphi_\cyl] = a(y,\tau) w + \frac y2 \, b(y,\tau)\, w + d(y,\tau) \, w_y.
  \end{equation}
  Note that the support of all three functions, $a(y,\tau)$, $b(y,\tau)$ and $d(y,\tau)$ is contained in $D_\theta$ and
  \[
  |a(y,\tau)| + |b(y,\tau)| + |d(y,\tau)| \le \frac{\bar{C}(\theta)}{\sqrt{|\tau|}}.
  \]
  Furthermore, by \eqref{eq-apriori-bounds} and Lemma \ref{lem-bounded-first-order} we get
  \[
  \|a(y,\tau) \, w\|_{\hv^*} \le \|a(y,\tau) w\|_{\hilb} \le \frac{\bar{C}(\theta)}{\sqrt{|\tau|}}\, \|w \chi_{D_{\theta}}\|_{\hilb},
  \]
  \[
  \|\frac y2\, b(y,\tau) w\|_{\hv^*} \le \|b(y,\tau) \, w\|_{\hilb} \le \frac{\bar{C}(\theta)}{\sqrt{|\tau|}}\, \|w \, \chi_{D_{\theta}}\|_{\hilb}
  \]
  and
  \[
  \begin{split}
    \|d(y,\tau) \, w_y\|_{\hv^*} &\le \|(d(y,\tau) w)_y\|_{\hv^*} + \|w d_y(y,\tau)\|_{\hv^*}\\
    &\le \|d(y,\tau) \, w\|_{\hilb} + \frac{\bar{C}(\theta)}{\sqrt{|\tau|}} \|w\, \chi_{D_{\theta}}\|_{\hilb} \\
    &\le \frac{\bar{C}(\theta)}{\sqrt{|\tau|}} \|w\, \chi_{D_{\theta}}\|.
  \end{split}
  \]
  The above estimates together with \eqref{eq-rewriting-error} readily imply the lemma.
\end{proof}

Finally, we now employ all the estimates shown above to conclude the proof of Proposition \ref{prop-cylindrical}.

\begin{proof}[Proof of Proposition \ref{prop-cylindrical}]
  By \eqref{eq-first-step} with $g := \cE[w_\cyl] + \bar{\cE}[w,\varphi_\cyl]$
  and using also Lemma \ref{lem-error1-est}, Lemma \ref{lem-error-bar} and the
  assumption that $\pr_+ w_\cyl (\tau_0) = 0$, we have that for every $\epsilon >
  0$ there exists a $\tau_0 \ll -1$ so that
  \[
  \|\hat{w}_C\|_{\hv,\infty}
  \le \epsilon\, \|w_\cyl\|_{\hv,\infty}
  + \frac{\bar{C}(\theta)}{\sqrt{|\tau_0|}}\,
  \|w \chi_{D_{\theta}}\|_{\hilb,\infty}.
  \]
  This readily gives the proposition.
\end{proof}


\section{The tip region}
\label{sec-tip}

Let $u_1(y,\tau)$ and $u_2(y,\tau)$ be the two solutions to equation \eqref{eq-u} as in the statement of Theorem \ref{thm-main} and let $u_2^{\alpha\beta\gamma}$ be defined
by \eqref{eq-ualphabeta}.
We will now estimate the difference of these solutions in the tip region which is defined by $\tip_{\theta} = \{(y,\tau)\,\,\,|\,\,\, u_1 \le 2\theta\}$, for $\theta >0$ sufficiently small, and $\tau \le \tau_0 \ll -1$, where $\tau_0$ is going to be chosen later.
In the tip region we need to switch the variables $y$ and $u$ in our both solutions, with $u$ becoming now an independent variable.
Hence, our solutions become $Y_1(u,\tau)$ and $Y_2^{\alpha\beta\gamma}(u,\tau)$.  
Define the difference $ W:=Y_1-Y_2^{\alpha\beta\gamma}$
and for a standard cutoff function $\varphi_T(u)$  as in \eqref{eqn-cutofftip}  we denote $W_T := \varphi_T\, W$.

\begin{rem}
  \label{rem-tip}
  By the change of variables \eqref{eqn-Y-expansion} and by the definition of $u_2(y,\tau) := u_2^{\alpha\beta\gamma}(y,\tau)$  as in \eqref{eq-ualphabeta}, we have that
  \[
  Z_2^{\alpha\beta\gamma}(\rho,\tau) = \sqrt{|\tau|}\, \Bigl(Y_2^{\alpha\beta\gamma}\big (\frac{\rho}{\sqrt{|\tau|}},\tau\big ) - Y_2^{\alpha\beta\gamma}(0,\tau)\Bigr)
  \]
  where
  \[
  Y_2^{\alpha\beta\gamma}(u,\tau) = \sqrt{1+\beta e^{\tau}}\, Y_2\big (\frac{u}{\sqrt{1+\beta e^{\tau}}}, \sigma\big ), \quad \sigma:= \tau+\gamma-\log(1+\beta e^{\tau}).
  \]
  Combining the above two equations yields
  \[Z_2^{\alpha\beta\gamma}(\rho,\tau) = \frac{\sqrt{|\tau|} \, \sqrt{1+\beta e^{\tau}}}{\sqrt{|\sigma|}}\, Z_2 \big (\rho\, \frac{\sqrt{|\sigma|}}{\sqrt{|\tau|}\sqrt{1+\beta e^{\tau}}}, \sigma \big ).
  \]
  Recall that $\alpha=\alpha(\tau_0)$, $\beta=\beta(\tau_0)$, $\gamma=\gamma(\tau_0)$ will be chosen in Section \ref{sec-conclusion} so that $(\alpha, \beta,\gamma)$ is admissible with respect to $\tau_0$.
  Using the fact that $Z_2(\rho,\tau)$ converges as $\tau\to -\infty$, uniformly smoothly on compact sets in $\rho$, to the Bowl soliton $Z_0(\rho)$ we have
  \[
  \begin{split}
    Z_2^{\alpha\beta\gamma}&(\rho,\tau) = (1 + o(1))\, \Bigl \{  Z_2(\rho, \sigma) + \Bigl  (  Z_2 \big (\rho\, \frac{\sqrt{|\sigma|}}{\sqrt{|\tau|}\sqrt{1+\beta e^{\tau}}}, \sigma \big ) - Z_2(\rho, \sigma) \Bigr) \Bigr \} \\
    &= (1 + o(1))\, \Bigl \{ Z_0(\rho) + o(1) + (Z_2)_\rho (\hat{\rho},\sigma) \, \rho \, \Bigl (\frac{\sqrt{\sigma|}}{\sqrt{|\tau|}\, \sqrt{1+\beta e^{\tau}}} - 1\Bigr ) \Bigr \}
  \end{split}
  \]
  where $o(1)$ denote functions that may differ from line to line, but are uniformly small for all $\tau\le \tau_0 \ll -1$ and for all $(\alpha,\beta,\gamma)$  that are admissible with respect to $\tau_0$.
  Note also that above we applied the mean value theorem, with $\hat{\rho}$ being a value in between $\rho$ and $\rho\, \frac{\sqrt{|\sigma|}}{\sqrt{|\tau|}\sqrt{1+\beta e^{\tau}}} = \rho \, (1 + o(1))$.
  By the monotonicity of $(Z_2)_\rho(\cdot,\sigma)$ in $\rho$, we see that for $\tau_0$ sufficiently small we have
  \[
  (Z_2)_\rho (\rho + \epsilon, \sigma) \le (Z_2)_\rho(\hat{\rho}, \sigma) \le (Z_2)_\rho(\rho -\epsilon, \sigma)
  \]
  for some small $\epsilon$ and all $\tau \le \tau_0$, implying
  \[
  (Z_0)_\rho (\rho+\epsilon) + o(1) \le (Z_2)_\rho (\hat{\rho}, \sigma) \le (Z_0)_\rho (\rho-\epsilon) + o(1)
  \]
  for $\tau \le \tau_0$ and $\tau_0 \ll -1$.
  All these together with $ \lim_{\tau\to -\infty} \frac{\sqrt{|\sigma|}}{\sqrt{|\tau|}\sqrt{1+\beta e^{\tau}}} = 0 $ imply that $ Z_2^{\alpha\beta\gamma}(\rho,\tau) = Z_0(\rho) + o(1)$, where $o(1)$ is a function that is, as before, uniformly small for all $\tau\le \tau_0 \ll -1$ and all $\alpha, \beta$ and $\gamma$ that are admissible with respect to $\tau_0$.
  
  Hence, it is easy to see that in all the estimates below, in this section, we can find a uniform $\tau_0 \ll -1$, independent of parameters $\alpha, \beta$ and $\gamma$ (as long as they are admissible with respect to $\tau_0$), so that all the estimates below hold for $Y_1(u,\tau) - Y_2^{\alpha\beta\gamma}(u,\tau)$, for all $\tau \le \tau_0$.
\end{rem}

Our goal in this section is to show the following bound.

\begin{prop}\label{prop-tip} There exist $\theta$ with
  $0 < \theta \ll 1$, $\tau_0 \ll -1$ and $C< +\infty$ such that
  \begin{equation}
    \label{eqn-tip}
    \| W_T \|_{2,\infty} \leq \frac{C}{|\tau_0| } \, \| W \, \chi_{_{[\theta, 2\theta]} } \|_{2,\infty}
  \end{equation}
  holds.
\end{prop}

To simplify the notation  throughout this section we will drop the subscript on $Y_1$ and write $Y=Y_1$ instead.
Also, we will denote $Y_2^{\alpha\beta\gamma}$ by $Y_2$.
As already explained in Section \ref{subsec-tip}, the proof of this proposition will be based on a Poincar\'e inequality for the function $W_T$ which is supported in the tip region.
These estimates will be shown to hold with respect to an appropriately chosen weight $e^{\mu(u,\tau)}\, du$,  where $\mu(u,\tau)$ is given by \eqref{eq-weight}.
We will begin by establishing various properties on the weight $\mu(u,\tau)$.
We will continue with the proof of the Poincar\'e inequality and we will finish with the proof of the Proposition.
Recall that the definitions of the {\em collar region} $\collar_{L,\theta}$ and the {\em soliton region} $\cS_L$ are given in Section \ref{subsec-tip}.

\subsection{Properties of $\mu(u,\tau)$}
In a few subsequent lemmas we show estimates for the weight $\mu(u,\tau)$, which is given by \eqref{eq-weight}.
Recall that in the soliton region $\mathcal{S}_L$ we have defined $\mu(u,\tau) := m(\rho) + a(L,\tau)\, \rho + b(L,\tau)$, where $a(L,\tau)$ and $b(L,\tau)$ are given by \eqref{eqn-aL} and \eqref{eqn-bL} respectively.
\begin{lemma}
  \label{lemma-prop-mu}
  For all sufficiently large $L$ the limit \(a_\infty(L) = \lim_{\tau\to-\infty} a(L, \tau)\) exists.
  Moreover, there is a constant $C<\infty$ such that
  \[
  |a_\infty(L)|\leq CL^{-1}.
  \]
  In particular, for every $\eta > 0$ there exist an $L_0$ so that for every $L \ge L_0$, there exists a $\tau_0 \ll -1$ such that
  \[
  |a(L,\tau)| \le \eta \text{ for all $\tau \leq \tau_0$.  }
  \]
\end{lemma}

\begin{proof}
  Recall that
  \[
  a(L,\tau) := -m'(L) - \frac{1}{2\sqrt{|\tau|}}Y\, Y_{u}\bigl({L / \sqrt{|\tau|}},\tau\bigr),
  \]
  where
  \[
  Y(u,\tau) = Y(0,\tau) + \frac{1}{\sqrt{|\tau|}}\, Z(\rho,\tau).
  \]
  Using $Y(0,\tau) = \sqrt{2|\tau|}(1 + o(1))$,
  $Y_{u}(u,\tau) = Z_{\rho}(\rho, \tau)$
  and $Z(\rho, \tau)\to Z_0(\rho)$ for $\tau\to-\infty$,
  we get
  \[
  a(L, \tau) = -m'(L) - \frac12\sqrt{2|\tau|}(1 + o(1))
  \frac{1}{\sqrt{|\tau|}} Z_\rho(L, \tau)
  \]
  so
  \[
  \lim_{\tau\to-\infty} a(L,\tau) 
  = a_\infty(L) 
  = -m'(L) - \frac12 \sqrt2 \, Z_0'(L).
  \]
  Since $m'(L) = \frac{n - 1}{L}\, (1 + Z_{0}'(L)^2)$, we have
  \[
  a_\infty(L) 
  = -\frac{n-1}{L} -\frac{n - 1}{L}\, Z_{0}'(L)^2 - \frac{1}{2}\sqrt{2}\, Z_{0}'(L).
  \]
  The asymptotic expansion \eqref{eq-Z0-asymptotics} for $Z_0(\rho)$ as $\rho\to\infty$ then implies \(a_\infty(L) = \cO(L^{-1})\) as \(L\to-\infty\).
\end{proof}

In the following lemma we prove further properties of $\mu(u,\tau)$ that will be used later in the text.

\begin{lemma}
  Fix $\eta > 0$ small.
  There exist
  $\theta > 0$, $L > 0$ and $\tau_0 \ll -1$ so that
  \begin{equation}\label{eqn-mut}
    \mu_\tau \leq \eta \, |\tau| \qquad \mbox{holds on} \,\,\,\, 0\le u \le 2\theta
  \end{equation}
  and
  \begin{equation}\label{eqn-mus}
    1- \eta \leq \frac{u\, \mu_u}{n-1}\, \frac{1}{1+Y_{u}^2} \leq 1+ \eta, \quad 1-\eta \le \frac{2(n-1)\,\mu_u}{u|\tau|} < 1 + \eta
  \end{equation} holds on $\collar_{\theta,L}$ and for all $\tau \le \tau_0$.
\end{lemma}

\begin{proof}
  To prove \eqref{eqn-mut} we first deal with the {\it collar region} ${\collar}_{\theta,L}$.
  By \eqref{eqn-Y} and \eqref{eq-weight} we have
  \[
  \mu_{\tau} = -\frac{Y Y_{\tau}}{2} = -\frac{Y}{2}\, \Bigl(\frac{Y_{uu}}{1+Y_u^2} + \big(\frac{n-1}{u} - \frac u2\big)\, Y_u + \frac Y2\Bigr).
  \]
  By Proposition \ref{prop-ratio-small} we have that for every $\eta > 0$ there exist $\theta, L > 0$ and $\tau_0 < 0$ so that 
  ${\displaystyle 
  {\lambda_1}/{\lambda_2} < \frac{\eta}{100}}$,  on $\collar_{\theta,L}$ and for $\tau \leq \tau_0$,
  implying the bound 
  \[
  \frac{|Y_{uu}|}{1 + Y_u^2} \le \frac{\eta}{100}\, \frac{|Y_u|}{u}, \qquad \mbox{on} \,\,\,\,\collar_{\theta,L},\,\, \tau \leq \tau_0.
  \]
  Using \eqref{eq-cor-Sigurd} and the previous estimate yields
  \[
  \frac{Y}{2}\, \frac{|Y_{uu}|}{1 + Y_u^2} \le \frac{\eta}{100}\, \frac{Y^2}{4(n-1)}\, (1 + \epsilon) < \frac{\eta}{50}\, |\tau|,
  \]
  if $\theta$ is chosen sufficiently small and $L$ sufficiently big (note that we used $Y(u,\tau) \le Y(0,\tau) = \sqrt{2|\tau|}\, (1 + o_{\tau}(1))$).
  Furthermore, by Corollary \ref{lemma-Sigurd} we have
  \begin{align*}
    -\frac Y2 \Bigl\{ \Bigl(\frac{n-1}{u} - \frac u2\Bigr)\, Y_u + \frac Y2\Bigr\}
    &\le -\frac Y2 \, \Bigl(\frac{n-1}{u} Y_u + \frac Y2\Bigr) \\
    &= \frac{Y |Y_u|(n-1)}{2u}\, \Bigl(1 + \frac{u Y}{2(n-1)\, Y_u}\Bigr) \\
    &< \frac{Y |Y_u| (n-1)}{2u}\, \epsilon(\theta,L) \le \tilde{C}\, |\tau| \epsilon(\theta,L) < \frac{\eta}{2}\, |\tau|.
  \end{align*}
  We conclude that  \eqref{eqn-mut} holds in the collar region $\collar_{\theta,L}$.  
  
  To estimate $\mu_{\tau}$ in the \emph{soliton region} $\cS_L$, where $\rho \le L$, we note that \eqref{eq-weight} implies 
  \[
  \mu_{\tau} 
  = \frac{d}{d \tau}  a(L,\tau)\, \rho + \frac{d}{d\tau} b(L,\tau).
  \]
  By \eqref{eqn-aL} and \eqref{eqn-bL}, we have that $b(L,\tau) = - m(L) - L\, a(L,\tau)$ and hence, 
  \[
  | \mu_{\tau} | = \big | \frac{d}{d \tau} a(L,\tau)\, ( \rho - L ) \big | \leq L\, \big | \frac{d}{d \tau} a(L,\tau) \big |.
  \]
  Now using the definition of $a(L,\tau)$ in \eqref{eqn-aL}, we have  
  \[
  \begin{split}
    \Bigl| \frac{d}{d\tau} & \, a(L,\tau)\Bigr| = \frac{1}{4|\tau|^{3/2}}\, |Y Y_{u}| +\\
    &+ \frac{1}{2\sqrt{|\tau|}}\Bigl( |Y_{\tau} Y_{u}| + |Y Y_{u\tau}| + \frac{L}{4|\tau|^{3/2}} Y_{u}^2 + \frac{L}{4|\tau|^{3/2}} |YY_{uu}|\Bigr),
  \end{split}
  \]
  where all terms on the right hand side in above equation are computed at
  $\bigl(L/\sqrt{|\tau|},\tau\bigr)$.
  Let us estimate all these terms.
  While
  doing so we will use \eqref{eqn-Y-expansion} and the smooth convergence of
  $Z(\rho,\tau)$, as $\tau\to -\infty$, to the Bowl soliton $Z_0(\rho)$.
  For
  example,
  \[
  \frac{|Y\,Y_u|}{|\tau|^{3/2}} \le C\, \frac{|Z_{\rho}|}{|\tau|} \ll \frac{\eta}{100} \, |\tau|,
  \]
  by choosing $\tau_0 \ll -1$.
  Furthermore, using \eqref{eqn-Y}  we have
  \[
  \frac{|Y_{\tau} Y_u|}{2\sqrt{|\tau|}} = \frac{|Z_{\rho}|}{2\sqrt{|\tau|}} \Bigl|\frac{Z_{\rho\rho}\, \sqrt{|\tau|}}{1 + Z_{\rho}^2} + \frac{(n-1)\sqrt{|\tau|}}{\rho}\, Z_{\rho}^2 - \frac{\rho}{2\sqrt{|\tau|}} Z_{\rho} + \frac{Y}{2}\Bigr|,
  \]
  leading to 
  \[
  \frac{|Y_{\tau} Y_u|}{2\sqrt{|\tau|}} \le C(L) \ll \frac{\eta}{100} |\tau|,
  \]
  for $\tau \leq \tau_0 \ll -1$.
  Next,
  \[
  \frac{L}{8|\tau|^2} Y_{u}^2 = \frac{L}{8|\tau|^2} Z_{\rho}^2 \le \frac{C(L)}{|\tau|^2} \ll \frac{\eta}{100}\, |\tau|,
  \]
  and
  \[
  \frac{L|Y Y_{uu}|}{8\,|\tau|^2} \le C(L)\frac{|Z_{\rho\rho}|}{|\tau|} \ll \eta\, |\tau|,
  \]
  for $\tau\le \tau_0 \ll -1$ sufficiently small.
  Finally, differentiating equation \eqref{eqn-Y} in $u$ and using \eqref{eqn-Y-expansion} we have
  \[
  \frac{|Y_{u\tau}|}{\sqrt{|\tau|}} = \frac{1}{\sqrt{|\tau|}}\, \Bigl|\Bigl(\frac{Z_{\rho\rho}}{1 + Z_{\rho}^2}\Bigr)_{\rho}\, |\tau| + \Bigl(\Bigl(\frac{(n-1)|\tau|}{\rho} - \rho\Bigr)\, Z_{\rho}\Bigr)_{\rho} + \frac{Z_{\rho}}{2}\Bigr| \le C(L)\, \sqrt{|\tau|} \ll \frac{\eta}{100}\, |\tau|,
  \]
  for $\tau \le \tau_0 \ll -1$.
  Combining the last estimates we conclude that   \eqref{eqn-mut} holds also in the soliton region.
  Combining the two estimates in the collar and soliton regions yilelds \eqref{eqn-mut}. 
  
  To prove the first estimate in \eqref{eqn-mus} note that
  \[
  \frac{u\mu_u}{(n-1)\, (1 + Y_{u}^2)} = -\frac{u\, Y Y_{u}}{2(n-1)\, (1+ Y_{u}^2)}.
  \]
  Using \eqref{eq-cor-Sigurd} we have that for every $\eta > 0$ we can choose $\theta \ll 1$ small and $L \gg 1$ big and $\tau_0 \ll -1$ so that
  \[
  (1 - \frac{\eta}{2})\, \frac{|Y_{u}|^2}{1 + Y_{u}^2} < \frac{u\mu_u}{(n-1)\, (1 + Y_{u}^2)} < (1 + \frac{\eta}{2})\, \frac{|Y_{u}|^2}{1 + Y_{u}^2}.
  \]
  Since $|Y_{u}|$ is large in $\collar_{\theta,L}$, we get that
  \[
  1 - \eta < \frac{u\mu_u}{(n-1)} < 1 + \eta, \qquad \mbox{in} \,\,\,\,\collar_{\theta,L},
  \]
  for $\theta \ll 1$, $L \gg 1$ and $\tau \le \tau_0 \ll -1$.
  
  To prove the second estimate in \eqref{eqn-mus}, note that
  \[
  \frac{2(n-1)\, \mu_u}{u|\tau|} = \frac{(n-1)Y |Y_{u}|}{u\, |\tau|},
  \]
  and use \eqref{eq-cor-Sigurd} together with the fact that $Y = \sqrt{|\tau|}\, (\sqrt{2} + o_{\tau,\theta}(1))$, where the limit $ \lim_{\tau\to-\infty, \theta\to 0} o_{\tau,\theta}(1) = 0$.
\end{proof}

\subsection{Poincar\'e inequality in the tip region}
We will next show a {\em weighted Poincar\'e type estimate}  (with respect to weight $\mu(u,\tau)$ defined in \eqref{eq-weight})  that will be needed in obtaining the coercive type estimate \eqref{eqn-tip} in the {\em tip region}  $\tip_{\theta}$.
As we discussed earlier, near the tip we switch the variables $y$ and $u$ in both solutions, with $u$ becoming now an independent variable.

\begin{proposition}
  \label{prop-Poincare}
  There exist uniform constants $C > 0$ and $C(\theta) > 0$, independent of  $\theta$, and $\tau_0$, so that for $\theta \le \theta_0$, and $\tau \le \tau_0$, for every compactly supported function $f$ in $\tip_{\theta}$ we have
  \begin{equation}
    \label{eq-Poincare}
    |\tau|\int_0^{\theta} f^2(u) \, e^{\mu(u,\tau)}\, du \le C\,\int_0^{2\theta} \frac{f_u^2}{1 + Y_{u}^2}\, e^{\mu(u,\tau)}\, du + \int_{\theta}^{2\theta} f^2\, e^{\mu(u,\tau)}\, du.
  \end{equation}
\end{proposition}

\begin{proof}
  We divide the proof in several steps.
  In Step \ref{step-Poincare-pretip} we show the weighted Poincar\'e inequality for compactly supported functions in $\collar_{\theta,\frac L2}$.
  In Step \ref{step-Poincare-verytip} we show the weighted Poincar\'e inequality for compactly supported functions in $\rho = u\sqrt{|\tau|} \in [0,\infty)$.
  In Step \ref{step-final} we use cut off functions to show \eqref{eq-Poincare}.
  \begin{step} \label{step-Poincare-pretip}
    We will first derive the weighted Poincar\'e inequality in the collar region, $\collar_{\theta, \frac L2}$, for $\theta$ small, $L$ big and $\tau \ll -1$.
  \end{step} Let $f(u)$ be a compactly supported function in $\cyl_{\theta,\frac L2}$.
  We claim we have
  \begin{equation}
    \label{eq-help-111}
    1+Y_{u}^2 \leq \frac 32 u \mu_u, \qquad \mbox{in} \,\,\,\collar_{\theta,\frac L2}.
  \end{equation}
  To show \eqref{eq-help-111}, lets first consider the case when $u\in\collar_{\theta,L}$.
  By \eqref{eqn-mus} we have
  \[
  1 + Y_{u}^2 \le (1 - \eta)\, \frac{u \mu_u}{n-1} \le \frac 32\, u\mu_u, \qquad \mbox{in} \,\,\,\,\collar_{\theta,L}.
  \]
  To finish the proof of \eqref{eq-help-111} we need to check the estimate holds for $u \in \Bigl[\frac{L}{2\sqrt{|\tau|}}, \frac{L}{\sqrt{|\tau|}}\Bigr]$, or equivalently, for 
  $\rho \in [L/2, L]$ as well.
  Recall that in this region $\mu(u,\tau) = m(u\sqrt{|\tau|}) + a(L,\tau) u \sqrt{|\tau|} + b(L,\tau)$, and hence
  \[
  u \mu_u = \rho \, m_{\rho} + a(L,\tau)\, \rho = (n - 1)\, (1 + (Z_0)_{\rho}^2) + a(L,\tau)\, \rho.
  \]
  By part (b) of Lemma \ref{lemma-prop-mu} we can make $|a(L,\tau)|$ as small as we want by taking $L$ sufficiently big and $\tau \le \tau_0 \ll -1$ sufficiently small.
  Moreover, using the asymptotics for $Z_0(\rho)$ and its derivatives one concludes that for $\rho \in  [ L/2, L]$, we have
  \[
  u\mu_u \ge (n - 1) (1 - \epsilon)\, ( 1 + (Z_0)_{\rho}^2).
  \]
  On the other hand, denote by $Z(\rho,\tau)$ a solution with respect to $\rho$ variable that corresponds to $Y(u,\tau)$, via rescaling \eqref{eqn-Y-expansion}.
  By results in \cite{ADS} we know $Z(\rho,\tau)$ converges uniformly smoothly on compact sets in $\rho$ to the Bowl soliton, $Z_0(\rho)$.
  This and the fact that $Y_{u} = Z_{\rho}$ yield
  \[
  u\, \mu_u \ge (n - 1) (1 - 2\epsilon)\, (1 + Y_{u}^2),
  \]
  for $L$ sufficiently big and $\tau \le \tau_0$, where $\tau_0 \ll -1$ is sufficiently small.
  This implies
  \[
  1 + Y_{u}^2 \le \frac 32 u \mu_u, \qquad \mbox{for} \,\,\,\, u \in \Bigl[\frac{L}{2\sqrt{|\tau|}}, \frac{L}{\sqrt{|\tau|}}\Bigr],
  \]
  hence concluding the proof of \eqref{eq-help-111}.
  Using this estimate, for any $f$ that is compactly supported in $\collar_{\theta,\frac L2}$ we have
  \begin{equation}\label{eqn-muu3}
    \int \frac{f^2}{u^2} \, (1+Y_{u}^2) \, e^{\mu(u,\tau)}\, du \leq \frac 32 \int \frac{f^2}{u} \, \mu_u \, e^{\mu(u,\tau)}\, du.
  \end{equation}
  \smallskip Furthermore,
  \begin{multline}\label{eqn-poin5}
    \int \frac {f^2}{u} \, \mu_u \, e^{\mu(u,\tau)}\,du = \int \frac{f^2}{ u }\, \frac{ \partial } {\partial u} \big ( e^{\mu} \big ) \, du
    = -2\int \frac{f f_u}{u }\, e^{\mu(u,\tau)}\, du +   \int \frac {f^2} {u^2}  \, e^{\mu(u,\tau)}\, du\\
    \leq 2\, \int \frac {f^2_u}{1+Y_{u}^2} \, e^{\mu(u,\tau)}\, du + \frac 12 \int \frac{f^2}{u^2} \, (1+Y_{u}^2) \, e^{\mu(u,\tau)}\, du + \int \frac {f^2} {u^2} \, e^{\mu(u,\tau)}\, du.
  \end{multline}
  Also observe that in the considered region where $u^2 |\tau| \geq L \gg1$, using \eqref{eqn-muu3} we have
  \begin{multline*}
    \int \frac {f^2} {u^2} \, e^{\mu(u,\tau)}\, du = \int \frac {f^2}{u^2 |\tau|} \, \frac{\mu_u |\tau|}{\mu_u}\, e^{\mu(u,\tau)}\, du \leq \frac{1}{L^2} \int \frac {f^2}{u} \, \mu_u\, \frac{|\tau| u}{\mu_u} \, e^{\mu(u,\tau)}\, du
    \\
    \leq \frac 1{8} \, \int \frac {f^2}{u} \, \mu_u \,e^{\mu(u,\tau)}\, du.
  \end{multline*}
  Inserting this and \eqref{eqn-muu3} in \eqref{eqn-poin5}, finally yields
  \[
  \int \frac {f^2}{u} \, \mu_u \, e^{\mu(u,\tau)}\, du
  \leq 16\, \int \frac {f^2_u}{1+Y_{u}^2} \, e^{\mu(u,\tau)}\, du.
  \]
  If we choose $\eta < 1/2$, the previous estimate and \eqref{eqn-muu3} imply
  \begin{equation} \label{eq-poin-pretip}
    |\tau|\, \int f^2\, e^{\mu(u,\tau)}\, du \le 64(n - 1)\, \int \frac{f_u^2}{1 + Y_{u}^2}\, e^{\mu(u,\tau)}\, du,
  \end{equation}
  for any compactly supported function $f$ in $\collar_{\theta,\frac L2}$.
  Observe that the Poincar\'e constant in \eqref{eq-poin-pretip} is uniform, independent of $L$, $\theta$ and $\tau$.
  
  \begin{step}\label{step-Poincare-verytip}
    Denote by $\bar{\mu}(\rho,\tau):= \mu(u,\tau) = m(\rho) + a(L,\tau)\, \rho + b(L,\tau)$.
    We show there exists a $\delta > 0$ so that for all $f \in C_c^{\infty}([0,\infty))$, with $f'(0) = 0$ we have,
    \begin{equation}
      \label{eq-Poincare-verytip}
      \delta\, \int_0^{\infty} f^2 e^{\bar{\mu}}\, d\rho \le \int_0^{\infty} \frac{f_{\rho}^2}{1 + (Z_0)_{\rho}^2}\, e^{\bar{\mu}}\, d\rho.
    \end{equation}
  \end{step}
  
  To prove \eqref{eq-Poincare-verytip} we begin by establishing the inequality for functions supported on the interval $[A, \infty)$ for sufficiently large $A$.
  Then we argue by contradiction to extend the inequality to functions defined on $[0, \infty)$.
  
  Let $A<\infty$ be large and let and consider for $f\in C^\infty_c((A, \infty))$
  \[
  \int_A^\infty f^2 e^{\bar{\mu}} d\rho = \int_A^\infty \frac{f^2} {\bar{\mu}_\rho} de^{\bar{\mu}} = -\int_A^\infty \Bigl( \frac{2ff_\rho} {\bar{\mu}_\rho} - \bigl(\bar{\mu}_\rho^{-1}\bigr)_\rho\, f^2 \Bigr) e^{\bar{\mu}} \, d\rho.
  \]
  We use the asymptotic relation for $Z_0(\rho)$, which implies $m(\rho) = \rho^2/4(n-1) + o(\rho^2)$ and $m_\rho = \rho/2(n-1)+o(\rho)$, and part (b) of Lemma \ref{lemma-prop-mu} to conclude that $(\bar{\mu}_\rho^{-1})_\rho = -2(n-1)/\rho^2 + o(\rho^{-2})$, for large $\rho$.
  Continuing our estimate, we find for any $\epsilon>0$
  \begin{equation}
    \begin{aligned}
      \int_A^\infty f^2 e^{\bar{\mu}} d\rho & \leq \int_A^\infty \Bigl(\epsilon f^2 + \frac{1}{\epsilon}\, \frac{f_\rho^2}{\bar{\mu}_\rho^2} +\frac{C}{\rho^2} f^2
      \Bigr) \, e^{\bar{\mu}}\, d\rho\\
      & \leq \bigl(\epsilon + CA^{-2}\bigr) \int _A^\infty f^2e^{\bar{\mu}}d\rho + \frac{1}{\epsilon}\int_A^\infty \frac{f_\rho^2}{\bar{\mu}_\rho^2} e^{\bar{\mu}}d\rho.
    \end{aligned}
  \end{equation}
  Choose $\epsilon=1/4$, and let $A$ be so large that $C/A^2<1/4$, then we find
  \[
  \int f^2 e^{\bar{\mu}} d\rho \leq 8 \int_A^\infty \frac{f_\rho^2}{\bar{\mu}_\rho^2} e^{\bar{\mu}}d\rho.
  \]
  Finally, we note that for large $\rho$ both, $Z_{0\rho}$ and $\bar{\mu}_\rho$, are asymptotically proportional to $\rho$, so that $(1+Z_{0\rho}^2)^{-1} \leq C (\bar{\mu}_\rho)^{-2}$, and thus we have
  \begin{equation}
    \label{eq-Poincare-beyond-L}
    \int_A^\infty f^2 e^{\bar{\mu}} d\rho \leq C \int_A^\infty \frac{f_\rho^2} {1+Z_{0\rho}^2}\,e^{\bar{\mu}} d\rho.
  \end{equation}
  Therefore the Poincar\'e inequality holds for all $f$ supported in $[A, \infty)$.
  It is clear from the proof above that the Poincar\'e constant $C$ in \eqref{eq-Poincare-beyond-L} is a universal constant, independent of $L$.
  
  We now show that the inequality holds for all $f\in C_c^{\infty}([0,\infty))$.
  Suppose the inequality does not hold.
  Then there is a sequence of functions $f_n\in C_c^{\infty}([0,\infty))$, for which
  \[
  \int_0^\infty f_n(\rho)^2 e^{\bar{\mu}} d\rho =1, \quad\text{and}\quad \lim_{n\to\infty} \int_0^\infty \frac{f_n'(\rho)^2}{1+Z_{0\rho}^2} e^{\bar{\mu}} d\rho =0.
  \]
  Since the weight $S(\rho) := e^{\bar{\mu}}/(1+Z_{0\rho}^2)$ is a positive continuous function on $(0, \infty)$ the assumption $\int_0^\infty f_{n}'(\rho)^2 S(\rho)d\rho \to 0$ implies that $f_n$ is bounded in $H^1_{\rm loc}(\R_+)$, and thus that any subsequence has a further subsequence that converges locally uniformly.
  Moreover, any limit $f(\rho) = \lim f_{n_i}(\rho)$ must have $\int_0^\infty f'(\rho)^2 S(\rho)d\rho=0$, i.e.~must be constant, and, because $\int_0^\infty f_n^2 e^{\bar{\mu}}d\rho =1$ for all $n$, the limit must also satisfy $\int_0^\infty f(\rho)^2 e^{\bar{\mu}}d\rho \leq 1$.
  Since $\bar{\mu}\sim C\rho^2$ for large $\rho$, the only possible limit is $f(\rho)=0$.
  We conclude that if the sequence $f_n\in C_c^{\infty}([0,\infty))$ exists, then it must converge locally uniformly to $f(\rho)=0$.

  Choose $\varphi\in C^\infty([0,\infty))$ with $\varphi(\rho)=0$ for $\rho\leq L$ and $\varphi(\rho)=1$ for $\rho\geq 2L$.
  Then $\varphi f_n$ is supported in $[L,\infty)$, so that the Poincar\'e inequality \eqref{eq-Poincare-beyond-L} that we already have established implies
  \begin{align*}
    \int_0^\infty (\varphi f_n)^2 e^{\bar{\mu}} d\rho & \leq C \int_0^\infty \frac{(\varphi f_n)_\rho^2}{1+Z_{0\rho}^2}
    e^{\bar{\mu}} d\rho\\
    & \leq C \int_0^\infty \Bigl\{ \frac{\varphi_\rho^2 f_n^2}{1+Z_{0\rho}^2} + \frac{\varphi^2 f_{n\rho}^2}{1+Z_{0\rho}^2}
    \Bigr\} e^{\bar{\mu}} d\rho\\
    & \leq C\int_L^{2L} f_n^2 e^{\bar{\mu}} d\rho + C\int_0^\infty \frac{f_{n\rho}^2}{1+Z_{0\rho}^2}e^{\bar{\mu}} d\rho, 
  \end{align*}
  where we have used that $\varphi_\rho$ is supported in $[L, 2L]$.
  Since $f_n$ converges to zero uniformly on $[L,2L]$, the first integral also converges to zero.
  The second integral tends to zero by assumption, and therefore $\lim_{n\to\infty}\int \varphi^2 f_n^2\, d\bar{\mu}\, d\rho = 0 $.
  
  Next, we consider $(1-\varphi)f_n$.
  These functions are supported in $[0, 2L]$.
  On this interval we have
  \[
  c\, \rho^{n-1} \le \frac{e^{\bar{\mu}}}{1+Z_{0\rho}^2} \le e^{\bar{\mu}} \le C\, \rho^{n-1},
  \]
  for suitable constants $c < C$ (these depend on $L$, but here $L$ is fixed).
  This allows us to compare the integrals with the $L^2$ and $H_0^1$ norms on $B_{2L}(0)\subset\R^n$.
  The standard Poincar\'e inequality on $B_{2L}(0)$ implies
  \[
  \int_0^{2L} f^2 \rho^{n-1}d\rho \leq C \int_0^{2L} f_\rho^2 \rho^{n-1} d\rho,
  \]
  for all $f\in C^1([0,2L))$ with $f'(0)=f(2L)=0$.
  Thus we have
  \begin{align*}
    \int(1-\varphi)^2f_n^2\, e^{\bar{\mu}}\, d\rho
    & \leq C \int_0^{2L} (1-\varphi)^2 f_n^2 \rho^{n-1}d\rho                                             \\
    & \leq C \int_0^{2L} \bigl((1-\varphi)f_n\bigr)_\rho^2 \; \rho^{n-1} d\rho                           \\
    & =C \int_0^{2L} \bigl(\varphi_\rho f_n + (1-\varphi)f_{n\rho}\bigr)^2 \rho^{n-1} d\rho              \\
    & \leq C\int_L^{2L} \varphi_\rho^2 f_n^2 \rho^{n-1}d\rho + C \int_0^{2L} f_{n\rho}^2 \rho^{n-1}d\rho. 
  \end{align*}
  Here the first integral tends to zero because $f_n$ converges to zero uniformly on the bounded interval $[L,2L]$, while the second integral can be bounded by
  \[
  \int_0^{2L} f_{n\rho}^2 \rho^{n-1}d\rho \leq C \int_0^{2L} \frac{f_{n\rho}^2}{1+Z_{0\rho}^2} e^{\bar{\mu}} d\rho
  \]
  which also converges to zero as $n\to\infty$.
  Thus we find that $ \int (1-\varphi)^2f_n^2\, e^{\bar{\mu}}\, d\rho\to 0$ as $n\to\infty$.
  Combined with our previous estimate for $\int \varphi^2 f_n^2 e^{\bar{\mu}}\, d\rho$ we get
  \[
  \lim_{n\to\infty} \int f_n^2 e^{\bar{\mu}}\, d\rho \leq \lim_{n\to\infty} \int(\varphi)^2 f_n^2 e^{\bar{\mu}}\, d\rho + \lim_{n\to\infty} \int(1-\varphi)^2 f_n^2 e^{\bar{\mu}}\, d\rho =0.
  \]
  This contradicts the assumption $\int f_n^2 e^{\bar{\mu}}\, d\rho =1$ for all $n$.
  
  \begin{step}
    \label{step-final}
    In this step we combine \eqref{eq-poin-pretip} and \eqref{eq-Poincare-verytip}, using cut off functions, to show \eqref{eq-Poincare}.
    More precisely, there exist uniform constants $C$ and $C(\theta) > 0$, independent of $\tau \le \tau_0$, so that
    \[
    |\tau| \int_0^{\theta} f^2 e^{\mu}\, du \le C\, \int_0^{2\theta} \frac{f_u^2}{1+Y_{\theta,u}^2} \, e^{\mu}\, du + C(\theta)\, \int_{\theta}^{2\theta} f^2 \, e^{\mu}\, du.
    \]
  \end{step}

  Let $\psi_1$ be a cut off function so that $\psi_1 = 1$ for $\frac{L}{\sqrt{|\tau|}} \le u \le \theta$ and $\psi_1 = 0$ outside of $[\frac{L}{\sqrt{2|\tau|}}, 2\theta]$.
  Let $\psi_2$ be a cut off function so that $\psi_2 = 1$ for $0 \le u \le \frac{L}{\sqrt{|\tau|}}$ and $\psi_2 = 0$ for $u \ge \frac{2L}{\sqrt{|\tau|}}$.
  
  By \eqref{eq-poin-pretip} applied to $\psi_1 f$ we have
  \[
  \int_{L/\sqrt{|\tau|}}^{\theta} \frac{f^2}{u}\, \mu_u \, d\sigma \le \int \frac{(\psi_1 f)_u^2}{1 + Y_{u}^2}\, d\sigma.
  \]
  This yields
  \begin{multline*}
    \int_{L/\sqrt{|\tau|}}^{\theta} \frac{f^2}{u}\, \mu_u\, d\sigma \le C\, \int_{L/(2\sqrt{|\tau|)}}^{2\theta} \frac{f_u^2}{1+Y_{u}^2}\, d\sigma + C(\theta)\, \int_{\theta}^{2\theta} f^2\, d\sigma
    \\
    + \frac{C|\tau|}{L^2}\, \int_{L/(2\sqrt{|\tau|)}}^{L/\sqrt{|\tau|}} \frac{f^2}{1+Y_{u}^2}\, d\sigma.
  \end{multline*}
  Combining this with the second estimate in \eqref{eqn-mus} yields
  \begin{multline}
    \label{eq-Poincare-part1}
    |\tau|\, \int_{L/\sqrt{|\tau|}}^{\theta} f^2\, d\sigma
    \le C\int_{L/\sqrt{|\tau|}}^{\theta} \frac{f^2}{u}\, \mu_u\, d\sigma\\
    \le C\, \int_{L/(2\sqrt{|\tau|}}^{2\theta} \frac{f_u^2}{1+Y_{u}^2} + C(\theta)\, \int_{\theta}^{2\theta} f^2\, d\sigma + \frac{C|\tau|}{L^2}\, \int_{L/\big(2\sqrt{|\tau|}\big)}^{L/\sqrt{|\tau|}} f^2\, d\sigma.
  \end{multline}
  We can rewrite the weighted Poincar\'e inequality \eqref{eq-Poincare-verytip}, applied to $\psi_2 f$ as,
  \[
  |\tau|\, \int (\psi_2 f)^2\, e^{\mu}\, du \le C\, \int \frac{(\psi_2 f)_u^2}{1 + Y_{u}^2}\, e^{\mu}\, du
  \]
  where we use again the fact that in the considered tip region we have uniformly smooth convergence of solutions to the Bowl soliton and we can replace $Z_{0\rho}$ by $Y_{u}$.
  This implies
  \begin{equation}
    \label{eq-Poincare-part2}
    \begin{split}
      \int_0^{L/\sqrt{|\tau|}} f^2 |\tau|\, e^{\mu}\, du &\le \int (\psi_2 f)^2 |\tau|\, d\sigma \le
      C\int\frac{(\psi_2 f)_u^2}{1+Y_{u}^2}\, e^{\mu}\, du \\
      &\le C\, \int_0^{2L/\sqrt{|\tau|}} \frac{f_u^2}{1+Y_{u}^2}\, e^{\mu}\, du + C\int \frac{(\psi_2)_u^2f^2}{1+Y_{u}^2}\, e^{\mu}\, du \\
      &+ C\, \int \frac{|\psi_2||(\psi_2)_u| |f||f_u|}{1+Y_{u}^2}\, e^{\mu}\, du \\
      &\le C\, \int_0^{2L/\sqrt{|\tau|}} \frac{f_u^2}{1+Y_{u}^2}\, e^{\mu}\, du + C\int \frac{(\psi_2)_u^2f^2}{1+Y_{u}^2}\, e^{\mu}\, du
    \end{split}
  \end{equation}
  where we applied Cauchy-Schwartz inequality to the last term on the right hand side.
  Add \eqref{eq-Poincare-part1} and \eqref{eq-Poincare-part2} to get
  \begin{equation*}
    \begin{split}
      |\tau| \int_0^{\theta} f^2\, d\sigma &\le C\Bigl( \int_0^{2L/\sqrt{|\tau|}} \frac{f_u^2}{1+Y_{u}^2}\, d\sigma + \int_{L/(2\sqrt{|\tau|})}^{2\theta} \frac{f_u^2}{1+Y_{u}^2}\, d\sigma\Bigr) \\
      &+ C(\theta) \, \int_{\theta}^{2\theta} f^2\, d\sigma + \frac{C|\tau|}{L^2}\int_{L/(2\sqrt{|\tau|})}^{2L/\sqrt{|\tau|}} f^2\, d\sigma \\
    &\le C\int_0^{2\theta} \frac{f_u^2}{1+Y_{u}^2}\, d\sigma + C(\theta) \, \int_{\theta}^{2\theta} f^2\, d\sigma + \frac{C|\tau|}{L^2}\int_{L/(2\sqrt{|\tau|}}^{2L/\sqrt{|\tau|}} f^2\, d\sigma. \end{split}
  \end{equation*}
  We can absorb the last term on the right hand side into the left hand side, for $|\tau|$ large, which finally yields
  \[
  |\tau|\int_0^{\theta} f^2\, d\sigma \le C\int_0^{2\theta} \frac{f_u^2}{1+Y_{u}^2}\, d\sigma + C(\theta)\, \int_{\theta}^{2\theta} f^2\, d\sigma.
  \]
\end{proof}

\subsection{Proof of Proposition \ref{prop-tip}}
We will now conclude the proof of Proposition \ref{prop-tip}.
In order to prove the Proposition, we combine an energy estimate for the difference $W$ which will be shown below, with our  Poincar\'e inequality \eqref{eq-Poincare} (recall that $W$ has been defined at the beginning of section \ref{sec-tip}).
Let $\varphi_T(u)$ be a standard smooth cutoff function supported on $0 < u < 2\theta$, with $\varphi =1$ on $0 \leq u \leq \theta$ and $\varphi = 0$ for $u \ge 2\theta$, and let $W_T := W\, \varphi_T$.

\begin{proof}[Proof of Proposition \ref{prop-tip}]
  
  After multiplying equation \eqref{eqn-WW} by $W \varphi_T^2 \, e^{\mu}$ and integrating by parts we obtain
  \begin{equation}\label{eqn-energy-i1}
    \begin{split}
      & \frac{d}{d\tau} \Bigl ( \frac 12 \int W^2_T  \, e^{\mu}\, du \Bigr  )  = -\int \frac{W_u^2}{1+Y_{u}^2}\,  \, \varphi_T^2 \,e^{\mu}\, du\\
      &+ \int \Bigl(\frac{n-1}{u} - \frac u2 - \frac{\mu_u}{1+Y_{u}^2} + \frac{2Y_{u} (Y)_{uu}}{(1+Y_{u}^2)^2} -  \frac{(Y_2)_{uu}}{1+ Y_{2u}^2}\, \frac{Y_{u} + Y_{2u}}{1 + Y_{u}^2}\Bigr) W_u W\, \varphi_T^2  \, e^{\mu}\, du \\
      &+ 2 \int \frac 1{1+Y_{u}^2}\, W_u W \, \varphi_T\,  (\varphi_T)_u \, e^{\mu}\, du + \int W^2_T \Bigl( \frac12 + \mu_{\tau}\Bigr)\, e^{\mu}\, du.  \\
    \end{split}
  \end{equation}
  Let us write
  \[
  \Bigl(\frac{n-1}{u} - \frac u2 - \frac{\mu_u}{1+Y_{u}^2} + \frac{2Y_{u} (Y)_{uu}}{(1+Y_{u}^2)^2} - \frac{(Y_2)_{uu} (Y_{u} + Y_{2u})}{(1 + Y_{u}^2)\,(1+ Y_{2u}^2)}\Bigr) = \frac {n-1}u \,G
  \]
  where
  \begin{multline}
    \label{eq-G-def}
    G:= \Biggl \{1 - \frac{u^2}{2(n-1)} - \frac{u \mu_u}{n-1} \, \frac 1{1+Y_{u}^2}\\
    + \frac{2 Y_{u} (Y)_{uu}}{(n-1)\, (1 + Y_{u}^2)^2} - \frac{u\,(Y_2)_{uu} (Y_{u} + Y_{2u})}{(n-1)\,(1 + Y_{u}^2)\,(1+ Y_{2u}^2)}\Biggr \}.
  \end{multline}
  Denote by $C$ a uniform constant independent of $\tau$ that can vary from line to line.
  Applying Cauchy-Schwarz to the two terms of \eqref{eqn-energy-i1} involving $W_u W$ we conclude
  \begin{multline}\label{eqn-energy-i2}
    \frac{d}{d\tau} \Bigl ( \frac 12 \int W^2_T \, e^{\mu}\, du \Bigr ) \leq - \frac 12 \int \frac{W_u^2}{1+Y_{u}^2}\,\varphi_T^2 \, e^{\mu}\, du
    +  \int W^2_T \, (\frac 12 + \mu_\tau) \, e^{\mu}\, du \\
    + \int \frac{(n-1)^2}{u^2}\, G^2 W^2_T \, (1+Y_{u}^2)\, e^{\mu}\, du + 4 \int \frac {W^2}{1+Y_{u}^2} \, (\varphi_T)_u^2\, e^{\mu}\, du. 
  \end{multline}
  
  Note that the support of $(\varphi_T)_u$ is contained in the region $\{\theta \le u \le 2\theta\}$.
  By the intermediate region asymptotics in \cite{ADS} we have that in this region we have $c_1(\theta)\, \sqrt{|\tau|} \le |Y_{u}| \le C_1 \sqrt{|\tau|}$, if $\tau \le \tau_0 \ll -1$.
  By this estimate and by \eqref{eqn-mut} we deduce from \eqref{eqn-energy-i2} the differential inequality
  \begin{equation}\label{eqn-energy-i3}
    \begin{split}
      \frac{d}{d\tau} \Bigl ( \frac 12 \int W^2_T \, e^{\mu}\, du \Bigr )
      &\leq - \frac 14 \int \frac{W_u^2}{1+Y_{u}^2}\,\varphi_T^2 \, e^{\mu}\, du + \eta |\tau|\,  \int W^2_T  \, e^{\mu}\, du \\
      & \quad + \int \frac{(n-1)^2}{u^2}\, G^2 W^2_T \, (1+Y_{u}^2)\, e^{\mu}\, du + \frac{C}{|\tau|}\, \int (W\chi_{[\theta,2\theta]})^2 \, e^{\mu}\, du,
    \end{split}
  \end{equation}
  for $\eta$ small (we have also used that $|(\varphi_T)_u| \le C(\theta)$ in $\{\theta \le u \le 2\theta\}$).
  
  We will next estimate the quantity $\frac{(n-1)^2}{u^2}\, G^2$, separately in the regions $L/{\sqrt{|\tau|}} \leq u \leq 2\theta$ and $0 \leq u \leq L/{\sqrt{|\tau|}}$.

  \begin{claim}
    Fix $\eta $ small.
    There exist $\theta, L >0$ depending on $\eta$ and $\tau_0 \ll0$ such that
    \begin{equation}\label{eqn-G5} \frac{(n-1)^2}{u^2}\, G^2 \,
      (1+Y_{u}^2) \, \le \eta |\tau|
    \end{equation}
    on $0 \leq u \leq 2\theta$ and $\tau \leq \tau_0$.
    
  \end{claim}
  
  \begin{proof}
    We begin by establishing the bound for $L/{\sqrt{|\tau|}} \leq u \leq 2\theta$, where $L \gg1$ is large.
    By the \eqref{eq-G-def}, \eqref{eqn-mus}, Remark \ref{cor-sigurd} and the fact that $|Y_{u}|$ is large in the considered region, we have
    \begin{equation}
      \label{eq-G-small1}
      \begin{split}
        & \frac{(n-1)}{u}\, \Bigl| 1 - \frac{u\mu_u}{(n-1)\, (1 + Y_{u}^2)} - \frac{u^2}{2(n-1)}\Bigr| \, \sqrt{1 + Y_{u}^2} \\
        &\le \frac{2(n-1)\, |Y_{u}|}{u}\, (\eta + \frac{u^2}{2})
        \le (1 + \eta) (\eta + 2\theta^2)\, Y \\
        &\le 4\, \eta \sqrt{|\tau|}.
      \end{split}
    \end{equation}
    In the above estimate we also use that $Y$ is close to $\sqrt{2|\tau|}$ in the considered region, and that we can choose $\theta$ small so that $2\theta^2 < \eta$.
    Note that by Remark \ref{cor-sigurd} we have
    \[
    1 - 2\eta \le (1 - \eta)\, \frac{Y}{Y_2} \le \frac{|Y_{u}|}{|Y_{2u}|} \le (1 + \eta)\, \frac{Y}{Y_2} \le 1 + 2\eta
    \]
    since $(1 - \eta)\, \sqrt{2|\tau|} \le Y_i \le (1 + \eta)\, \sqrt{2|\tau|}$, for $i \in \{1,2\}$, in the considered region, provided that $\theta$ is small enough, $L$ is big enough and $\tau \le \tau_0$ is big enough in its absolute value.
    Furthermore, by Proposition \ref{prop-ratio-small} and Remark \ref{cor-sigurd} we have
    \[
    \Bigl|\frac{(Y_2)_{uu}}{1+Y_{2u}^2}\Bigr| \le \eta \frac{|Y_{2u}|}{u} \le \frac{\eta\, (1 + \eta)}{2(n-1)}\, Y_2 \le \frac{4\eta}{n-1}\, \sqrt{|\tau|},
    \]
    hence, implying
    \[
    \Bigl|\frac{(Y_2)_{uu}}{1+ Y_{2u}^2}\, \frac{Y_{u} + Y_{2u}}{1 + Y_{u}^2}\Bigr| \le 4\,\sqrt{2}\eta (1 + \eta)\,\sqrt{|\tau|}\, \frac{|Y_{u}|}{\sqrt{1 + Y_{u}^2}}.
    \]
    Similarly, we get
    \[
    \Bigl|\frac{2\,(Y)_{uu} Y_{u}}{(1+Y_{u}^2}\Bigr| \le 4\, \sqrt{2}\eta (1 + \eta)\,\sqrt{|\tau|}\, \frac{|Y_{u}|}{\sqrt{1 + Y_{u}^2}},
    \]
    and hence,
    \begin{equation}
      \label{eq-G-small2}
      \Bigl|\frac{(Y_2)_{uu}}{1+ Y_{2u}^2}\, \frac{Y_{u} + Y_{2u}}{1 + Y_{u}^2}\Bigr| + \Bigl|\frac{2\,(Y)_{uu} Y_{u}}{(1+Y_{u}^2}\Bigr| < 16\eta\, \sqrt{|\tau|}.
    \end{equation}
    Combining \eqref{eq-G-small1} and \eqref{eq-G-small2}, by the definition of $G$ (see \eqref{eq-G-def}), we get that \eqref{eqn-G5} holds in $\collar_{\theta,L}$, if we take $\eta < \frac{1}{544}$.
    
    For the other region, $0 \leq u \leq L/{\sqrt{|\tau|}}$, which is very near the tip, we will use the fact that our solutions $Y, Y_2$ after rescaling are close to the soliton $Z_0(\rho)$.
    Recall that in this case we have
    \[
    Y_i(u,\tau) = Y(0,\tau) + \frac 1{\sqrt{|\tau|}} \, Z_i(\rho,\tau), \qquad \rho:= u \sqrt{|\tau|}
    \]
    which gives $Y_{u} = Z_{\rho}$.
    Hence,
    \[
    \frac{(n-1)^2}{u^2}\, G^2 \, (1+Y_{u}^2) = \frac{(n-1)^2\, |\tau| }{\rho^2}\, G^2 \, (1+Z_{\rho}^2).
    \]
    Also, $\mu_u = \bar{\mu}_\rho \, \sqrt{|\tau|}$, which gives $u\, \mu_u = \rho \, \bar{\mu}_\rho = \rho\, (m_{\rho}(\rho) + a(L,\tau))$ (we use the definition of weight $\mu(u,\tau)$ given by \eqref{eq-weight}).
    Lets write $G = G_1 + G_2$, where
    \[
    G_1 = 1 - \frac{\rho^2}{2(n-1) |\tau|} - \frac{\rho \, m_\rho }{n-1} \, \frac 1{1+Z_{\rho}^2}
    \]
    and
    \[
    G_2 = -\frac{\rho\, a(L,\tau)}{(n-1)\, (1 + Z_{\rho}^2)} - \frac{\rho (Z_2)_{\rho\rho}\, (Z_{\rho} + Z_{2\rho})}{(1 + Z_{\rho}^2)\,(1 + Z_{2\rho}^2)} + \frac{2\rho\,(Z)_{\rho\rho} Z_{\rho}}{(1 + Z_{\rho}^2)^2}.
    \]
    From the definition of $m(\rho)$ we have
    \[
    1 - \frac{\rho \, m_\rho }{n-1} \, \frac 1{1+ Z_{0 \rho}^2} =0,
    \]
    which implies that
    \[
    G_1 = \frac{\rho \, m_\rho }{n-1} \Bigl ( \frac 1{1+Z_{\rho}^2} - \frac 1{1+ Z_{0\rho}^2} \Bigr ) - \frac{\rho^2}{2(n-1)|\tau|},
    \]
    and after squaring
    \[
    G_1^2 \leq \frac{2\rho^2 \, m_\rho^2 }{(n-1)^2} \Bigl ( \frac 1{1+Z_{\rho}^2} - \frac 1{1+ Z_{0\rho}^2} \Bigr )^2 + \frac{\rho^4}{2(n-1)^2|\tau|^2}.
    \]
    It follows that
    \[
    \frac{(n-1)^2\, |\tau| }{\rho^2}\, G_1^2 \, (1+Z_{\rho}^2) 
    \leq m_\rho^2 \, |\tau| \, (1+Z_{\rho}^2) 
    \Bigl ( \frac 1{1+Z_{\rho}^2} - \frac 1{1+ Z_{0\rho}^2} \Bigr )^2
    + \frac{\rho^2}{2|\tau|} \, (1+Z_{\rho}^2).
    \]
    Using that $Z(\rho,\tau)$ converges uniformly smoothly on compact sets to the soliton $Z_0(\rho)$, we have
    \begin{equation*}
      \begin{split}
        m_{\rho}^2 |\tau|\, (1+Z_{\rho}^2) 
        \Bigl ( \frac 1{1+Z_{\rho}^2} - \frac 1{1+ Z_{0\rho}^2} \Bigr )^2 
        &=\frac{(n-1)^2}{\rho^2\, (1 + Z_{\rho}^2)}\, 
        (Z_{0\rho}^2 - Z_{\rho}^2)^2 \, |\tau|\\
        &\leq \,\frac{(n-1)^2 \,(Z_{0\rho} + Z_{\rho})^2}{\rho^2} (Z_{0\rho} - Z_{\rho})^2\, |\tau| \\
        &\le C\, (Z_{0\rho} - Z_{\rho})^2\, |\tau| < \frac{\eta}{3}\, |\tau|,
      \end{split}
    \end{equation*}
    where we also used that $\frac{(n-1)^2 \,(Z_{0\rho} + Z_{\rho})^2}{\rho^2}$ is uniformly bounded for all $\tau \le \tau_0 \ll -1$ and all $\rho \in [0,L]$.
    This follows from the fact that $Z(\rho,\tau)$ uniformly smoothly converges to $Z_0(\rho)$ as $\tau\to-\infty$, for $\rho \in [0,L]$, the asymptotics of $Z_0(\rho)$ around the origin and infinity and the fact that $(Z_0)_{\rho}(0) = Z_{\rho}(0,\tau) = 0$, for all $\tau$.
    Furthermore,
    \[
    \frac{\rho^2}{2|\tau|}\, (1 + Z_{\rho}^2) \le \frac{C_1 L^2 + C_2}{2|\tau|} < \frac{\eta}{3},
    \]
    which can be achieved by taking $|\tau| \ge |\tau_0| \gg 1$ very large (relative to a fixed constant $L$).
    Finally, the fact that $|m_\rho| \leq C(L) \, \rho^{-1}$, for $\rho\in [0,L]$ and the fact that $Z(\rho,\tau)$ converges uniformly smoothly, as $\tau\to -\infty$, to the soliton $Z_0(\rho)$, implies that ${\displaystyle m_\rho^2 \, (1+Z_{\rho}^2) \Bigl ( \frac 1{1+Z_{\rho}^2} - \frac 1{1+ Z_{0\rho}^2} \Bigr )^2 }$ can be made arbitrarily small if $\tau \leq \tau_0(L)\ll-1$.
    Above estimates guarantee that
    \begin{equation}
      \label{eq-G1-est}
      \frac{(n-1)^2}{\rho^2} |\tau|\, G_1^2 (1 + Z_{\rho}^2) \le \frac{\eta}{4}\, |\tau|,
    \end{equation}
    for $\tau \le \tau_0 \ll -1$.
    
    On the other hand, by Lemma \ref{lemma-prop-mu} and the fact that both $Z(\rho,\tau)$ and $Z_2(\rho,\tau)$ converge uniformly smoothly on $\rho\in [0,L]$, as $\tau\to-\infty$, to the soliton $Z_0(\rho)$ we have that
    \begin{equation}
      \label{eq-G2-est}
      \frac{(n-1)^2}{\rho^2}\, |\tau| G_2^2 (1 + Z_{\rho}^2) \le \frac{\eta}{4}\, |\tau|.
    \end{equation}
    Combining \eqref{eq-G1-est} and \eqref{eq-G2-est} yields \eqref{eqn-G5}.
  \end{proof}
  
  We now continue the proof of Proposition \ref{prop-tip}.
  Lets insert the bound \eqref{eqn-G5} in the differential inequality \eqref{eqn-energy-i3}.
  Using also the bound $| (\varphi_T)_u \chi_{[\theta,2\theta]}| \leq C(\theta) $ we obtain
  \begin{multline*}
    \frac{d}{d\tau} \Bigl ( \frac 12 \int W^2_T \, e^{\mu}\, du \Bigr )
    \leq - \frac 14 \int \frac{(W_T)_u^2}{1+Y_{u}^2}\,e^{\mu}\, du \\
    + (2 \eta |\tau| +C) \int W^2_T \, e^{\mu}\, du + C(\theta) \, \int_{\theta}^{2\theta} W^2\, e^{\mu}\, du.
  \end{multline*}
  On the other hand, our Poincar\'e inequality says that
  \[
  \int \frac{(W_T)_u^2}{1+Y_{u}^2}\,e^{\mu}\, du + \int_{\theta}^{2\theta} W_T^2\, e^{\mu}\, du \geq c_0 \, |\tau | \int W^2_T \, e^{\mu}\, du,
  \]
  with $c_0 >0$ a constant which is uniform in $\tau$ and independent of $\theta$.
  Hence,
  \begin{equation*}
    \begin{split}
      - \frac 14 \int \frac{(W_T)_u^2}{1+Y_{u}^2}\,e^{\mu}\, du  &+  (2 \eta |\tau| +C)   \int W^2_T  \, e^{\mu}\, du \leq \\
      &- \frac {c_0}4\, |\tau|  \int W^2_T  \,e^{\mu}\, du  +  (2 \eta^2 |\tau| +C)  \int W^2_T  \, e^{\mu}\, du \\
      &\leq - \frac {c_0} 8 \int |\tau| \, W^2_T \,e^{\mu}\, du
    \end{split}
  \end{equation*}
  if $\tau \leq \tau_0$, with $\tau_0$ depending on $\eta$, $c_0$ and  $C$.
  We conclude that in the tip region $\tip_{\theta}$ the following holds
  \begin{equation}\label{eqn-diff-ineq-tip1}
    \frac{d}{d\tau} \int W^2_T \, e^{\mu}\, du \leq - \frac {c_0}8\, |\tau| \int \, W^2_T \, e^{\mu}\, du + \frac{C(\theta)}{|\tau|} \, \int (W\chi_{[\theta,2\theta]})^2\, e^{\mu}\, du.
  \end{equation}
  Define
  \[
  f(\tau) := \int W_T^2\, e^{\mu}\, du, \qquad g(\tau) := \int (W\chi_{[\theta,2\theta]})^2\, e^{\mu}\, du .
  \]
  Then equation \eqref{eqn-diff-ineq-tip1} becomes
  \[
  \frac{d}{d\tau} f(\tau) \le - \frac{c_0}{8} \, |\tau|\, f(\tau) + \frac{C(\theta)}{|\tau|}\,g(\tau).
  \]
  Furthermore, setting ${\displaystyle F(\tau) := \int_{\tau-1}^{\tau} f(s)\, ds}$ and ${\displaystyle G(\tau) := \int_{\tau-1}^{\tau} g(s)\, ds}$, we have
  \begin{equation*}
    \begin{split}
      \frac{d}{d\tau} F(\tau) = f(\tau) - f(\tau-1) &= \int_{\tau-1}^{\tau} \frac{d}{ds} f(s)\, ds \\
      &\le \frac{c_0}{8}\,\int_{\tau-1}^{\tau} s f(s)\, ds + \int_{\tau-1}^{\tau} \frac{C(\theta)}{|s|} g(s)\, ds
    \end{split}
  \end{equation*}
  implying
  \[
  \frac{d}{d\tau} F(\tau) \le \frac{c_0}{16} \, \tau \, F(\tau) + \frac{C(\theta)}{|\tau|}G(\tau).
  \]
  This is equivalent to
  \[
  \frac{d}{d\tau} \bigl( e^{-c_0 \tau ^2/32} F(\tau) \bigr) \le \frac{C(\theta)}{|\tau|} e^{-c_0 \tau^2/32}\, G(\tau).
  \]
  Since $W_T$ is uniformly bounded for $\tau \leq \tau_0 \ll -1$, it follows that $f(\tau)$ and therefore also $F(\tau)$ are uniformly bounded functions for $\tau \leq \tau_0$.
  Hence, ${\displaystyle \lim_{\tau\to-\infty} e^{-c_0 \tau^2/32} F(\tau) = 0}$, so that from the last differential inequality we get
  \begin{equation*}
    \begin{split}
      e^{-c_0 |\tau|^2/32}\, F(\tau)
      &\le C\, \int_{-\infty}^{\tau} \frac{G(s)}{s^2} (|s|\, e^{-c_0 s^2/32} )\, ds \\
      &\le \frac{C}{|\tau|^2} \, \sup_{s\le\tau} G(s) \, \int_{-\infty}^{\tau} |s|\,  e^{-c_0 s^2/32} \, ds\\
      &\le \frac{C}{|\tau|^2} \sup_{s\le\tau} G(s) \, e^{-c_0 \tau^2/32}
    \end{split}
  \end{equation*}
  with $C=C(\theta,\delta)$.
  This yields
  \[
  \sup_{s\le \tau} F(s) \le \frac{C}{|\tau|^2}\, \sup_{s\le\tau} G(s),
  \]
  or equivalently,
  \[
  \|W_T\|_{2,\infty} \le \frac{C(\theta)}{|\tau_0|} \|W \chi_{[\theta,2\theta]}\|_{2,\infty},
  \]
  therefore concluding  the proof of Proposition \ref{prop-tip}.
\end{proof}


\section{Proofs of Theorems  \ref{thm-main-main} and \ref{thm-main}}\label{sec-conclusion}

We will now combine  Propositions \ref{prop-cylindrical} and \ref{prop-tip} to conclude the proof of our main result 
Theorem \ref{thm-main}.
Our most general result, Theorem \ref{thm-main-main} will then readily follow by combining  Theorem \ref{thm-rot-symm} and Theorem \ref{thm-main}. 

\smallskip
In fact as  have seen at the beginning of Section \ref{sec-regions} that 
Translating and dilating the original solution has an effect on the rescaled solution rotationally symmetric solution $u(y,\tau)$ as given in formula \eqref{eq-ualphabeta}.
Thus for any two   rotationally symmetric solutions  $u_1(y,\tau), u_2(y,\tau)$ 
Let $u_1(y,\tau)$ and $u_2(y,\tau)$ be any  two solutions to equation \eqref{eq-u} as in the statement of Theorem \ref{thm-main} and let $u_2^{\alpha\beta\gamma}$ be defined by \eqref{eq-ualphabeta}.
Our goal is to find parameters $(\alpha,\beta,\gamma)$  so that the difference 
\[
w^{\alpha\beta\gamma}:= u_1-u_2^{\alpha\beta\gamma}\equiv 0.
\]

Proposition \ref{prop-tip} says that the weighted $L^2$-norm $\|W^{\alpha\beta\gamma} \|_{2,\infty}$ of the difference of our solutions $W^{\alpha\beta\gamma}(u,\tau):=Y_1(u,\tau) - Y_2^{\alpha\beta\gamma}(u,\tau)$ (after we switch the variables $y$ and $u$) in the whole tip region $\tip_\theta$ is controlled by   $\| W^{\alpha\beta\gamma} \, \chi_{[\theta,2\theta]}\|_{2,\infty}$,  where  $\chi_{[\theta,2\theta]}(u)$ is supported in the transition region 
between the cylindrical and tip regions and is   included in the cylindrical region $ \cyl_{\theta} = \{(y, \tau) : u_1(y,\tau) \ge \theta/2\}$.
Lemma \ref{prop-norm-equiv} below says  that the norms $\| W^{\alpha\beta\gamma} \, \chi_{D_{2\theta}}\|_{2,\infty}$ and $\| w^{\alpha\beta\gamma} \, \chi_{D_{2\theta}}\|_{\hilb,\infty}$ are equivalent for every number $\theta >0$  sufficiently small (recall the definition of $\| \cdot \|_{\hilb,\infty}$ in \eqref{eqn-normp0}-\eqref{eqn-normp}). 
Therefore combining  Propositions \ref{prop-cylindrical} and \ref{prop-tip} gives the crucial estimate \eqref{eqn-w1230} which will be shown in detail in Proposition \ref{prop-cor-main} below. 
This estimate says that the norm of the difference $w^{\alpha\beta\gamma}_{\cC}$ of our solutions when restricted in the cylindrical region  is dominated by the norm of its projection of $w^{\alpha\beta\gamma}_{\cC}$  onto the zero eigenspace of the operator $\cL$ (the linearization of our equation on  the limiting cylinder).  
However,  Proposition \ref{prop-cylindrical} holds under the assumption that the projection of $w_\cyl^{\alpha\beta\gamma}$ onto the positive eigenspace  of $\cL$ is zero, that  is  $\pr_+ w_\cyl(\tau_0)^{\alpha\beta\gamma} =0$.  
Recall  that the  zero eigenspace of $\cL$ is spanned by the function $\psi_2(y) = y^2 - 2$ and the positive eigenspace is spanned by the eigenvectors  $\psi_0(y) = 1$ (corresponding to eigenvalue $1$) and  $\psi_1(y) = y$ (corresponding to eigenvalue $1/2$). 

\smallskip  
After having established that the projection onto the zero eigenspace $a(\tau):= \langle w_\cyl^{\alpha\beta\gamma}, \psi_2 \rangle$ 
dominates in the $\| w^{\alpha\beta\gamma}_{\cC} \|_{\hilb,\infty}$, 
the conclusion of Theorem \ref{thm-main} will follow  by establishing an appropriate   differential inequality for $a(\tau)$, for $\tau \leq \tau_0\ll -1$ and also having that
$a(\tau_0) = \mathcal{P}_0 w_C^{\alpha\beta\gamma}(\tau_0) = 0$ at the same time.
The above discussion shows that it is essential for our proof to have 
\begin{equation}\label{eqn-abc}
  \pr_+ w_\cyl^{\alpha\beta\gamma}(\tau_0) = \pr_0 w_\cyl^{\alpha\beta\gamma}(\tau_0) = 0.
\end{equation}
We will next show that for every $\tau_0 \ll -1$ we  can  find  parameters $\alpha=\alpha(\tau_0), \beta=\beta(\tau_0)$ and  $\gamma=\gamma(\tau_0)$ such that \eqref{eqn-abc} holds 
and we will also   give their asymptotics relative to $\tau_0$.
Let us emphasize that we need to be able {\em  for every}  $\tau_0 \ll -1$ to find parameters $\alpha, \beta, \gamma$ so that  \eqref{eqn-abc} holds,
since up to  the final step of our proof  we have to keep adjusting    $\tau_0$ by taking it even more negative  so that our estimates hold (see Remark \ref{rem-choice-par} below). 

\smallskip 

For $v_i$ related to $u_i$ by $u_i = \sqrt{2(n-1)}(1+v_i)$, the corresponding dilations by $(\alpha,\beta,\gamma)$ are given by 
\[
v_i^{\alpha\beta\gamma}(y, \tau) = \sqrt{1+\beta\,  e^{\tau}} \Bigl\{ 1+ v_i\Bigl( \frac{y - \alpha \,  e^{\tau/2}} {\sqrt{1+\beta e^\tau} }, \tau+\gamma-\log\bigl(1+\beta e^\tau\bigr) \Bigr) \Bigr\} -1.
\]
Simply write $v$ for $v_1$ and $\bv$ for $v_2^{\alpha\beta\gamma}$.

Our asymptotics in Theorem \ref{thm-old} imply that each $v_i$ satisfies the following estimates in the cylindrical region $\cyl_{\theta}$:  for any $\epsilon_0>0$ and any number $M>0$ there is a $\tau_{\epsilon_0,M} < 0$ such that
\begin{equation}\label{eqn-good3}
  v_i(y, \tau) = - \frac{y^2-2} {4|\tau|} + \frac{\epsilon(y, \tau)} {|\tau|}, \qquad \mbox{for} \,\,\, 0\leq y \leq 2M,\, \tau\leq \tau_{\epsilon_0,M}
\end{equation}
where $\epsilon(y, \tau)$ is a generic function whose definition may change from line to line, but which always satisfies
\begin{equation}
  |\epsilon(y, \tau)| \leq \epsilon_0,  \qquad \mbox{for} \,\,\, 0\leq y \leq 2M,\, \tau\leq \tau_{\epsilon_0,M}.
\end{equation}
Furthermore, by choosing $\tau_{\epsilon_0,M} \ll -1$ we also have
\begin{equation}
  0\leq -v_i(y, \tau) \leq C \, \frac{y^2}{|\tau|} \qquad \text{in } \,\,\,\, \cyl_{\theta} \cap \{ |y| \geq M \}, \, \tau\leq \tau_{\epsilon_0,M}.
  \label{eq-v-quadratic-upper-bound}
\end{equation}

We will next estimate the first three components of the truncated difference $\varphi_\cyl(\bv-v)$,
\[
\Bigl\langle \psi_j , \, \varphi_\cyl\, ( \bv - v ) \Bigr\rangle \qquad (j=0, 1, 2)
\]
where $\varphi_\cyl$ is the cut-off function for the cylindrical region $ \cyl_{\theta}$ and we will show that the coefficients $\alpha, \beta$ and $\gamma$ can be chosen so as to make these components vanish.
Instead of working directly with $\alpha, \beta$ and $\gamma$ it will be more convenient to use
\begin{equation}
  \label{btheta}
  b= \sqrt{1+\beta e^\tau} -1, \qquad \Gamma = \frac{\gamma - \log(1+\beta e^\tau)}{\tau}, \qquad A = \alpha\,  e^{\tau/2}.
\end{equation}
Then
\begin{equation}
  \bv(y, \tau) = b + (1+b)\, v_2\Bigl( \frac{y-A} {1+b} , (1+\Gamma)\tau \Bigr)
\end{equation}

Our next goal is to show the following result. 

\begin{prop}\label{lem-rescaling-components-zero}
  There is a number $ \tau_* \ll -1   $ such that for all $\tau\leq \tau_*$ there exist $ b $, $ \Gamma $ and $A$ such that the difference $w^{\alpha\beta\gamma} :=u_1 - u_2^{\alpha\beta\gamma}$ satisfies
  \[
  \langle \psi_0, \varphi_\cyl \, w^{\alpha\beta\gamma} \rangle = \langle \psi_1, \varphi_\cyl \, w^{\alpha\beta\gamma} \rangle = \langle \psi_2, \varphi_\cyl \,w^{\alpha\beta\gamma} \rangle = 0.
  \]
  In addition, the parameters $\alpha, \beta$ and $\gamma$ can be chosen so that $b$, $\Gamma$ and $A$ defined in \eqref{btheta} satisfy
  \begin{equation}\label{eq-b-Theta-bound}
    b = o\Bigl( |\tau|^{-1} \Bigr), \qquad \Gamma = o(1)\qquad \mbox{and} \qquad |A| = o(1), \qquad \mbox{as}\,\, \tau \to -\infty.
  \end{equation}
  Equivalently, this means that the triple $(\alpha, \beta, \gamma)$ is admissible with respect to $\tau$, according to our Definition \ref{def-admissible}.
\end{prop}

The proof of the proposition will be based on the following estimate. 

\begin{lemma} \label{lem-rescaling-effect-on-components} For every $ \eta>0 $ there exist $ \tau_\eta <0 $ such that for all $ \tau\leq \tau_\eta $, and all $ b , \Gamma, A\in \R $ with
  \[
  |b| \leq \frac{1}{|\tau|}, \qquad |\Gamma|\leq \frac12, \qquad |A| \le 1
  \]
  one has
  \begin{equation}\label{eq-rescaling-effect-on-components}
    \begin{split}
      \Bigl| \langle\hat\psi_0, \varphi_\cyl(\bv-v)\rangle - b + \frac{A^2}{4(\Gamma+1)|\tau|}\Bigr|  &+  \Bigl| \langle\hat\psi_1, \varphi_\cyl(\bv-v)\rangle - \frac{A}{2|\tau|(\Gamma+1)} \Bigr| \\&+ \Bigl| \langle\hat\psi_2, \varphi_\cyl(\bv-v)\rangle - \frac{\Gamma} {4(\Gamma+1)|\tau|} \Bigr| \leq \frac{\eta}{|\tau|}
    \end{split}
  \end{equation}
  where $ \hat\psi_j = \psi_j/\langle \psi_j, \psi_j \rangle $.
\end{lemma}
The conditions on $b$, $\Gamma$ and $A$ are met if the original parameters $\alpha, \beta$ and $\gamma$ satisfy
\[
|\alpha\,  e^{\tau/2}| \le 1, \qquad   |\beta e^\tau|\leq \frac{C}{|\tau|}, \qquad |\gamma|\leq \tfrac{1}{3} |\tau|.
\]

\begin{proof}[Proof that Lemma~\ref{lem-rescaling-effect-on-components} implies Proposition~\ref{lem-rescaling-components-zero}]
  Let $\eta_*>0$ be given, and consider the disc
  \[
  \cB = \Bigl\{ (b, \Gamma, A) \mid |\tau|^2b^2 + \Gamma^2 + A^2 \leq \eta_*^2\Bigr\}.
  \]
  On this ball we define the map $ \Phi : \cB \to\R^3 $ given by
  \[
  \Phi(b, \Gamma) =
  \begin{pmatrix}
    |\tau|\langle \hat\psi_0, \varphi_\cyl(\bv-v) \rangle\\
    |\tau|\langle \hat\psi_1, \varphi_\cyl(\bv-v) \rangle\\
    |\tau| \langle \hat\psi_2, \varphi_\cyl(\bv-v) \rangle
  \end{pmatrix}.
  \]
  The map $ \Phi $ is continuous because the solution $ \bv $ depends continuously on the parameters $b, \Gamma , A$.
  
  It follows from \eqref{eq-rescaling-effect-on-components} that if $\eta \ll\eta_*$ is chosen small enough, and if $ \tau $ is restricted to $ \tau < \tau_\eta $, with $ \tau_\eta $ defined as in Lemma~\ref{lem-rescaling-effect-on-components}, then the map $ \Phi $ restricted to the boundary of the ball $ \cB $ is homotopic to the injective map 
  \[(b, \Gamma,A) \mapsto \left(|\tau|b - \frac{A^2}{4(\Gamma+1)}, \frac{A}{2(\Gamma+1)},\frac{\Gamma}{4(\Gamma+1)}\right),\]
  through maps from $ \partial\cB $ to $ \R^3\setminus\{0\} $.
  The map $ \Phi $ from the full ball to $\R^3$ therefore has degree one, and it follows that for some $ (b',\Gamma',A')\in\cB $ one has $ \Phi(b', \Gamma',A') = 0 $.
  The fact that $(b',\Gamma',A')$ that we have just found do indeed satisfy~\eqref{eq-b-Theta-bound} follows from the definition of the disc $ \cB $.
\end{proof}

\begin{proof}[Proof of Lemma~\ref{lem-rescaling-effect-on-components}]
  
  First we consider the outer region $\cyl_{\theta} \cap \{ |y| \geq M \}$ where \eqref{eq-v-quadratic-upper-bound} implies that
  \begin{equation}\label{eq-v-bv-difference}
    \big |   \varphi_\cyl \, (\bv  - v )  \big | \leq |b| + \tilde C \frac{y^2}{|\tau|} \leq \frac{1+\tilde Cy^2} {|\tau|},  \quad\text{for all } \cyl_{\theta} \cap \{ |y| \geq M \}.
  \end{equation}
  
  For the inner region $\cyl_{\theta} \cap \{ |y| \leq M \}$, \eqref{eqn-good3} implies that for all $\tau\leq 2\, \tau_{\epsilon_0,M} \ll -1$ one has
  \begin{align*}
    \bv(y,\tau) &= b - \frac{1}{(1+b)} \frac{(y-A)^2-2(1+b)^2}{4(1+\Gamma)|\tau|}
    + \frac{2\epsilon(y,\tau)}{(1+\Gamma)|\tau|}\\
    &= b + \frac{2b+b^2}{(1+b)} \frac{1}{2(1+\Gamma)|\tau|} - \frac{1}{(1+b)} \frac{y^2-2}{4(1+\Gamma)|\tau|} \\
    &\qquad +  \frac{2A y - A^2}{4|\tau|(\Gamma+1)(b+1)} + \frac{2\epsilon(y,\tau)}{(1+\Gamma)|\tau|}.
  \end{align*}
  If $v(y, \tau)$ is the other unrescaled  solution, then we have
  \begin{align*}
    &  \bv(y, \tau) - v(y,\tau)\\
    &= b + \frac{2b+b^2}{2(1+\Gamma)(1+b)} \frac{1}{|\tau|} + \Bigl \{ 1 - \frac{1}{(1+b)} \frac{1}{(1+\Gamma)} \Bigr \}\frac{y^2-2}{4|\tau|} + \frac{2Ay - A^2}{4|\tau|(\Gamma+1)(b+1)}
    + \frac{5\epsilon(y,\tau)}{|\tau|}\\
    &= b + \frac{2b+b^2}{2(1+\Gamma)(1+b)} \frac{1}{|\tau|} + \frac{b+\Gamma+b\Gamma}{(1+\Gamma)(1+b)} \frac{y^2-2}{4|\tau|} + \frac{2Ay - A^2}{4|\tau|(\Gamma+1)(b+1)} + \frac{5\epsilon(y,\tau)}{|\tau|}.
  \end{align*}
  We conclude that in the region $ |y|\leq M$ and for $ \tau\leq 2 \tau_{\epsilon_0,M}$ we have 
  \begin{equation*} 
    \bv(y, \tau) - v(y,\tau) = b - \frac{A^2}{4(\Gamma+1)|\tau|} + \frac{\Gamma}{\Gamma+1} \frac{(y^2-2)}{4|\tau|}  + \frac{Ay}{2|\tau|(\Gamma+1)} 
    + R(y,\tau)
  \end{equation*}
  where the remainder $R$ satisfies
  \begin{equation}\label{eq-R-in-cylindrical-bound}
  |R(y,\tau)| \leq C \, \Bigl( \frac{1+y^2} {|\tau|^2} + \frac{\epsilon(y,\tau)}{|\tau|} \Bigr), \quad \mbox{on} \,\, \cyl_\theta \cap \{ |y| \leq M \}\end{equation} with $ \epsilon(y,\tau) \leq \epsilon_0$ and $C$ is a universal constant that does not depend on $\tau$ or $M$.
  
  Combining  the last bound  which holds  on $|y| \leq M$ with our first  bound  \eqref{eq-v-bv-difference}  on  $|y|\leq M$
  yields  that there exists $\tau_* \ll -1$ such that for all $y\in \mathbb{R}$ and $ \tau\leq \tau_*$, we have 
  \begin{equation}
    \label{eq-diff-expression}
    \varphi_\cyl\, \big (\bv - v\big ) = b - \frac{A^2}{4(\Gamma+1)|\tau|}   + \frac{A\, y}{2|\tau|(\Gamma+1)} + \frac{\Gamma}{\Gamma+1} \, \frac{y^2-2}{4|\tau|} 
    + R(y,\tau)
  \end{equation}
  where the new error $R$ still satisfies \eqref{eq-R-in-cylindrical-bound} when $|y|\leq M$, and
  \begin{equation}
    \label{eq-R-estimate}
    |R(y, \tau)|\leq C \, \dfrac{1+y^2} {|\tau|}, \quad \mbox{on} \,\, \cyl_\theta \cap \{ |y| \geq M \}
  \end{equation}
  for a universal constant $C$.
  
  \subsubsection*{Components of the error}
  We now estimate the inner products $\langle \psi_j, \varphi_\cyl(\bv-v)\rangle$
  \begin{align*}
    \bigl\langle \psi_j, \varphi_\cyl(\bv-v)\bigr\rangle &= \langle\psi_j, 1\rangle \Bigl (b - \frac{A^2}{4(\Gamma+1)|\tau|}\Bigr ) + \langle \psi_j,  y\rangle \, \frac{A}{2(\Gamma+1)|\tau|}\\
    &\quad  + \langle\psi_j, y^2-2\rangle \frac{\Gamma} {4(\Gamma+1)|\tau|} + \langle\psi_j, R\rangle.
  \end{align*}
  In view of the fact that $\psi_0=1$, $\psi_1 = y$ and $\psi_2=y^2-2$, we have 
  $$
  \frac{\langle\psi_0, \varphi_\cyl(\bv-v)\rangle} {\langle\psi_0,\psi_0\rangle}
  = b - \frac{A^2}{4(\Gamma+1)|\tau|}+ \frac{\langle\psi_0, R\rangle} {\langle\psi_0,\psi_0\rangle} 
  $$
  and
  
  $$\frac{\langle\psi_1, \varphi_\cyl(\bv-v)\rangle} {\langle\psi_1,\psi_1\rangle}
  =  \frac{A}{2(\Gamma+1)|\tau|} + \frac{\langle\psi_1, R\rangle} {\langle\psi_1,\psi_1\rangle}, \quad 
  \frac{\langle\psi_2, \varphi_\cyl(\bv-v)\rangle} {\langle\psi_2,\psi_2\rangle} = \frac{\Gamma} {4(\Gamma+1)|\tau|} + \frac{\langle\psi_2, R\rangle} {\langle\psi_2,\psi_2\rangle}.$$
  %
\smallskip

We claim that  \eqref{eq-R-in-cylindrical-bound}-\eqref{eq-R-estimate} imply that for every $ \eta>0 $ there exist $ \tau_\eta < 0 $ and $ M_\eta >0 $ such that for all $ \tau \leq \tau_\eta $ and $ M\geq M_\eta $ one has
\begin{equation}\label{eqn-finalR}
  \Bigl|\langle \psi_j, R\rangle\Bigr| \leq \frac{\eta} {|\tau|}.
\end{equation}
Indeed, to prove the above claim we notice that all three inner products can be bounded by the integral
\[
\int_0^\infty (1+ y + y^2) \, |R(y, \tau)|\, e^{-y^2/4} \, dy.
\]
Split the integral into three parts, the first from $y=0$ to $y=M$:
\begin{align*}
  \int_0^M |R|\, (1+y+y^2) \, e^{-y^2/4}\, dy
  &\leq \frac{C}{|\tau|} \int_0^M (1+y+y^2) \Bigl\{ \frac{1}{|\tau|}(1+y^2) + \epsilon \Bigr\} e^{-y^2/4} dy\\
  &\leq \frac{C}{|\tau|^2} + \frac{C\epsilon}{|\tau|}.
\end{align*}
If we choose $|\tau_{\epsilon,M}|$ so that $-\tau_{\epsilon,M} < 1/\epsilon $ then
\[
\int_0^M |R| \, (1+y+y^2)\, e^{-y^2/4} \, dy \leq \frac{C\epsilon}{|\tau|} \leq \frac{\eta}{ 2|\tau|}
\]
by taking $\epsilon$ small enough.
For the remaining part, using \eqref{eq-R-estimate} we obtain
\begin{align*}
  \int_{|y|\geq M} |R| \, (1+y + y^2)\, e^{-y^2/4}\, dy
  &\leq C\int_{M}^\infty  (1+y^2)\,  \frac{1+y^2}{|\tau|} e^{-y^2/4} \, dy \\
  &\leq \frac{C}{|\tau|} \int_M^\infty (1+y^2)^2 e^{-y^2/4} \, dy\\
  &\leq \frac{C}{|\tau|} M^3e^{-M^2/4} \leq \frac{\eta}{2 |\tau|}
\end{align*}
by chooing $ M $ sufficiently large.
The lemma now   readily follows from \eqref{eq-diff-expression}, \eqref{eq-v-bv-difference} and \eqref{eqn-finalR}.  
\end{proof}

\begin{remark}[The choice of parameters $(\alpha,\beta, \gamma)$]\label{rem-choice-par}
  We can choose $\tau_0 \ll -1$ to be any small number so that $\tau_0 \le \tau_*$, where $\tau_*$ is as in Proposition \ref{lem-rescaling-components-zero} and so that all our uniform estimates in previous sections hold for $\tau \le \tau_0$.
  Note also that having Proposition \ref{lem-rescaling-components-zero} we can decrease $\tau_0$ if necessary and choose parameters $\alpha, \beta$ and $\gamma$ again so that we still have $\pr_+ w_\cyl(\tau_0) = \pr_0 w_\cyl(\tau_0) = 0$, without effecting our estimates.
  Hence, from now on we will be assuming that we have fixed parameters $\alpha, \beta$ and $\gamma$ at some time $\tau_0 \ll -1$, to have both projections zero at time $\tau_0$.
  As a consequence of Proposition \ref{lem-rescaling-components-zero} which shows  that the parameters $(\alpha, \beta, \gamma)$ are {\em admissible}  with respect to $\tau_0$, Remark \ref{rem-cylindrical} and Remark \ref{rem-tip}, all the estimates for $w = u_1- u_2^{\alpha\beta\gamma}$ will  then hold for all $\tau \le \tau_0$, {\em independently of our choice}  of 
  $(\alpha, \beta, \gamma)$.
\end{remark}

As we pointed out above,  we need to  show next  that the norms of the difference of our two solutions with respect to the weights defined in the cylindrical and the tip regions are equivalent in the intersection between the regions, the so called {\em transition} region. 

\begin{lemma}[Equivalence of the norms in the transition region]
  \label{prop-norm-equiv}
  Let $w, W$ denote the difference of the two solutions, $w:=u_1-u_2^{\alpha\beta\gamma}$ and $W:=Y_1-Y_2^{\alpha\beta\gamma}$ in the cylindrical and tip regions respectively.
  Then, for every $\theta > 0$ small there exist $\tau_0 \ll -1$ and uniform constants $c(\theta), C(\theta) > 0$, so that for $\tau\le \tau_0$, we have
  \begin{equation}\label{eqn-normequiv}
    c(\theta ) \, \| W \chi_{_{[\theta, 2\theta]}} \|_{2,\infty} \leq \| w \, \chi_{_{D_{2\theta}}} \|_{\hilb,\infty} \leq C(\theta) \, \| W \chi_{_{[\theta, 2\theta]}} \|_{2,\infty}
  \end{equation}
  where $D_{2\theta} := \{(y,\tau):\,\,  \theta \le u_1(y,\tau) \le 2\theta\}$.
\end{lemma}

\begin{proof}
  To simplify the notation we  put $u_2:= u_2^{\alpha\beta\gamma}$ and $Y_2:=Y_2^{\alpha\beta\gamma}$ in this proof.
  Define $A_{2\theta} := D_{2\theta}\cup\{ (y,\tau):\,\, \theta \le u_2(y,\tau) \le 2\theta\}$.
  The convexity of both our solutions $u_1$ and $u_2$ imply that
  \begin{equation}
    \label{eq-secant}
    \min_{A_{2\theta}}\,|(u_2)_y| \le \Bigl| \frac{u_1(y,\tau) - u_2(y,\tau)} {Y_1(u,\tau) - Y_2(u,\tau)}\Bigr| \le \max_{A_{2\theta}} |(u_2)_y|.
  \end{equation}
  This easily follows from
  \[
  \frac{|u_1(y,\tau) - u_2(y,\tau)|}{|Y_1(u,\tau) - Y_2(u,\tau)|} = \frac{|u_2(Y_1(u,\tau),\tau) - u_2(Y_2(u,\tau)|}{|Y_1(u,\tau) - Y_2(u,\tau)|} = |u_{2y}(\xi,\tau)|
  \]
  where $\xi$ is a point in between $Y_1(u,\tau)$ and $Y_2(u,\tau)$.
  
  The results in \cite{ADS} (see also Theorem \ref{thm-old} in the current paper) show that by the asymptotics in the intermediate region for $u_2$, we have
  \begin{equation}
    \label{eq-u2y}
    \frac{c_1(\theta)}{\sqrt{|\tau|}} \le |u_{2y}(y,\tau)| \le \frac{C_1(\theta)}{\sqrt{|\tau|}}, \qquad \mbox{for} \,\,\,\, y\in \{y \,\,\,|\,\,\, \theta \le u_2(y,\tau) \le 2\theta\}
  \end{equation}
  for uniform constants $c_1(\theta) > 0$ and $C_1(\theta) > 0$, independent of $\tau$ for $\tau \le \tau_0$.
  On the other hand, using that $u_2$ has the same asymptotics in the intermediate region as $u_1$, it is easy to see that for $\tau \le \tau_0 \ll -1$,
  \[
  D_{2\theta} \subset \{ (y,\tau):\,\,  \frac{\theta}{2} \le u_2(y,\tau) \le 3\theta\}
  \]
  and hence
  \[
  \frac{c_1(\theta)}{\sqrt{|\tau|}} \le |u_{2y}| \le \frac{C_1(\theta)}{\sqrt{|\tau|}}, \qquad \mbox{for} \,\,\,\, y\in D_{2\theta}.
  \]
  Combining this, \eqref{eq-u2y} and \eqref{eq-secant} yields
  \begin{equation}
    \label{eq-equiv-secant}
    \frac{c_1(\theta)}{\sqrt{|\tau|}} \le \frac{|w(y,\tau)|}{|W(u,\tau)|} \le \frac{C_1(\theta)}{\sqrt{|\tau|}}
  \end{equation}
  for all $y\in D_{2\theta}$ and $\tau \le \tau_0 \ll -1$.
  See Figure \ref{fig-conversion} on the next page. 
  
  By \eqref{eq-weight} we have $\mu(u,\tau) = - Y_1^2(u,\tau)/4$ for $u \in [\theta, 2\theta]$.
  Introducing the change of variables $y = Y_1(u,\tau)$ (or equivalently $u = u_1(y,\tau)$), the inequality \eqref{eq-equiv-secant} yields
  \[
  \begin{split}
    \int_{\theta}^{2\theta} W^2 \, e^{\mu(u,\tau)}\, du = \int_{\theta}^{2\theta} W^2 \, e^{- \frac{Y_1^2(u,\tau)}4}\, du &\le C(\theta) \sqrt{|\tau|} \int_{D_{2\theta}} w^2 e^{-\frac{y^2}{4}}\, dy
  \end{split}
  \]
  where we used that $du = (u_1)_y\, dy$ and that due to our asymptotics from \cite{ADS} in the intermediate region, we have
  \begin{equation}\label{eqn-cv}
    c_2(\theta) \sqrt{|\tau|} \le |(u_1)_y| \le \sqrt{|\tau|}\, C_2(\theta), \qquad \text{for } y\in D_{2\theta}.
  \end{equation}
  In conclusion
  \[
  \| W \chi_{_{[\theta, 2\theta]}} \|_{2,\infty} \leq C(\theta) \, \| w \, \chi_{_{D_{2\theta}}} \|_{2,\infty}
  \]
  which proves one of the inequalities in \eqref{eqn-normequiv}.

  \begin{figure}
    \includegraphics{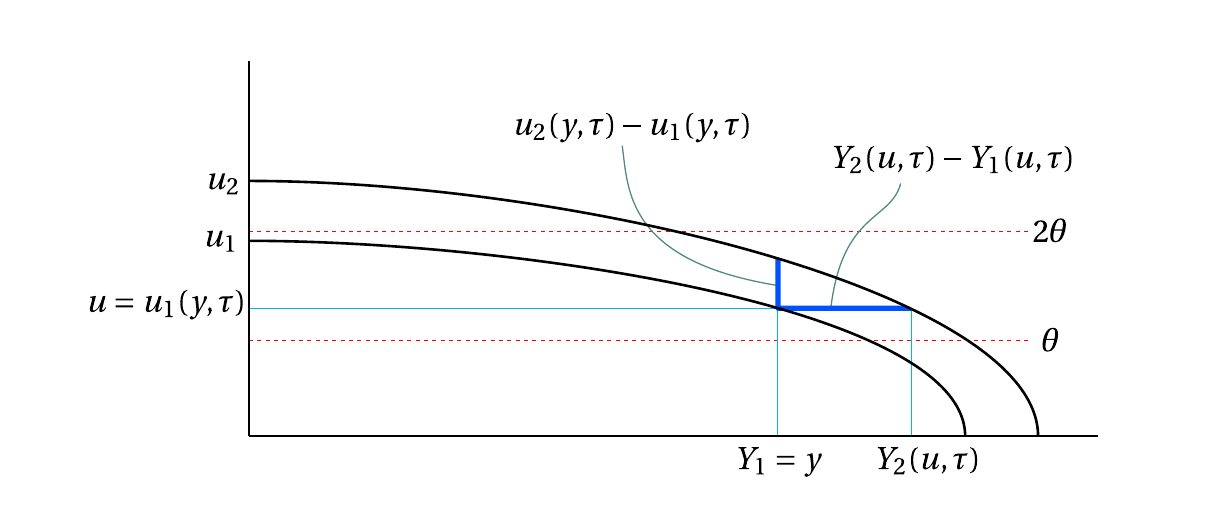}
    \caption{Converting the vertical distance $u_2(y, \tau)-u_1(y,\tau)$ to the horizontal distance $Y_2(u, \tau)-Y_1(u,\tau)$.
    Given a point $(y, u)$ on the graph of $u_1(\cdot,\tau)$ we define $Y_1 = y$, $u=u_1(y, \tau)$, $Y_2 = Y_2(u, \tau)$.
    By the Mean Value Theorem the ratio $\frac{u_2-u_1}{Y_2-Y_1}$ must equal the derivative $u_{2,y}(\tilde y,\tau)$ at some $\tilde y\in (Y_1, Y_2)$. }
    \label{fig-conversion}
  \end{figure}
  
  We will next show the other inequality in \eqref{eqn-normequiv}.
  To this end, we use again \eqref{eq-equiv-secant}, the change of variables $u=u_1(y,\tau)$ (or equivalently $y = Y_1(u,\tau)$) and \eqref{eqn-cv}, to obtain
  \begin{equation}
    \int_{D_{2\theta}} w^2 e^{-\frac{y^2}{4}}\, dy \le \frac{C(\theta)}{\sqrt{|\tau|}} \int_{\theta}^{2\theta} W^2 e^{-\frac{Y_1^2(u,\tau)}{4}}\, du = \frac{C(\theta)}{\sqrt{|\tau|}}\, \int_{\theta}^{2\theta} W^2 \, e^{\mu(u,\tau)}\, du
  \end{equation}
  from which the bound
  \[
  \| w \, \chi_{_{D_{2\theta}}} \|_{2,\infty} \leq C(\theta) \, \| W \chi_{_{[\theta, 2\theta]}} \|_{2,\infty}
  \]
  readily follows.
\end{proof}

We will next combine the main results in the previous two sections, Propositions \ref{prop-cylindrical} and \ref{prop-tip}, with the estimate \eqref{eqn-normequiv} above 
to establish our  {\em crucial estimate} which says that what actually dominates in the norm $\|w_\cyl\|_{\hv,\infty}$ is $\|\pr_0 w_\cyl\|_{\hv,\infty}$.  

\begin{prop}
  \label{prop-cor-main}
  For any $\epsilon >0$ there exists a $\tau_0 \ll -1$ so that we have
  \begin{equation}
    \label{eqn-w1230}
    \| \hat w_\cyl \|_{\hv,\infty} \leq \epsilon\, \|\pr_0 w_\cyl(\tau)\|_{\hv,\infty}.
  \end{equation}
\end{prop}

\begin{proof}
  By Proposition \ref{lem-rescaling-components-zero} we know that for every $\tau_0 \ll -1$ sufficiently small we can choose parameters $\alpha, \beta$ and $\gamma$ which are admissible with respect to $\tau_0$ and such that $\pr_+ w_\cyl(\tau_0) = \pr_0 w_\cyl(\tau_0) = 0$.
  From now on we will always consider $w(y,\tau) = u_1(y,\tau) - u_2^{\alpha\beta\gamma}(y,\tau)$, for these chosen parameters $\alpha, \beta$ and $\gamma$.
  
  By Proposition \ref{prop-cylindrical}, for every $\epsilon > 0$, there exists a $\tau_0 \ll -1$ so that
  \[
  \|\hat{w}_C\|_{\hv,\infty} < \frac{\epsilon}{3} ( \|w_\cyl\|_{\nu,\infty} + \|w \chi_{D_{\theta}\|_{\hilb,\infty}})
  \]
  where $D_{\theta} = \{y \,\,\,|\,\,\, {\theta}/{2} \le u_1(y,\tau) \le \theta\}$.
  Furthermore, by Lemma \ref{prop-norm-equiv}, by decreasing $\tau_0$ if necessary we ensure that the following holds
  \begin{equation}
    \label{eq-dominates1}
    \begin{split}
      \|\hat{w}_C\|_{\hv,\infty} &< \frac{\epsilon}{3} (\|w_\cyl\|_{\hv,\infty} + C(\theta) \|W\chi_{[\theta/2,\theta]}\|_{2,\infty}) \\
      &< \frac{\epsilon}{3}\, (\|w_\cyl\|_{\hv,\infty} + C(\theta)\,  \|W_T\|_{2,\infty})
    \end{split}
  \end{equation}
  where $\chi_{[\theta/2,\theta]}$ is the characteristic function of interval $u\in [\theta/2,\theta]$ and where we used the property of the cut off function $\varphi_T$ that $\varphi_T \equiv 1$ for $u\in [\theta/2,\theta]$.
  By Proposition \ref{prop-tip}, there exist $0 < \theta \ll 1$ and $\tau_0 \ll -1$ so that
  \[
  \|W_T\|_{2,\infty} < \frac{C(\theta)}{\sqrt{|\tau_0|}} \, \|W \chi_{[\theta,2\theta]}\|_{2,\infty}.
  \]
  By Lemma \ref{prop-norm-equiv} we have
  \[
  \|W_T\|_{2,\infty} \le \frac{C(\theta)}{\sqrt{|\tau_0|}} \, \|w \, \chi_{D_{2\theta}}\|_{\hilb,\infty} \le \frac{C(\theta)}{\sqrt{|\tau_0|}}\, \|w_\cyl\|_{\hilb,\infty}
  \]
  where we also use that $\varphi_\cyl \equiv 1$ on $D_{2\theta}$.
  Combining this with \eqref{eq-dominates1} yields
  \[
  \|\hat{w}_C\|_{\nu,\infty} < \epsilon\Bigl( \|w_\cyl\|_{\nu,\infty} + \frac{C(\theta)}{\sqrt{|\tau_0|}}\, \|w_\cyl\|_{\hilb,\infty}\Bigr) < \frac{2\epsilon}{3} \|w_\cyl\|_{\hv,\infty}
  \]
  by choosing $|\tau_0|$ sufficiently large  relative to $C(\theta)$.
  By choosing $\epsilon$ small,  the last estimate yields \eqref{eqn-w1230} finishing the proof of the proposition.
  
\end{proof}

\begin{proof}[Proof of the Main Theorem \ref{thm-main}]
  Recall that $w^{\alpha\beta\gamma}(y,\tau) = u_1(y,\tau) - u_2^{\alpha\beta\gamma}(y,\tau)$, which we shortly denote by $w(y,\tau) = u_1(y,\tau) - u_2(y,\tau)$, where $u_2^{\alpha\beta\gamma}(y,\tau)$ is given by \eqref{eq-ualphabeta}.
  Proposition \ref{lem-rescaling-components-zero} tells us that for every $\tau_0 \ll -1$ sufficiently small we can choose parameters $\alpha, \beta$ and $\gamma$ which are admissible with respect to $\tau_0$ and such that $\pr_+ w_\cyl(\tau_0) = \pr_0 w_\cyl(\tau_0) = 0$.
  For a given $\tau_0$ sufficiently small we fix parameters $\alpha, \beta$ and $\gamma$ so that above holds.
  Due to admissibility of the parameters all our estimates hold for the difference $w^{\alpha\beta\gamma}(y,\tau)$, for $\tau \le \tau_0 \ll -1$, independently of $\alpha, \beta$ and $ \gamma$.
  Our goal is to show that for that choice of parameters $w(y,\tau) \equiv 0$.
  
  Following the notation from previous sections we have
  \[
  \frac{\partial}{\partial \tau} w_\cyl = \mathcal{L}[w_\cyl] + \cE[w_\cyl] + \bar{\cE}[w,\varphi_\cyl]
  \]
  with $w_\cyl = \hat{w}_\cyl + a(\tau) \, \psi_2$, where $a(\tau) = \langle w_\cyl, \psi_2\rangle$.
  Projecting the above equation on the eigenspace generated by $\psi_2$ while using that $\langle \mathcal{L}[w_\cyl] , \psi_2 \rangle =0$ we obtain
  \[
  \frac{d}{d\tau}a(\tau) = \langle \cE[w_\cyl] + \bar{\cE}[w,\varphi_\cyl] , \psi_2 \rangle.
  \]
  Since ${\displaystyle \frac{\langle \psi_2^2,\psi_2\rangle}{\| \psi_2\|^2} = 8}$ we can write the above equation as
  \[
  \frac{d}{d\tau}a(\tau) = \frac{2a(\tau)}{|\tau|} + F(\tau)
  \]
  where
  \begin{equation}
    \label{eq-F-tau}
    \begin{split}
      F(\tau) &:= \frac{\langle \cE[w_\cyl]  + \bar{\cE}[w,\varphi_\cyl] - \frac{a(\tau)}{4|\tau|} \psi_2^2, \psi_2\rangle}{\|\psi_2\|^2} \\
      &= \frac{\langle\bar{\cE}[w,\varphi_\cyl], \psi_2\rangle}{\|\psi_2\|^2} + \frac{\langle \cE[w_\cyl] - \frac{a(\tau)}{4|\tau|} \psi_2^2, \psi_2\rangle}{\|\psi_2\|^2}.
    \end{split}
  \end{equation}
  Furthermore, solving the above ordinary differential equation for $a(\tau)$ yields
  \[
  a(\tau) = \frac{C}{\tau^2} - \frac{\int_{\tau}^{\tau_0} F(s) s^2\, ds}{\tau^2}.
  \]
  By Remark \ref{rem-choice-par} we may assume $\alpha(\tau_0) = 0$ and hence, $C = 0$, which implies
  \begin{equation}
    \label{eq-alpha-CC}
    |a(\tau)| = \frac{|\int_{\tau}^{\tau_0} F(s) s^2\, ds|}{\tau^2}.
  \end{equation}
  Define $\|a\|_{\hilb,\infty}(\tau) = \sup_{s\le \tau} \Bigl(\int_{s-1}^{s} |a(\zeta)|^2\, d \zeta \Bigr)^{\frac12}$.
  Since $\pr_0 w_\cyl(\cdot,\tau) = a(\tau)\, \psi_2(\cdot)$, we have
  \[
  \|\pr_0 w_\cyl\|_{\hv,\infty}(\tau) =\|a\|_{\hilb,\infty}(\tau) \, \|\psi_2\|_{\hv}.
  \]
  Denote by $\|a\|_{\hilb,\infty} :=\|a\|_{\hilb,\infty}(\tau_0)$.
  Note that
  \[
  \Bigl| \int_{\tau}^{\tau_0} F(s)\, s^2\, ds\Bigr| \le \sum_{j=[\tau]-1}^{\tau_0} \Bigl|\int_j^{j+1} s^2 F(s)\, ds\Bigr| \le C\, \sum_{j=[\tau]-1}^{\tau_0} j^2 \Bigl|\int_j^{j+1} F(s)\, ds\Bigr|
  \]
  where with no loss of generality we may assume $\tau_0$ is an integer.
  Next we need the following claim.
  
  \begin{claim}
    \label{claim-Fs}
    For every $\epsilon > 0$ there exists a $\tau_0$ so that
    \[\Bigl|\int_{\tau-1}^{\tau} F(s)\, ds\Bigr| \le
    \frac{\epsilon}{|\tau|}\,\|a\|_{\hilb,\infty}
    \]
    for $\tau \le \tau_0$.
  \end{claim}
  
  Assume for the moment that the Claim holds.
  Then,
  \begin{equation*}
    \begin{split}
      \Bigl|\int_{\tau}^{\tau_0} F(s) \, s^2\, ds\Bigr| &\le \sum_{j=[\tau]}^{\tau_0} \int_{j-1}^j s^2 F(s)\, ds \le \epsilon\, \|a\|_{\hilb,\infty}\,\sum_{j=[\tau]-1}^{\tau_0} |j|\,   \\
      &\le \epsilon\,\|\alpha\|_{\hilb,\infty} \, \sum_{j=[\tau]-1}^{\tau_0} | j |   \\
      &\le \epsilon\, |\tau|^2\,\|a\|_{\hilb,\infty}.
    \end{split}
  \end{equation*}
  Combining this with \eqref{eq-alpha-CC}, where $\epsilon \le 1/2$, yields
  \[
  |a(\tau)| \le \frac12\|a\|_{\hilb,\infty}, \qquad \mbox{for all} \,\,\, \tau\le \tau_0.
  \]
  This implies
  \[
  \|a\|_{\hilb,\infty} \le \frac12 \, \|a\|_{2,\infty}
  \]
  and hence $\|a\|_{\hilb,\infty} = 0$, which further gives
  \[
  \|\pr_0 w_\cyl\|_{\hv,\infty} = 0.
  \]
  Finally, \eqref{eqn-w1230} implies $\hat{w}_C \equiv 0$ and hence, $w_\cyl \equiv 0$ for $\tau \le \tau_0$.
  By \eqref{eqn-normequiv} and the fact that $\varphi_\cyl \equiv 1$ on $D_{2\theta}$ we have $W\chi_{[\theta,2\theta]} \equiv 0$ for $\tau \le \tau_0$.
  Proposition \ref{prop-tip} then yields that  $W_T \equiv 0$ for $\tau\le \tau_0$.
  All these imply $u_1(y,\tau) \equiv u_2^{\alpha\beta\gamma}(y,\tau)$, for $\tau\le \tau_0$.
  By forward uniqueness of solutions to the mean curvature flow (or equivalently to cylindrical equation \eqref{eq-u}), we have $u_1 \equiv u_2^{\alpha\beta\gamma}$, and hence $
  M_1 \equiv M_2^{\alpha\beta\gamma}.$
  This  concludes the proof of the main Theorem \ref{thm-main}.
  
  \smallskip 
  To complete the proof of Theorem \ref{thm-main} we still need to prove Claim \ref{claim-Fs}, what we do below.
  
  \begin{proof}[Proof of Claim \ref{claim-Fs}]
    Throughout the proof we will use the estimate
    \begin{equation}
      \label{eq-par-int}
      \|w_\cyl\|_{\hv,\infty} \leq C \, \| a\|_{\hilb,\infty}, \qquad \mbox{for} \,\, \tau_0 \ll -1
    \end{equation}
    which follows from Proposition \ref{prop-cor-main}.
    By the proof of the same Proposition we also have
    \[
    \|w \, \chi_{D_{\theta}}\|_{\hilb,\infty} < \frac{C(\theta)}{\sqrt{|\tau_0|}}\, \|w_\cyl\|_{\hilb,\infty}, \qquad \mbox{for} \,\, \tau_0 \ll -1.
    \]
    Also throughout the proof we will use the a'priori  estimates on the solutions $u_i$ shown in our previous work \cite{ADS} which continue to hold here  without the assumption of $O(1)$ symmetry, 
    as we discuss  in Theorem \ref{thm-O1} of our current paper.  
    
    \smallskip
    
    From the definition of $\bar{\cE}[w,\varphi_\cyl]$ given in \eqref{eq-bar-E} and the definition of cut off function $\varphi_\cyl$, we see that the support of $\cE[w,\varphi_\cyl]$ is contained in
    \[
    \Bigl(\sqrt{2 - \frac{\theta^2}{n-1}} - \epsilon_1\Bigr)\, \sqrt{|\tau|} \le |y| \le \Bigl(\sqrt{2 - \frac{\theta^2}{4(n-1)}} + \epsilon_1\Bigr)\, \sqrt{|\tau|}
    \]
    where $\epsilon_1$ is so tiny that $\sqrt{2 - \frac{\theta^2}{4(n-1)}} + \epsilon_1 < \sqrt{2}$.
    Also by the \emph{a priori} estimates proved in \cite{ADS} we have \begin{equation}
      \label{eq-der-bounds10}
      |u_y| + |u_{yy}| \le \frac{C(\theta)}{\sqrt{|\tau|}}, \qquad \mbox{for} \,\, |y| \leq \big (\sqrt{2 - \frac{\theta^2}{4(n-1)}} + \epsilon_1\big ) \sqrt{|\tau|}.
    \end{equation}
    Furthermore, Lemma 5.14 in \cite{ADS} shows that our ancient solutions $u_i$, $i\in \{1,2\}$ satisfy
    \begin{equation}
      \label{eq-L2-asymp0}
      \begin{split}
        \Bigl \|u_i - \sqrt{2(n-1)} + \frac{\sqrt{2(n-1)}}{4|\tau|}\, \psi_2 \Bigr\| &= o(|\tau|^{-1}), \\
        \Bigl\|\Bigl(u_i + \frac{\sqrt{2(n-1)}}{4|\tau|}\, \psi_2\Bigr)_y\Bigr\| &= o(|\tau|^{-1}).
      \end{split}
    \end{equation}
    In particular, this implies
    \begin{equation}
      \label{eq-L2-asymp1}
        \Bigl\|u_i - \sqrt{2(n-1)} \Bigr\| = O(|\tau|^{-1}) \qquad \mbox{and} \qquad \Bigl\|(u_i) _y\Bigr\| = O(|\tau|^{-1}).
    \end{equation}
    
    We start by estimating the first term on the right hand side in \eqref{eq-F-tau}.
    Using Lemma \ref{lem-error1-est} we conclude
    \begin{equation}
      \label{eq-barE-est}
      \begin{split}
        |\langle \bar{\cE}[w,\varphi_\cyl], \psi_2\rangle| \le \|\bar{\cE}[w,\varphi_\cyl]\|_{\hv^*} \|\psi_2 \, \bar {\chi} \|_{\hv} < \epsilon \, \|w_\cyl\|_{\hv} \, e^{-|\tau|/4}.
      \end{split}
    \end{equation}
    where $\bar {\chi} $ denotes a smooth function with a support in $|y| \ge (\sqrt{2-\theta^2/(4(n-1))} - 2\epsilon_1)\, \sqrt{|\tau|}$, being equal to one for $|y| \ge (\sqrt{2-\theta^2/(4(n-1))} - \epsilon_1)\, \sqrt{|\tau|}$.
    This implies for every $\epsilon > 0$ we can find a $\tau_0 \ll -1$ so that for $\tau \le \tau_0$ we have
    \[
    \Bigl|\int_{\tau-1}^{\tau} \langle \bar{\cE}[w,\varphi_\cyl], \psi_2\rangle\, ds \Bigr| \leq \frac{\epsilon\|a\|_{\hilb,\infty}}{|\tau|}
    \]
    where we used \eqref{eq-par-int}.
    
    We focus next on the second term on the right hand side in \eqref{eq-F-tau}.
    Lets write $w_\cyl = \hat{w}_\cyl + a(\tau) \psi_2$.
    Recall that
    \begin{equation}
      \label{eq-E-recall}
      \cE[w_\cyl] = \frac{2(n-1) - u_1u_2}{2u_1u_2}w_\cyl -\frac{u_{1y}^2}{1+u_{1y}^2} (w_\cyl)_{yy} -\frac{(u_{1y}+u_{2y})u_{2yy}}{(1+u_{1y}^2)(1+u_{2y}^2)} (w_\cyl)_y.
    \end{equation}
    Then, for the first term on the right hand side of \eqref{eq-E-recall} we get
    \begin{equation}
      \label{eq-first-to-estimate}
      \begin{split}
        & \Bigl| \Bigl\langle \frac{2(n-1) - u_1 u_2}{2u_1 u_2}\, w_\cyl - \frac{a(\tau)}{4|\tau|}\, \psi_2^2, \psi_2\Bigr\rangle\Bigr|  \le \\
        &\le \Bigl| \Bigl\langle \frac{2(n-1) - u_1 u_2}{2 u_1 u_2}\, \hat{w}_C, \psi_2\Bigr\rangle\Bigr| + |a(\tau)|\, \Bigl|\Bigl\langle \frac{2(n-1) - u_1 u_2}{2 u_1 u_2} - \frac{1}{4|\tau|}\, \psi_2, \psi_2^2\Bigr\rangle\Bigr|.
      \end{split}
    \end{equation}
    To estimate the first term on the right hand side in \eqref{eq-first-to-estimate}, we write
    \begin{equation}
      \label{eq-help-help100}
      \begin{split}
        & \Bigl|\Bigl\langle \frac{2(n-1) - u_1 u_2}{2 u_1 u_2}\, \hat{w}_C, \psi_2\Bigr\rangle\Bigr| \le
        \Bigl|\Bigl\langle \frac{(\sqrt{2(n-1)} - u_1)(\sqrt{2(n-1)} + u_1)}{2 u_1 u_2} \hat{w}_C, \psi_2\Bigr\rangle\Bigr| + \\
        &+ \Bigl|\Bigl\langle \frac{u_1 - \sqrt{2(n-1)}}{2 u_2}\, \hat{w}_C, \psi_2\Bigr\rangle\Bigr| + \Bigl|\Bigl\langle \frac{\sqrt{2(n-1)} - u_2}{2 u_2}\, \hat{w}_C, \psi_2\Bigr\rangle\Bigr|.
      \end{split}
    \end{equation}
    Note that ${\displaystyle u_i \ge {\theta}/{2}}$ on the support of $\hat{w}_C$ and hence the arguments for estimating either of the terms on the right hand side in \eqref{eq-help-help100}  are analogous to estimating the second term in \eqref{eq-help-help100}.
    Using Lemma \ref{lem-Poincare}, Proposition \ref{prop-cor-main} and \eqref{eq-L2-asymp1} we get that for every $\epsilon > 0$ there exists a $\tau_0 \ll -1$ so that for $\tau \le \tau_0$ we have
    \[
    \begin{split}
      \Bigl|\Bigl\langle &\frac{u_1 - \sqrt{2(n-1)}}{2u_2}\, \hat{w}_C, \psi_2\Bigr\rangle\Bigr| \\
      &\le C(\theta)\, \Bigl(\int \hat{w}_C^2 |\psi_2| e^{-y^2/4}\, dy\Bigr)^{1/2} \, \Bigl(\int (\sqrt{2(n-1)} - u_1)^2 |\psi_2| e^{-y^2/4}\, dy\Bigr)^{1/2}\\
      &\le C(\theta) \|\hat{w}_C\|_{\hv} \, \|\sqrt{2(n-2)} - u_1\|_{\hv} \\
      &< \frac{\epsilon}{|\tau|}\,\|a\|_{\hilb,\infty}
    \end{split}
    \]
    implying
    \begin{equation}
      \label{eq-term1-bar}
      \Bigl|\int_{\tau-1}^{\tau} \Bigl\langle \frac{2(n-1) - u_1 u_2}{2u_1u_2}\, \hat{w}_C, \psi_2\Bigr\rangle\, ds\Bigr| < \frac{\epsilon}{|\tau|}\,\|a\|_{\hilb,\infty}.
    \end{equation}
    Lets now estimate the second term on the right hand side in \eqref{eq-first-to-estimate}.
    Writing $u_i = \sqrt{2(n-1)} (1 + v_i)$, we get
    \begin{equation}
      \label{eq-noname111}
      \begin{split}
        \Bigl\langle & \frac{2(n-1) - u_1 u_2}{2 u_1 u_2} - \frac{1}{4|\tau|} \psi_2, \psi_2^2\Bigr\rangle \\
        &= - \Bigl\langle \frac{v_1 + v_2 + v_1 v_2}{2(1 + v_1) (1 + v_2)} + \frac{1}{4|\tau|}\, \psi_2, \psi_2^2\Bigr\rangle \\
        &= -\frac 12 \Bigl\langle \frac{v_1}{(1+ v_1)(1+v_2)} + \frac{\psi_2}{4|\tau|}, \psi_2^2\Bigr\rangle - \frac12 \Bigl\langle \frac{v_2}{1 + v_2} + \frac{\psi_2}{4|\tau|}, \psi_2^2\Bigr\rangle.
      \end{split}
    \end{equation}
    The two terms on the right hand side in above equation can be estimated in the same way so we will demonstrate how to estimate the second one.
    Using \eqref{eq-L2-asymp0}, \eqref{eq-L2-asymp1}, the bound \eqref{eq-v-quadratic-upper-bound} and H\"older's inequality we get that for every $\epsilon > 0$ there exist $K$ large enough and $\tau_0 \ll -1$ so that for $\tau \le \tau_0$ we have
    \[
    \begin{split}
      \Bigl\langle \frac{v_2}{1+v_2} + \frac{\psi_2}{4|\tau|}, \psi_2^2\Bigr\rangle
      &= \langle v_2 + \frac{\psi_2}{4|\tau|}, \psi_2^2\rangle - \langle \frac{v_2^2}{1+v_2}, \psi_2^2\rangle  \\
      &\le C\, \Bigl\| v_2 + \frac{\psi_2}{4|\tau|}\Bigr\| + C\int_{\mathbb{R}} v_2^2\,  y^4\,  e^{-\frac{y^2}{4}}\, dy \\
      &\le \frac{o(1)}{|\tau|} + \Bigl(\int_{\mathbb{R}} v_2^2 \, e^{-\frac{y^2}{4}}\, dy\Bigr)^{\frac12} \, \Bigl(\int_{\mathbb{R}} v_2^2 \, y^8 \, e^{-\frac{y^2}{4}}\, dy\Bigr)^{\frac12} \\
      &\le \frac{o(1)}{|\tau|} + \frac{C}{|\tau|} \left( \Bigl(\int_{|y| \le K} v_2^2 e^{-\frac{y^2}{4}}\, dy\Bigr)^{\frac12} + \Bigl(\int_{|y| \ge K} y^{10} e^{-\frac{y^2}{4}}\, dy\Bigr)^{\frac12} \right) \\
      &< \frac{\epsilon}{|\tau|}.
    \end{split}
    \]
    To justify the last inequality note that for a given $\epsilon > 0$ we can find $K$ large enough so that ${\displaystyle \Bigl(\int_{|y|\ge K} y^{10} e^{-\frac{y^2}{4}}\, dy\Bigr)^{\frac12} < \frac{\epsilon}{6C}}$.
    On the other hand, using our asymptotics result proven in \cite{ADS}, for a chosen $K$ we can find a $\tau_0 \ll -1$, so that for $\tau \le \tau_0$ we have ${\displaystyle |v_i| < \frac{\epsilon}{6C\sqrt{K}}}$.
    Finally, we conclude that for every $\epsilon > 0$ there exists a $\tau_0 \ll -1$, so that for all $\tau \le \tau_0$,
    \begin{equation}
      \label{eq-term2-bar}
      \Bigl|\int_{\tau-1}^{\tau} \Bigl\langle \frac{2(n-1) - u_1 u_2}{2 u_1 u_2} - \frac{\psi_2}{4|\tau|}, \psi_2^2\Bigr\rangle \, ds\Bigr| < \frac{\epsilon}{2 |\tau|}\,\|a\|_{\hilb,\infty}.
    \end{equation}
    Since the first term on the right hand side in \eqref{eq-noname111} can be estimated in a similar manner, we conclude that this inequality holds.
    
    It remains now to estimate the second and third terms in the error term \eqref{eq-E-recall}, which involve first and second order derivative bounds for our solutions $u_i$.
    We claim that for every $K$ there exist $\tau_0 \ll -1$ and a uniform constant $C$ so that
    \begin{equation}
      \label{eq-der-bounds11}
      |(u_i)_y| + |(u_i)_{yy}| \le \frac{C}{|\tau|}, \qquad \mbox{for} \,\,\,\, |y| \le K, \,\, \tau \leq \tau_0, \,\,\, \quad i=1,2.
    \end{equation}
    This follows by standard derivative estimates applied to the equation satisfied by each of the  $v_i$, $i=1,2$  and the $L^\infty$  bound  $|v_i| \leq \frac{C}{|\tau|}$, 
    which holds on $|y| \le 2K, \, \tau \leq \tau_0
    \ll -1$. 
    
    Let us use \eqref{eq-der-bounds11} to estimate the projection involving the third term in \eqref{eq-E-recall}: for every $\epsilon > 0$, there exists a $\tau_0 \ll -1$ so that for 
    $\tau \le \tau_0$
    \[
    \begin{split}
      \Bigl|\Bigl\langle
      &\frac{({u_1}_y + {u_2}_y) u_{2yy}}{(1+u_{1y}^2)\, (1+u_{2y}^2)}\, (w_\cyl)_y, \psi_2\Bigr\rangle\Bigr| \\
      &\le C \,   \int_{|y| \le K} (|u_{1y}| + |u_{2y}|)\, |u_{2yy}|\,  |(w_\cyl)_y| \, (y^2 +1 ) \, e^{-\frac{y^2}{4}}\, dy  \\
      &\qquad  +C\,  \int_{|y| \ge K } (|u_{1y}| + |u_{2y}|)\,  |u_{2yy}|\,  |(w_\cyl)_y| \,y^2 \, e^{-\frac{y^2}{4}}\, dy\\
      &\le \frac{C}{|\tau|^2}\, \|w_\cyl\|_{\hv} + \frac{C}{|\tau|}\, \|w_\cyl\|_{\hv} \, \Bigl(\int_{|y| \ge K} \,  y^4 \, e^{-\frac{y^2}{4}}\, dy \Bigr)^{\frac12} \\
      &< \frac{\epsilon}{|\tau|}\, \|w_\cyl\|_{\hv}
    \end{split}
    \]
    where we used H\"older's inequality, estimate \eqref{eq-der-bounds10} in the region $\{|y| \ge K\}\cap \supp w_\cyl\}$ and estimate \eqref{eq-der-bounds11} in the region $\{|y| \le K\}$.
    This implies that for every $\epsilon > 0$ there exists a $\tau_0 \ll -1$ so that for
    \begin{equation}
      \label{eq-term3-bar}
      \Bigl|\int_{\tau-1}^{\tau} \Bigl\langle \frac{(u_1y + u_{2y}) u_{2yy}}{(1+u_{1y}^2)(1+u_{2y}^2)}\, (w_\cyl)_y, \psi_2\Bigr\rangle\Bigr| < \frac{\epsilon}{|\tau|}\, \|w_\cyl\|_{\hv,\infty} < \frac{\epsilon}{|\tau|}\,\|a\|_{\hilb,\infty}.
    \end{equation}
    
    Finally, to estimate the projection involving the second term in \eqref{eq-E-recall}, we note that integration by parts yields
    \begin{equation}
      \label{eq-term4-bar}
      \begin{split}
        \Bigl\langle
        & \frac{u_{1y}^2}{1 + u_{1y}^2} (w_\cyl)_{yy}, \psi_2\Bigr\rangle \\
        &= -2\int_{\mathbb{R}} \frac{u_{1yy} u_{1y}}{1 + u_{1y}^2}\, (w_\cyl)_y \, \psi_2 \, e^{-\frac{y^2}{4}}\, dy
        +  2\int_{\mathbb{R}} \frac{u_{1y}^3 u_{1yy}}{(1 + u_{1y}^2)^2}\,  (w_\cyl)_y \, \psi_2 \, e^{-\frac{y^2}{4}}\, dy  \\
        &\quad - \int_{\mathbb{R}} \frac{ u_{1y}^2}{1 + u_{1y}^2 }\, (w_\cyl)_{y} \, (\psi_2)_y \, e^{-\frac{y^2}{4}}\, dy + \frac12 \int_{\mathbb{R}} \frac{u_{1y}^2 }{1+u_{1y}^2} (w_\cyl)_y \psi_2 \, y \, e^{-\frac{y^2}{4}}\, dy.
      \end{split}
    \end{equation}
    It is easy to see that all terms on the right hand side in \eqref{eq-term4-bar} can be estimated very similarly as in \eqref{eq-term3-bar}.
    Hence, for every $\epsilon > 0$ there exists a $\tau_0$ so that for all $\tau \le \tau_0$ we have
    \begin{equation}
      \label{eq-term5-bar}
      \Bigl|\int_{\tau-1}^{\tau} \Bigl\langle \frac{u_{1y}^2}{1 + u_{1y}^2} (w_\cyl)_{yy}, \psi_2\Bigr\rangle\Bigr| < \frac{\epsilon}{|\tau|}\,\|a\|_{\hilb,\infty}.
    \end{equation}
    Combining \eqref{eq-F-tau}, \eqref{eq-barE-est}, \eqref{eq-E-recall}, \eqref{eq-term1-bar}, \eqref{eq-term2-bar}, \eqref{eq-term3-bar}, \eqref{eq-term4-bar} and \eqref{eq-term5-bar} concludes the Claim \ref{claim-Fs}.
    
  \end{proof}
  
  The proof of our theorem is now also complete. 
\end{proof}

\section{Appendix - Reflection symmetry}
In this appendix we will justify why the conclusions of Theorem \ref{thm-old} proved in \cite{ADS} under the assumption on $O(1)\times O(n)$ symmetry hold in the presence of $O(n)$-symmetry only.
More precisely we will show the following result.

\begin{theorem}
  \label{thm-O1}
  If $M_t$ is an Ancient Oval that is  rotationally symmetric,  then the conclusions of Theorem \ref{thm-old} hold.
\end{theorem}

\begin{proof}
  We will follow closely the arguments in Theorem \ref{thm-old} and point out below only steps in which the arguments slightly change because of the lack of reflection symmetry.
  All other estimates can  be argued in exactly the same way. 
  
  Recall that we consider noncollapsed, ancient solutions (and hence convex due to \cite{HK}) which are $O(n)$-invariant hypersurfaces in $\mathbb{R}^{n+1}$.
  Such hypersurfaces can be represented as
  \[\{(x,x') \,\,\in \mathbb{R}\times\mathbb{R}^n\,\,\,|\,\,\, -d_1(t) < x < d_2(t), \,\,\, \|x'\| = U(x,t)\}\]
  for some function $\|x'\| = U(x,t)$.
  The points $(-d_1(t),0)$ and $(d_2(t),0)$ are called the tips of the surface.
  The profile function $U(x,t)$ is defined only for $x\in [-d_1(t), d_2(t)]$.
  After parabolic rescaling
  \[U(x,t) = \sqrt{T - t} \, u(y,\tau), \qquad y = \frac{x}{\sqrt{T-t}}, \, \qquad \tau = -\log(T-t)\]
  the profile function $u(y,\tau)$ is defined for $-\bar{d}_1(\tau) \le y \le \bar{d}_2(\tau)$.
  Theorem 1.11 in \cite{HK} and Corollary 6.3 in \cite{W} imply that as  $\tau \to -\infty$, surfaces $M_{\tau}$  converge in $C_{loc}^{\infty}$ to a cylinder of radius $\sqrt{2(n-1)}$, with axis passing through the origin. 
  
  Due to concavity, for every $\tau$, there exists a $y(\tau)$ so that $u_y(\cdot,\tau) \le 0$ for $y \ge y(\tau)$, $u_y(\cdot,\tau) \ge 0$ for $y \le y(\tau)$ and $u_y(y(\tau),\tau) = 0$.
  To finish the proof of Theorem \ref{thm-O1} we need the following lemma saying the maximum of $H$ is attained at one of the tips.
  
  \begin{lemma}
    \label{lem-Hmax-tips}
    We have that $(\lambda_1)_y \ge 0$ for $y \in [y(\tau), \bar{d}_1(\tau))$ and $(\lambda_1)_y \le 0$ for $y \in (-\bar{d}_2(\tau), y(\tau)]$.
    As a consequence, the  mean curvature $H$ on $M_t$ attains its maximum at one of the tips $(-d_1(t),0)$ or $(d_2(t),0)$.
  \end{lemma}
  
  \begin{proof}
    We follow the proof of  Corollary 3.8 in \cite{ADS} where the result  followed from the fact that the scaling invariant quantity 
    \[ R := \frac{\lambda_n}{\lambda_1} = -\frac{u u_{yy}}{1 + u_y^2} \geq  0\]
    satisfies 
    \begin{equation}\label{eqn-RRR}
      R \leq 1. 
    \end{equation}
    Let us then show that \eqref{eqn-RRR} still holds in our case.
    Note that at umbilic points one has $R = 1$.
    Both tips of the surface are umbilic points and hence we have $R = 1$ at the tips for all $\tau$ (here we use that the surface is smooth and strictly convex and radially symmetric at the tips).
    Hence, $R_{\max}(\tau)$ is achieved on the surface for all $\tau$ and is larger or equal than one.
    Thus it is sufficient to show that   $R_{\max}(\tau) \leq 1$. 
    We first note that  the quantity ${\displaystyle Q := \frac{u_y^2}{u^2(1+u_y^2)}}$ that we considered before in \cite{ADS} satisfies  $Q_y \ge 0$ for $y \ge y(\tau)$ and $Q_y \le 0$ for $y \le y(\tau)$.     
    
    To prove \eqref{eqn-RRR}, we may assume $R_{\max}(\tau) =R(\bar y_\tau,\tau) > 1$,  for all $\tau \le \tau_0$ and some $\bar y_\tau \in \bar M_\tau$, since otherwise the statement is true.  
    The convergence to the  cylinder in the middle implies that $|\bar{y}_{\tau}| \to +\infty$,  as $\tau \to -\infty$.
    As in the proof of Lemma 3.5 in \cite{ADS} it is enough to show that the
    \begin{equation}
      \label{eq-Q-lower}
      \liminf_{\tau\to -\infty} Q(\bar{y}_{\tau},\tau) \ge c > 0
    \end{equation}
    for a uniform constant $c > 0$ and all $\tau \le \tau_0$. 
    
    The same proof as in \cite{ADS} implies there exists a uniform constant $c_1 > 0$ so that for all $\tau \le \tau_0 \ll -1$ we have that 
    \begin{equation}
      \label{eq-Q-1}
      Q(y,\tau) \ge c_1, \qquad \mbox{whenever} \,\,\, R(y,\tau) = 1.
    \end{equation}
    
    We claim that  this implies \eqref{eq-Q-lower}.
    To prove this claim we argue by contradiction and hence, assume that there exists a sequence $\tau_i\to -\infty$ for which $Q(\bar{y}_{\tau_i},\tau_i) \to 0$ as $i\to \infty$.
    This implies that the $\lim_{\tau\to -\infty} R(y,\tau) = 0$, uniformly for $y$ bounded.
    We conclude that for all $\tau \le \tau_0$ there exists at least one point $y_{\tau}$ such that $R(y_{\tau},\tau) = 1$.
    The convergence to the cylinder also implies that
    Without loss of generality we may take  that for a subsequence,  $y(\tau_i) < \bar{y}_{\tau_i}$.
    We consider two different cases.
    
    \smallskip 
    \noindent{\em Case} 1.  $R(y(\tau_i),\tau_i) \le 1$. 
    Then, either  $R(y(\tau_i),\tau_i) = 1$ (in which  case set  $\hat y_{\tau_i} := y(\tau_i)$),  or $R(y(\tau_i),\tau_i) < 1$ (in which case we  find $\hat y_{\tau_i} \in ( y(\tau_i), \bar{y}_{\tau_i})$ so that $R(y_{\tau_i},\tau_i) = 1$).
    In either case, since $R(\hat y_{\tau_i}, \tau_i)=1$, \eqref{eq-Q-1} implies  that
    $Q(\hat y_{\tau_i},\tau_i) \ge c_1$, for $i \geq i_0$. 
    Since $Q_y(\cdot,\tau) \ge 0$ for $y \ge y(\tau)$ and $\bar y_{\tau_i} \geq \hat y_{\tau_i} \geq y(\tau_i)$,   we conclude that 
    $Q(\bar{y}_{\tau_i},\tau_i) \ge c_1 > 0$,  for $ i\ge i_0$ 
    contradicting our assumption that the $\lim_{i\to\infty} Q(\bar{y}_{\tau_i},\tau_i) = 0$.
    \smallskip
    
    \noindent{\em Case} 2.  $R(y(\tau_i),\tau_i) >  1$.
    Recall that $u(y,\tau)$ satisfies the equation
    \[\frac{\partial}{\partial\tau} u = \frac{u_{yy}}{1 + u_y^2} - \frac y2 u_y + \frac u2 - \frac{n-1}{u} = -H \sqrt{1 + u_y^2} - \frac y2 u_y + \frac u2.\]
    The maximum of $u(\cdot,\tau)$ is achieved at $y(\tau)$ and hence, by \eqref{eq-H} we have
    \[\frac{d}{d\tau} u_{\max} \ge -C + \frac{u_{\max}}{2}\]
    implying that
    \[u(y(\tau),\tau) = u_{\max}(\tau) \le \max\{2C, u_{\max}(\tau_0)\}, \qquad \mbox{for} \,\,\,\, \tau \le \tau_0.\]
    On the other hand, due to the convergence to the   cylinder of radius $\sqrt{2(n-1)}$ in the middle we have that $u_{\max}(\tau) \ge u(0,\tau) \ge \frac 12 \,\sqrt{2(n-1)}$ for $\tau \le \tau_0 \ll -1$.
    All these imply that for $\tau\le \tau_0 \ll -1$ we have
    \[C_0 \ge H(y(\tau),\tau) \ge \frac{n-1}{u} \ge c_0 > 0.\]
    Hence, we can take a limit around $(y(\tau_i), u(y(\tau_i),\tau_i)$ to conclude that the limit is a complete graph of a concave, nonnegative function $\hat{u}(y,\tau)$ so that $\hat{u}_y(0,0) = 0$.
    All these yield $\hat{u} \equiv \text{ constant}$, that is  the limit is the round cylinder $\mathbb{R}\times S^{n-1}$, contradicting that $R(y(\tau_i),\tau_i) > 1$.
    
    This finishes  the proof of estimate \eqref{eq-Q-lower} and then we can argue as in the proof of Lemma 3.5 in \cite{ADS} to conclude the proof that $R \le 1$, for $\tau \le \tau_0 \ll -1$. 
    
    To finish the proof of Lemma \ref{lem-Hmax-tips}, note that $R \le 1$ on $M_{\tau}$, for $\tau \le \tau_0$ implies that 
    $$(\lambda_1)_y  \ge 0,\,\,\,  \mbox{for} \,\,  y\in [y(\tau),\bar{d}_1(\tau)] \quad \mbox{and} \quad 
    (\lambda_1)_y \le 0 \,\,\, \mbox{for} \,\, y\in [-\bar d_2(\tau),y(\tau)].$$
    We now conclude as in the proof of Corollary 3.8 in \cite{ADS} that  
    Hence, 
    \[H(y,\tau) \le \max \big ( H(\bar{d}_1(\tau),\tau), H(\bar{d}_2(\tau),\tau) \big ),  \qquad y\in M_{\tau}\]
    for all $\tau \le \tau_0 \ll -1$ finishing the proof of Lemma \ref{lem-Hmax-tips}.
  \end{proof}
  
  The a'priori estimates from Section 4 in \cite{ADS} hold as well in our case,  one has just to use that $u_y \le 0$ for $y \in [y(\tau), \bar{d}_1(\tau)]$ and $u_y \ge 0$ for $y\in [-\bar{d}_2(\tau),y(\tau)]$.
  By using the  same barriers that  we constructed in \cite{ADS} one can easily see that we still  have the inner-outer estimate we showed  in Section 4.5 in \cite {ADS}.
  Note that the same inner-outer estimates were proved and the same barriers were used in \cite{BC} without assuming any symmetry.
  
  \begin{lemma}\label{lem-must-be-inside}
    There is an $L_n>0$ such that for any  rescaled Ancient Oval  $u(y,\tau)$   there exist sequences $\tau_i, \tau_i'\to-\infty$ such that for all $i=1, 2, 3, \dots$ one has
    \[
    u(L_n, \tau_i) < \sqrt{2(n-1)}  \quad\text{and}\quad u(-L_n, \tau_i') < \sqrt{2(n-1)}.
    \]
  \end{lemma}
  \begin{proof}
    Choose $L_n$ so that the region $\{(y, u) : y\geq L_n, 0\leq u\leq \sqrt{2(n-1)}\}$ is foliated by self-shinkers as in \cite{ADS}, i.e.~for each $a\in(0, \sqrt{2(n-1)})$ there is a unique solution $U_a:[L_n, \infty)\to\R$ of 
    \begin{equation}
      \frac{U_{yy}}{1+U_y^2} - \frac y2 U_y + \frac 12 U - \frac{n-1}{U} = 0,
      \qquad U(L_n) = a.
    \end{equation}
    To prove the Lemma we argue by contradiction and assume that the sequence $\tau_i$ does not exist.
    This means that for some $\tau_*$ one has $u(L_n, \tau)\geq \sqrt{2(n-1)}$ for all $\tau\leq \tau_*$.
    The same arguments as in \cite[Section 4]{ADS} then imply that $u(y, \tau) \geq U_a(y)$ for all $y\geq L_n$, any $\tau\leq \tau_*$ and any $a\in (0, \sqrt{2(n-1)})$.
    This implies that $u(y,\tau)\geq \sqrt{2(n-1)}$ for all $y\geq L_n$  and therefore contradicts the compactness of $M_\tau$.
  \end{proof}
  
  For  any of our rescaled rotationally symmetric Ancient Ovals  $u(y,\tau)$,  then we can consider the truncated difference
  \[
  v(y, \tau) = \varphi(\frac yL) \bigl ( \frac{u(y, \tau)}{\sqrt{2(n-1)}} - 1 \bigr ) 
  \]
  for some large $L$.
  This function satisfies
  \begin{equation}
    v_\tau = \cL v + E(\tau)
  \end{equation}
  where $E$ contains the nonlinear as well as the cut-off terms, and where $\cL$ is the operator
  \[
  \cL\phi = \phi_{yy} - \frac y2 \phi_y + \phi.
  \]
  Using the fact that $v$ comes from an ancient solution, and by comparing the Huisken functionals of $M_\tau$ with that of the cylinder we can show as in \cite{ADS} that for any $\epsilon>0$ one can choose $\ell=\ell_\epsilon$ and $\tau_\epsilon<0$ large enough so that 
  \begin{equation}
    \label{eq-error-small}
    \| E(\tau) \|_\hilb \leq \epsilon \|v(\cdot, \tau)\|_\hilb
  \end{equation}
  holds for all $\tau\leq\tau_\epsilon$.
  
  As in \cite{ADS} we can decompose $v$ into eigenfunctions of the linearized equation, i.e.
  \[
  v(y, \tau) = v_-(y, \tau) + c_2(\tau) \psi_2(y) + v_+(y, \tau) 
  \] 
  with the only difference that $v_\pm$ are no longer necessarily even functions of $y$.
  The component in the unstable directions now has two terms,
  \[
  v_+(y) = c_0(\tau) \psi_0(y) + c_1(\tau) \psi_1(y) = c_0(\tau) + c_1(\tau)\, y.
  \]
  The estimate \eqref{eq-error-small} implies that the exponential growth rates of the various components $v_-$, $c_2$, $c_1$, $c_0$ are close to the growth rates predicted by the linearization,  i.e.~if we write $V_-(\tau) = \|v_-(\cdot, \tau)\|_\hilb$, then we have
  \begin{subequations}
    \begin{align}
      &V_-'(\tau)
      \leq -\tfrac 12 V_-(\tau)
      + \epsilon \|v(\cdot, \tau)\| \label{eq-component-Vmin}\\
      &|c_2'(\tau)| \leq \epsilon \|v(\cdot, \tau)\| \\
      &|c_1'(\tau) - \tfrac 12 c_1(\tau)| \leq \epsilon \|v(\cdot, \tau)\| \\
      &|c_0'(\tau) - c_0(\tau)| \leq \epsilon \|v(\cdot, \tau)\| 
    \end{align}
  \end{subequations}
  The total norm, which appears on the right in each of these inequalities, is given by Pythagoras
  \[
  \|v(\cdot, \tau)\|_\hilb^2 = V_-(\tau)^2 + c_0(\tau)^2 + c_1(\tau)^2 + c_2(\tau)^2.
  \]
  Using the ODE Lemma (see Lemma  in \cite{ADS}) we conclude that for $\tau\to-\infty$ exactly one of the four quantities $V_-(\tau)$, $c_0(\tau)$, $c_1(\tau)$, and $c_2(\tau)$ is much larger than the others.
  Similarly to  \cite{ADS}, we will now argue that $c_2(\tau)$ is in fact the largest term:
  \begin{lemma}
    For $\tau\to -\infty$ we have 
    \[
    V_-(\tau) + |c_0(\tau)| + |c_1(\tau)|= o\bigl(|c_2(\tau)|\bigr).
    \]  
  \end{lemma}
  \begin{proof}
    We must rule out that any of the three components $V_-$, $c_0$, or $c_1$ dominates for $\tau\ll 0$.
    
    The simplest is $V_-$, for if $\|v(\tau)\|_\hilb = \cO(V_-(\tau))$, then \eqref{eq-component-Vmin} implies that $V_-(\tau)$ is exponentially decaying.
    Since $v(\cdot, \tau)\to0$ as $\tau\to -\infty$, it would follow that $V_-(\tau)\equiv0$, and thus $v(\cdot, \tau)\equiv 0$, which is impossible.
    
    If $\|v(\cdot, \tau)\|_\hilb = o\bigl(c_0(\tau)\bigr)$ then on any bounded interval $|y|\leq L$ we have 
    \[
    v(y,\tau) = c_0(\tau) \bigl( 1 + o(1)\bigr) \qquad (\tau\to -\infty).
    \]
    In this case we derive contradiction using the same arguments as in \cite{ADS}.
    
    Finally, if $c_1(\tau)$ were the largest component, then we would have
    \[
    v(y,\tau) = c_1(\tau) \bigl( y + o(1)\bigr) \qquad (\tau\to -\infty)
    \]
    so that we would have either $v(L,\tau)>0$, or $v(-L,\tau)>0$ for all $\tau\ll0$.
    This again contradicts Lemma~\ref{lem-must-be-inside}.

  \end{proof}

  Once we have the result in Lemma \ref{lem-must-be-inside},  it follows as in \cite{ADS} that
  \[
  u(y,\tau) = \sqrt{2(n-1)}\, \left(1 - \frac{y^2 - 2}{4|\tau|}\right) + o(|\tau|^{-1}) \qquad |y| \le M 
  \]
  as $\tau \to -\infty$.
  This implies that 
  $y(\tau)$, the maximum point of $u(y,\tau)$ (such that $u_y(y(\tau),\tau) = 0$) satisfies 
  \[
  |y(\tau)| = o(1), \qquad \mbox{as}\,\,  \tau \to -\infty.
  \]
  In particular we have that $y(\tau) \leq 1$ for $\tau \leq \tau_0 \ll -1$.
  After we conclude this, the arguments in the intermediate and the tip region asymptotics in \cite{ADS} go through in our current  case where we lack the reflection symmetry.

\end{proof}

\end{document}